\documentclass[a4paper,11pt]{amsart}

\usepackage{amssymb,amsmath,amsthm,mathrsfs,enumerate,graphicx, color}
\usepackage[pdfpagelabels,colorlinks,linkcolor=blue,citecolor=black,urlcolor=blue]{hyperref}
\usepackage{esint}
\usepackage{tikz}
\usetikzlibrary{arrows}
\usepackage{citeref}

\newtheorem{thm}{Theorem}[section]

\newtheorem{cor}[thm]{Corollary}
\newtheorem{lem}[thm]{Lemma}
\newtheorem{prop}[thm]{Proposition}
\newtheorem{defn}[thm]{Definition}
\newtheorem{rem}[thm]{Remark}

\newcommand{\ip}[1]{\langle #1 \rangle} 

\newcommand{\lesi}{\lesssim}

\newcommand{\supp}{\operatorname{supp}}
\newcommand{\f}{\frac}

\newcommand{\om}{\omega}

\newcommand{\vc}{\infty}
\newcommand{\Rn}{\mathbb{R}^n}

\title[Harmonic analysis associated with Laguerre Expansions]{Hardy spaces and Campanato spaces associated with  Laguerre expansions and higher order Riesz transforms }         

\author[T. A. Bui]{The Anh Bui}
\address{School of Mathematical and Physical Sciences, Macquarie University, NSW 2109,
	Australia}
\email{the.bui@mq.edu.au}

\author[X. T. Duong]{Xuan Thinh Duong}
\address{School of Mathematical and Physical Sciences, Macquarie University, NSW 2109,
	Australia}
\email{xuan.duong@mq.edu.au}


\begin{document}

\begin{abstract}
Let \(\mathcal{L}_\nu\) be the Laguerre differential operator which is the self-adjoint extension of the differential operator  
\[
L_\nu := \sum_{i=1}^n \left[-\frac{\partial^2}{\partial x_i^2} + x_i^2 + \frac{1}{x_i^2} \left(\nu_i^2 - \frac{1}{4} \right) \right]
\]  
initially defined on \(C_c^\infty(\mathbb{R}_+^n)\) as its natural domain, where  \(\nu \in [-1/2,\infty)^n\), \(n \geq 1\). In this paper, we first develop the theory of Hardy spaces \(H^p_{\mathcal{L}_\nu}\) associated with \(\mathcal{L}_\nu\) for the full range \(p \in (0,1]\). Then we investigate the corresponding BMO-type spaces and establish that they coincide with the dual spaces of \(H^p_{\mathcal{L}_\nu}\). Finally, we show boundedness of higher-order Riesz transforms on Lebesgue spaces, as well as on our new Hardy and BMO-type spaces.

\end{abstract}
\date{}

\maketitle

\tableofcontents

\section{Introduction}\label{sec: intro}
For each $\nu =(\nu_1,\ldots, \nu_n)\in (-1,\vc)^n$, we consider the Laguerre differential operator 
\[
\begin{aligned}
	L_{\nu}
	&=\sum_{i=1}^n\Big[-\f{\partial^2 }{\partial x_i^2} + x_i^2+\frac{1}{x_i^2}\Big(\nu_i^2 - \frac{1}{4}\Big)\Big], \quad \quad x \in (0,\vc)^n.
\end{aligned}
\]
The $j$-th partial derivative associated with $L_{\nu}$ is given by
\[
\delta_{\nu_j} = \frac{\partial}{\partial x_j} + x_j-\frac{1}{x_j}\Big(\nu_j + \f{1}{2}\Big).
\]
Then the  adjoint of $\delta_{\nu_j}$ in $L^2(\mathbb{R}^n_+)$ is
\[
\delta_{\nu_j}^* = -\frac{\partial}{\partial x_j} + x_j-\frac{1}{x_j}\Big(\nu_j + \f{1}{2}\Big).
\]
It is straightforward that
\[
\sum_{i=1}^{n} \delta_{\nu_i}^* \delta_{\nu_i} = L_{\nu} {-2}(|\nu| + n).
\]
Let $k = (k_1, \ldots, k_n) \in \mathbb{N}^n$, $\mathbb N = \{0, 1, \ldots\}$, and $\nu = (\nu_1, \ldots, \nu_n) \in (-1, \infty)^n$ be multi-indices. The Laguerre function $\varphi_k^{\nu}$ on $\mathbb{R}^n_+$ is defined as
\[
\varphi_k^{\nu}(x) = \varphi^{\nu_1}_{k_1}(x_1) \ldots \varphi^{\nu_n}_{k_n}(x_n), \quad x = (x_1, \ldots, x_n) \in \mathbb{R}^n_+,
\]
where $\varphi^{\nu_i}_{k_i}$ are the one-dimensional Laguerre functions
\[
\varphi^{\nu_i}_{k_i}(x_i) = \Big(\f{2\Gamma(k_i+1)}{\Gamma(k_i+\nu_i+1)}\Big)^{1/2}L_{\nu_i}^{k_i}(x_i^2)x_i^{\nu_i + 1/2}e^{-x_i^2/2}, \quad x_i > 0, \quad i = 1, \ldots, n.
\]
Note that for $\nu_i > -1$ and $k_i \in \mathbb{N}$, $L_{\nu_i}^{k_i}$ denotes the Laguerre polynomial of degree $k_i$ and order $\nu_i$ outlined in  \cite[p.76]{L}.

Then it is well-known that the system $\{\varphi_k^{\nu} : k \in \mathbb{N}^n\}$ is an orthonormal basis of $L^2(\mathbb{R}^n_+, dx)$. Moreover, each $\varphi_k^{\nu}$ is an eigenfunction of  $L_{\nu}$ corresponding to the eigenvalue of $4|k| + 2|\nu| + 2n$, i.e.,
\begin{equation}\label{eq-eignevalue eigenvector}
L_{\nu}\varphi_k^{\nu} = (4|k| + 2|\nu| + 2n) \varphi_k^{\nu},
\end{equation}
where $|\nu| = \nu_1 + \ldots + \nu_n$ and  $|k| = k_1 + \ldots + k_n$. The operator $L_{\nu}$ is positive and symmetric in $L^2(\mathbb{R}^n_+, dx)$.

If $\nu = \left(-\frac{1}{2}, \ldots, -\frac{1}{2}\right)$, then $L_{\nu}$ becomes the harmonic oscillator $-\Delta + |x|^2$ on $\mathbb R^n_+$. Note that the Riesz transforms associated to the harmonic oscillator $-\Delta + |x|^2$ on $\mathbb R^n$ were investigated in \cite{ST, Th}.

The operator
\[
\mathcal{L}_{\nu}f = \sum_{k\in \mathbb{N}^{n}} (4|k| + 2|\nu| + 2n) \langle f, \varphi_{k}^\nu\rangle \varphi_{k}^\nu 
\]
defined on the domain $\text{Dom}\, \mathcal{L}_{\nu} = \{f \in L^2(\mathbb{R}^n_+): \sum_{k\in \mathbb{N}^{d}} (4|k| + 2|\nu| + 2n) |\langle f, \varphi_{k}^\nu\rangle|^2 < \infty\}$ is a self-adjoint extension of $L_{\nu}$ (the inclusion $C^{\infty}_c(\mathbb{R}^{d}_{+}) \subset \text{Dom}\, \mathcal{L}_{\nu}$ may be easily verified), has the discrete spectrum $\{4\ell + 2|\nu| + 2n: \ell \in \mathbb{N}\}$, and admits the spectral decomposition
\[
\mathcal{L}_{\nu}f = \sum_{\ell=0}^{\infty} (4\ell + 2|\nu| + 2n) P_{\nu,\ell}f,
\]
where the spectral projections are
\[
P_{\nu,\ell}f = \sum_{|k|=\ell} \langle f, \varphi_{k}^\nu\rangle \varphi_{k}^\nu.
\]

Moreover,
\begin{equation}\label{eq- delta and eigenvector}
	\delta_{\nu_j} \varphi_k^\nu =-2\sqrt{k_j} \varphi_{k-e_j}^{\nu+e_j}, \ \ \delta_{\nu_j}^* \varphi_k^\nu =-2\sqrt{k_j+1} \varphi_{k+e_j}^{\nu-e_j},
\end{equation}
where $\{e_1,\ldots, e_n\}$ is the standard basis for $\mathbb R^n$. Here and later on we use the convention that $\varphi_{k-e_j}^{\nu+e_j}=0$ if $k_j-1<0$ and $\varphi_{k+e_j}^{\nu-e_j}=0$ if $\nu_j-1<0$. See for example \cite{NS}.


\medskip

The first goal of the paper is to develop a new theory on Hardy spaces and BMO-type spaces associated with the Laguerre operator $\mathcal{L}_\nu$. The study of function spaces associated with operators is an active area of research in harmonic analysis, attracting significant interest. See, for example, \cite{DY, HLMMY, Yan, CFYY, Dziu, JY, SY, BDK, YZ} and references therein. A key problem in this direction is to characterize Hardy spaces via new atomic decompositions as in \cite{Dziu, YZ, BDK}, rather than using abstract atom definitions based on the underlying operator as in \cite{DY, HLMMY, Yan}.  

Let us provide an overview of this research direction. Let $L=-\Delta+V$ be a Schr\"odinger operator on $\Rn$ with $n\ge 3$, where $V$ is a nonnegative potential satisfying a suitable reverse H\"older inequality. Define
\begin{equation}\label{eq1}
\rho(x) =\sup\Big\{r>0: \f{1}{r^{n-2}}\int_B V(x)dx\le 1\Big\}.
\end{equation}
Let $p\in (0,1]$. A function $a$ is called a  $(p,\rho)$-atom associated to the ball $B(x_0,r)$ if
\begin{enumerate}[{\rm (i)}]
	\item ${\rm supp}\, a\subset B(x_0,r)$;
	\item $\|a\|_{L^\vc}\leq |B(x_0,r)|^{-1/p}$;
	\item $\displaystyle \int a(x)x^\alpha dx =0$ for all multi-indices $\alpha$ with $|\alpha| \le n(1/p-1)$ if $r< \rho(x_0)$.
\end{enumerate}
Then we define the Hardy space $H^{p}_{\rho}(\mathbb{R}^n)$ as the completion of all $f\in L^2(\Rn)$ functions under the norm
\[
\|f\|_{H^{p}_{\rho}(\mathbb{R}^n)}=\inf\Big\{\Big(\sum_j|\lambda_j|^p\Big)^{1/p}:
f=\sum_j \lambda_ja_j \ \text{in $L^2(\Rn)$}, \text{where $a_j$ are $(p,\rho)$ atoms} \Big\}.
\]
We also have the Hardy $H^p_L(\mathbb R^n)$ defined as the completion of all $f\in L^2(\Rn)$ functions under the norm
\[
\|f\|_{H^p_L(\mathbb R^n)}=\|\sup_{t>0}|e^{-tL}f|\|_p.
\]
Under the reverse H\"older condition
\begin{equation}\label{eq- reverse holder}
	\Big(\frac{1}{|B|}\int_BV(x)^{n/2}dx\Big)^{2/n}\lesssim \frac{1}{|B|}\int_BV(x) dx,
\end{equation}
for all balls $B\subset \mathbb{R}^n$, it was shown in \cite{Dziu} that $H^{1}_{\rho}(\mathbb{R}^n)$ and $H^1_L(\mathbb{R}^n)$ coincide. In particular, for polynomial potentials $V$, these spaces coincide for all $p\in (0,1]$ (see \cite{Dziu2}).  Later, Yang and Zhou \cite{YZ} extended these results to Hardy spaces on metric space $(X,d)$ equipped with a doubling measure in the sense of Coifmann and Weiss \cite{CW} by introducing a critical function $\rho$ satisfying
\begin{equation}\label{criticalfunction}
	\rho(y)\lesi \rho(x)\left(1 +\frac{d(x,y)}{\rho(x)}\right)^{\frac{k_0}{k_0+1}}
\end{equation}
for some $k_0>0$. They developed atomic Hardy spaces $H^1_\rho(X)$ and provided maximal function characterizations, leading to applications in settings such as Lie groups, doubling manifolds, and Heisenberg groups.

In the Laguerre setting, similar results were shown in \cite{DZ1} for $n=1$, $\nu>-1/2$, and $p=1$, later extended to $n\geq1$, $\nu\in [-1/2,\infty)^n$, and $p\in (\frac{n}{n+1},1]$ in \cite{B}. This paper advances the theory further by covering the full range $p\in(0,1]$ (Theorems \ref{mainthm2s} and \ref{dual}). In addition, we also investigate the BMO function type spaces, which will be proved as 
the dual spaces of the Hardy spaces. The extension presents challenges since:
\begin{itemize}
	\item The potential $V$ in $\mathcal{L}_\nu$ is neither a polynomial nor satisfies \eqref{eq- reverse holder}.
	\item The critical function \eqref{eq- critical function} does not satisfy \eqref{criticalfunction} as in \cite{YZ}.
	\item The heat kernel of $\mathcal{L}_\nu$ lacks certain conditions, including the H\"older continuity (see \cite{BDK}).
	\item Extending results from $p\in(\frac{n}{n+1},1]$ to $p\in(0,1]$ requires new techniques (see Section 3). Note that, until now, such results for the full range $p \in (0,1]$ were only known in the setting of the Schr\"odinger operator on $\mathbb{R}^n$ and on the Heisenberg group with a \textit{nonnegative polynomial} potential; see \cite{Dziu2, BHH}. In our setting, the potential is not a polynomial, and $\mathbb{R}^n_+$ can be viewed as a domain of $\mathbb{R}^n$, which introduces additional difficulties. 
\end{itemize}

The second goal of this paper is to investigate the boundedness of the higher-order Riesz transform  
\[
\delta_\nu^k \mathcal{L}_\nu^{-|k|/2}, \quad k\in \mathbb{N}^n,
\]
where \(\delta_\nu^k = \delta_{\nu_n}^{k_n} \ldots \delta_{\nu_1}^{k_1}\). The study of Riesz transforms in the context of orthogonal expansions was initiated by Muckenhoupt and Stein \cite{MS}, and since then, significant contributions have been made, particularly concerning Riesz transforms associated with Laguerre and Hermite operators. See, for example, \cite{Betancor, Betancor2, Muc, NS, NS2, ST, Th} and the references therein.  

In the one-dimensional case (\(n=1\)), higher-order Riesz transforms were studied in \cite{Betancor} for \(\nu > -1\). However, the techniques in \cite{Betancor} do not extend naturally to higher dimensions. In the case \(n\geq 2\), only the first-order Riesz transforms  
\[
\delta_{\nu_j} \mathcal{L}_\nu^{-1/2}, \quad j=1,\ldots, n,
\]
(not the higher-order transforms) have been analyzed. Specifically, it was shown in \cite{NS} that these transforms are Calder\'on-Zygmund operators provided that \(\nu_j \in \{-1/2\} \cup (1/2, \infty)\), and this result was later extended to the full range \(\nu_j \in [-1/2, \infty)\) in \cite{B}.  

To the best of our knowledge, the boundedness of higher-order Riesz transforms in dimensions \(n \geq 2\) remains an open problem due to several challenges. First, while the \(L^2\)-boundedness of the first-order Riesz transforms \(\delta_{\nu_j} \mathcal{L}_\nu^{-1/2}\) follows directly from spectral theory, proving the \(L^2\)-boundedness of the higher-order Riesz transforms \(\delta_\nu^k \mathcal{L}_\nu^{-|k|/2}\) for \(|k| \geq 2\) is significantly more involved. Second, obtaining kernel estimates for these operators requires precise control over the higher-order derivatives of the heat kernel associated with \(\mathcal{L}_\nu\), which have been absent in the literature.  

Motivated by these gaps, we aim to establish that the higher-order Riesz transforms \(\delta_\nu^k \mathcal{L}_\nu^{-|k|/2}\) are Calder\'on-Zygmund operators for \(\nu \in [-1/2, \infty)^n\) (see Theorem \ref{thm-Riesz transform}). This range is sharp, as in general, if \(\nu \in (-1,\infty)^n\), the Riesz transforms may fail to be Calder\'on-Zygmund operators, even in the one-dimensional case (see \cite{Betancor}). Furthermore, we show that these operators are bounded on our new Hardy space and the corresponding BMO-type space, completing the characterization of their boundedness (see Theorem \ref{thm- boundedness on Hardy and BMO}). Our results not only extend those in \cite{NS, B, Betancor2} but also provide a comprehensive description of the boundedness properties of higher-order 
Riesz transforms associated to the Lageurre operators.

\medskip

\bigskip

The paper is organized as follows. In Section 2, we recall some  properties of the critical function $\rho_\nu$, revisit the theory of the Hardy space associated with operators, including its maximal characterization and the theory of tent spaces. Section 3 is dedicated to proving kernel estimates related to the heat semigroup $e^{-t\mathcal L_\nu}$. In particular, we prove estimates related to the higher-order derivatives of the heat kernel. Section 4 is devoted to the theory of Hardy spaces and BMO type spaces associated to $\mathcal L_\nu$. Finally, Section 5 provides the
proofs of the boundedness of the Riesz transforms on Hardy spaces and Campanato spaces associated with $\mathcal L_\nu$ ( Theorem \ref{thm- boundedness on Hardy and BMO}).

\bigskip
 
 Throughout the paper, we always use $C$ and $c$ to denote positive constants that are independent of the main parameters involved but whose values may differ from line to line. We will write $A\lesi B$ if there is a universal constant $C$ so that $A\leq CB$ and $A\sim B$ if $A\lesi B$ and $B\lesi A$. For $a \in \mathbb{R}$, we denote the integer part of $a$ by $\lfloor a\rfloor$.  For a given ball $B$, unless specified otherwise, we shall use $x_B$ to denote the center and $r_B$ for the radius of the ball.
 
In the whole paper, we will often use the following inequality without any explanation $e^{-x}\le c(\alpha) x^{-\alpha}$ for any $\alpha>0$ and $x>0$.
\section{Preliminaries}
In this section, we will establish some preliminary results regarding the critical functions and the theory of Hardy spaces associated to an abstract operator.
\subsection{Critical functions}

 For $\nu =(\nu_1,\ldots,\nu_n) \in [-1/2,\vc)^n$ and $x=(x_1,\ldots, x_n)\in (0,\vc)^n$, define
\begin{equation}\label{eq- critical function}
	\rho_\nu(x) = \f{1}{16}\min \Big\{\f{1}{|x|}, 1,  x_j: j\in \mathcal J_\nu  \Big\},
\end{equation}
where $\mathcal J_\nu:=\{j: \nu_j>-1/2\}$.

It is easy to see that 
\begin{enumerate}[{\rm (i)}]
	\item If $\mathcal J_\nu=\emptyset$, then we have
	\[
	\rho_\nu(x) \sim  \min \Big\{\f{1}{|x|}, 1\Big\}\sim \f{1}{1+|x|}, \ \ x\in \mathbb R_+^n.
	\]
	
	\item If $\mathcal J_\nu\ne \emptyset$, then
	\begin{equation*} 
		\rho_\nu(x) = \f{1}{16}\min \Big\{\f{1}{|x|}, x_j: j\in \mathcal J_\nu  \Big\}, \ \ x\in \mathbb R_+^n.
	\end{equation*}
	\item If $\mathcal J_\nu=\{1,\ldots,n\}$, then
	\begin{equation*}
		\rho_\nu(x) = \f{1}{16}\min \Big\{\f{1}{|x|}, x_1,\ldots,x_n  \Big\}, \ \ x\in \mathbb R_+^n. 
	\end{equation*}
	In this case, the critical function $\rho_\nu$ is independent of $\nu$. More precisely, if $\nu,\nu'\in (-1/2,\vc)^n$, then 
	\[
	\rho_\nu(x) =\rho_{\nu'}(x)= \f{1}{16}\min \Big\{\f{1}{|x|}, x_1,\ldots,x_n  \Big\}, \ \ x\in \mathbb R_+^n.
	\]
\end{enumerate}

For each $j=1,\ldots, n$, define
\[
\rho_{\nu_j}(x_j) =\begin{cases} \displaystyle \f{1}{16}\times \min\{1,x_j^{-1}\}, \ \ \ &\nu_j=-1/2,\\
\displaystyle\f{1}{16}\min\{x_j,x_j^{-1}\}, \ \ \ &\nu_j>-1/2.	
\end{cases}
\]
Then we have 
\begin{equation}\label{eq- equivalence of rho}
\rho_\nu(x) \sim \min \Big\{\rho_{\nu_1}(x_1),\ldots, \rho_{\nu_n}(x_n)\Big\}.
\end{equation}

We recall a couple of important properties regarding the critical function in \cite[Lemma 2.1 \& Corollary 2.2]{B}.
\begin{lem}[\cite{B}]\label{lem-critical function}
	Let  $\nu\in [-1/2,\vc)^n$ and $x\in \mathbb R^n_+$. Then for every $y\in 4B(x,\rho_\nu(x))$,
	\[
	\f{1}{2}\rho_\nu(x)\le\rho_\nu(y)\le 2\rho_\nu(x).
	\]
\end{lem}

\begin{cor}[\cite{B}]\label{cor1}
	Let $\nu\in[-1/2,\vc)^n$. There exist a family of balls $\{B(x_\xi,\rho_\nu(x_\xi)): \xi\in \mathcal I\}$ and a family of functions $\{\psi_\xi: \xi \in \mathcal I\}$ such that 
	\begin{enumerate}[{\rm (i)}]
		\item $\displaystyle \bigcup_{\xi\in \mathcal I}B(x_\xi,\rho_\nu(x_\xi)) = \mathbb R^n_+$;
		\item $\{B(x_\xi,\rho_\nu(x_\xi)/5): \xi\in \mathcal I\}$ is pairwise disjoint;
		\item  $\displaystyle \sum_{\xi\in \mathcal I} \chi_{B(x_\xi,\rho_\nu(x_\xi))}\lesi 1$;
		\item $\supp \psi_\xi\subset B(x_\xi,\rho_\nu(x_\xi))$ and $0\le \psi_\xi \le 1$ for each $\xi \in \mathcal I$;
		\item $\displaystyle \sum_{\xi\in \mathcal I} \psi_\xi =1$.
	\end{enumerate}
\end{cor}	

As the following estimate is standard, we omit the proof.

\begin{lem}\label{lem- product two kernels}
	Let $a>0$. Then there exist $C, b>0$ such that 
	\[
	\int_{\Rn_+}  \f{1}{t^{n/2}}\exp\Big(-\f{|x-z|^2}{at}\Big)\f{1}{s^{n/2}}\exp\Big(-\f{|y-z|^2}{as}\Big) dz\le C\f{1}{(t+s)^{n/2}}\exp\Big(-\f{|x-z|^2}{b(t+s)}\Big) 
	\]
	for all $s,t>0$ and $x,y\in \Rn_+$.
\end{lem}

\subsection{Hardy spaces associated to operators}
On our setting $(\mathbb R^n_+, |\cdot |, dx)$, we consider an operator $L$ satisfying the following two conditions:
\begin{enumerate}
	\item[(A1)] $L$ is  a nonnegative, {possibly unbounded}, self-adjoint operator on $L^2(\mathbb{R}^n_+)$;
	\item[(A2)] $L$ generates a semigroup $\{e^{-tL}\}_{t>0}$ whose kernel $p_t(x,y)$ admits a Gaussian upper bound. That is, there exist two positive finite constants $C$ and  $c$ so that for all $x,y\in \mathbb{R}^n_+$ and $t>0$,
	\begin{equation*}
		\displaystyle |p_t(x,y)|\leq \f{C}{t^{n/2}}\exp\Big(-\f{|x-y|^2}{ct}\Big).
	\end{equation*}
\end{enumerate}

We now recall the definition of the Hardy spaces associated to $L$ in \cite{HLMMY, JY}.

Let $0<p\le 1$. Then the Hardy space $H_{S_L}^p(\mathbb R^n_+)$  is
defined as the completion of
$$
\{f\in L^2(\Rn_+) : S_Lf \in L^p(\mathbb R^n_+)\}
$$
under the norm $\|f\|_{H_{S_L}^p(\mathbb R^n_+)}=\|S_Lf\|_{L^p(\mathbb R^n_+)}$, where the square function $S_{L}$ is defined as
$$
S_Lf(x)=\Bigl(\int_0^\vc \int_{|x-y|<t}|t^2L e^{-t^2L}f(y)|^2\f{ dy dt}{t^{n+1}}\Bigr)^{1/2}.
$$

\begin{defn}[Atoms associated to $L$]\label{def: L-atom}
	Let $p\in (0,1]$ and $M\in \mathbb{N}$. A function $a$ supported in a ball $B$ is called a  $(p,M)_L$-atom if there exists a
	function $b$ in the domain of $L^M$ such that
	\begin{enumerate}[{\rm (i)}]
		\item  $a=L^M b$;
		\item $\supp L ^{k}b\subset B, \ k=0, 1, \dots, M$;
		\item $\|L^{k}b\|_{L^\vc(\mathbb{R}^n_+)}\leq
		r_B^{2(M-k)}|B|^{-\f{1}{p}},\ k=0,1,\dots,M$.
	\end{enumerate}
\end{defn}

\noindent Then the atomic Hardy spaces associated to the operator $L$ are defined as follows:
\begin{defn}[Atomic Hardy spaces associated to $L$]
	
	Given $p\in (0,1]$ and $M\in \mathbb{N}$, we  say that $f=\sum
	\lambda_ja_j$ is an atomic $(p,M)_L$-representation if
	$\{\lambda_j\}_{j=0}^\infty\in l^p$, each $a_j$ is a $(p,M)_L$-atom,
	and the sum converges in $L^2(\mathbb{R}^n_+)$. The space $H^{p}_{L,at,M}(\mathbb{R}^n_+)$ is then defined as the completion of
	\[
	\left\{f\in L^2(\mathbb{R}^n_+):f \ \text{has an atomic
		$(p,M)_L$-representation}\right\},
	\]
	with the norm given by
	$$
	\|f\|_{H^{p}_{L,at,M}(\mathbb{R}^n_+)}=\inf\left\{\left(\sum|\lambda_j|^p\right)^{1/p}:
	f=\sum \lambda_ja_j \ \text{is an atomic $(p,M)_L$-representation}\right\}.
	$$
\end{defn}

\noindent The maximal Hardy spaces associated to $L$ are defined as follows.

\begin{defn}[Maximal Hardy spaces associated to $L$]\label{defn-maximal Hardy spaces}
	
	For $f\in L^2(\mathbb{R}^n_+)$, we define  the   maximal function  by
	\[
	\mathcal M_{L}(x)=\sup_{t>0}|e^{-tL}f(x)|.
	\]
	Given $p\in (0,1]$,  the Hardy space $H^{p}_{L}(\mathbb{R}^n_+)$ is defined as the completion of
	$$
	\left\{f
	\in L^2(\mathbb{R}^n_+): \mathcal M_{L} \in L^p(\mathbb{R}^n_+)\right\},
	$$
	with the norm given by
	$$
	\|f\|_{H^{p}_{L}(\mathbb{R}^n_+)}=\|\mathcal M_{L}\|_{L^p(\mathbb{R}^n_+)}.
	$$
\end{defn}
The following theorem is taken from \cite[Theorem 1.3]{SY}.
\begin{thm}\label{mainthm1-Hardy}
	Let $p\in (0,1]$, and $M>\f{n}{2}\big(\f{1}{p}-1\big)$. Then the Hardy spaces $H_{S_L}(\mathbb R^n_+)$, $H^{p}_{L,at,M}(\mathbb{R}^n_+)$ and $H^{p}_{L}(\mathbb{R}^n_+)$  coincide with equivalent norms.
\end{thm}

	\subsection{Tent spaces}
The tent spaces, which were first introduced in \cite{CMS}, become an effective tool in the study of function spaces in harmonic analysis including Hardy spaces. In this section we will recall the definitions of the tent spaces and their fundamental properties in \cite{CMS}. We begin with some notations in \cite{CMS}.
\begin{itemize}
	\item For $x\in \Rn_+$, we denote $$\Gamma  (x):=\big\{(t,y)\in \Rn_+ \times \big(0,\infty\big): |x-y| \leq t\big\}.$$
	
	\item For any closed subset $F \subset \Rn_+$ , define 
	$$ R (F) = \cup_{x\in F} \Gamma(x).$$ 
	
	\item If $O$ is an open set in $\Rn_+$, then the tent over $\widehat{O}$ is defined as 
	$$\widehat{O}=\big(R(O^c)\big)^c.$$
	
\end{itemize}
For a measurable function $F$ defined in ${\mathbb{R}^n_+\times (0,\vc)}$, we define
$$\mathcal A(F)(x)=\Big(\int_{\Gamma(x)}|F(y,t)|^2\frac{dydt}{t^{n+1}}\Big)^\frac{1}{2},$$

$$\mathcal C(F)(x)=\underset{x\in B}{\sup}\Big(\frac{1}{|B|}\int_{\widehat{B}}|F(y,t)|^2\frac{dydt}{t}\Big)^\frac{1}{2},$$
and
$$\mathcal C_p(F)(x)=\underset{x\in B}{\sup}\Big(\frac{1}{|B|^{\frac{1}{p}-\frac{1}{2}}}\int_{\widehat{B}}|F(y,t)|^2\frac{dydt}{t}\Big)^\frac{1}{2} \ \ \ \text{ for $0 < p \leq 1$}.$$

The following definition is taken from \cite{CMS}.		
\begin{defn}[The tent spaces] The tent spaces  are defined as follows.
	\begin{itemize}
		\item For $0<p < \infty$, we define $T_2^p:= \{F: \ \mathcal A(F)\in L^p(\Rn) \}$ with the norm $\|F\|_{T^p_2}= \|\mathcal A(F)\|_{p}$ which is a Banach space for $1\le p<\vc$.
		\item  For $p = \infty$, we define $T_2^\infty=\big\{F: \  \mathcal C(F) \in L^\vc(\Rn) \big\}$ with the norm $\|F\|_{T_2^\vc(\Rn)}=\|\mathcal C(F)\|_{\vc}$ which  is a Banach space.
		\item For $0<p \leq 1$, we define $T_2^{p,\infty}=\big\{F: \|\mathcal C_p(F)\|_p<\infty\big\}$ with the norm $\|F\|_{T_2^{p,\infty}}=\|\mathcal C_pF\|_\vc$. Obviously, $T^{1,\infty}_2=T_2^\infty$.
	\end{itemize}
\end{defn}

	

We also recall duality results of the tent spaces.
\begin{prop}[\cite{CMS}]\label{05}
	\begin{enumerate}[\rm (i)]\label{T1 TP}
		\item The following inequality holds, whenever $f\in T_2^1$ and $g \in T_2^\infty$:
		$$\int_{\mathbb{R}^n_+\times (0,\vc)}|f(x,t)g(x,t)|\frac{dxdt}{t}\leq c\int_{\Rn_+}\mathcal A(f)(x)\mathcal C(g)(x)dx.$$
		\item Suppose $1<p< \infty$, $f \in T^p_2$ and $ g\in T_2^{p'}$ with $\frac{1}{p}+\frac{1}{p'}=1$ then 
		$$\int_{\mathbb{R}^n_+\times (0,\vc)}|f(y,t)g(y,t)|\frac{dydt}{t}\leq \int_{\Rn_+}\mathcal A(f)(x)\mathcal A(g)(x)dx.$$
		\item If $0 <p \leq 1$, then the dual space of $T_2^p$ is $T_2^{p,\infty}$. More precisely, the pairing 
		$$\left\langle f,g \right\rangle =\int_{\mathbb{R}^n_+\times (0,\vc)}f(x,t)g(x,t)\frac{dxdt}{t}$$ 
		realizes $T_2^{p,\infty}$ 	as   the dual of $T_2^p$. 
	\end{enumerate}
\end{prop}
\section{Some   kernel estimates}

This section is devoted to establishing some kernel estimates related to the heat kernel of $\mathcal L_\nu$. These estimates play an essential role in proving our main results. We begin by providing an explicit formula for the heat kernel of $\mathcal L_\nu$.

Let $\nu \in [-1/2,\vc)^n$. For each $j=1,\ldots, n$, similarly to $\mathcal L_\nu$, denote by $\mathcal L_{\nu_j}$ the self-adjoint extension of the differential operator 
\[
 L_{\nu_j} :=  -\frac{\partial^2}{\partial x_j^2} + x_j^2 + \frac{1}{x_j^2}(\nu_j^2 - \frac{1}{4})
\] 
on $C_c^\infty(\mathbb{R}_+)$ as the natural domain. It is easy to see that 
\[
\mathcal L_\nu =\sum_{j=1}^n \mathcal L_{\nu_j}.
\]

Let $p_t^\nu(x,y)$ be the kernel of $e^{-t\mathcal L_\nu}$ and let $p_t^{\nu_j}(x_j,y_j)$ be the kernel of $e^{-t\mathcal L_{\nu_j}}$ for each $j=1,\ldots, n$. Then we have
\begin{equation}\label{eq- prod ptnu}
	p_t^\nu(x,y)=\prod_{j=1}^n p_t^{\nu_j}(x_j,y_j).
\end{equation}
For $\nu_j\ge -1/2$, $j=1,\ldots, n,$, the kernel of $e^{-t\mathcal L_{\nu_j}}$ is given by
\begin{equation}
	\label{eq1-ptxy}
	p_t^{\nu_j}(x_j,y_j)=\f{2(rx_jy_j)^{1/2}}{1-r}\exp\Big(-\f{1}{2}\f{1+r}{1-r}(x_j^2+y_j^2)\Big)I_{\nu_j}\Big(\f{2r^{1/2}}{1-r}x_jy_j\Big),
\end{equation}
where $r=e^{-4t}$ and $I_\alpha$ is the usual Bessel funtions of an imaginary argument defined by
\[
I_\alpha(z)=\sum_{k=0}^\vc \f{\Big(\f{z}{2}\Big)^{\alpha+2k}}{k! \Gamma(\alpha+k+1)}, \ \ \ \ \alpha >-1.
\]
See for example \cite{Dziu, NS}.

Note that for each $j=1,\ldots, n$, we can rewrite the kernel $p_t^{\nu_j}(x_j,y_j)$ as follows
\begin{equation}
	\label{eq2-ptxy}
	\begin{aligned}
		p_t^{\nu_j}(x_j,y_j)=\f{2(rx_jy_j)^{1/2}}{1-r}&\exp\Big(-\f{1}{2}\f{1+r}{1-r}|x_j-y_j|^2\Big)\exp\Big(-\f{1-r^{1/2}}{1+r^{1/2}}x_jy_j\Big)\\
		&\times \exp\Big(-\f{2r^{1/2}}{1-r}x_jy_j\Big)I_{\nu_j}\Big(\f{2r^{1/2}}{1-r}x_jy_j\Big),
	\end{aligned}
\end{equation}
where $r=e^{-4t}$.

The following  properties of the Bessel function $I_\alpha$  with $\alpha>-1$ are well-known and are taken from \cite{L}:
\begin{equation}
	\label{eq1-Inu}
	I_\alpha(z)\sim z^\alpha, \ \ \ 0<z\le 1,
\end{equation}
\begin{equation}
	\label{eq2-Inu}
	I_\alpha(z)= \f{e^z}{\sqrt{2\pi z}}+S_\alpha(z),
\end{equation}
where
\begin{equation}
	\label{eq3-Inu}
	|S_\alpha(z)|\le  Ce^zz^{-3/2}, \ \ z\ge 1,
\end{equation}
\begin{equation}
	\label{eq4-Inu}
	\f{d}{dz}(z^{-\alpha}I_\alpha(z))=z^{-\alpha}I_{\alpha+1}(z)
\end{equation}
and
\begin{equation}
	\label{eq6s-Inu}
	0< I_\alpha(z)-I_{\alpha+2}(z)=\f{2(\alpha+1)}{z}I_{\alpha+1}(z), \ \ \ z>0.
\end{equation}

Let $\nu > -1$, we have
\[
I_\alpha(z)-I_{\alpha+1}(z)=[I_\alpha(z)-I_{\alpha+2}(z)] -[I_{\alpha+1}(z)-I_{\alpha+2}(z)].
\]
Applying \eqref{eq6s-Inu}, 
\[
I_\alpha(z)-I_{\alpha+2}(z) =\f{2(\alpha+1)}{z}I_{\alpha+1}(z).
\]
Hence,
\[
|I_\alpha(z)-I_{\alpha+1}(z)|\le \f{2(\alpha+1)}{z}I_{\alpha+1}(z)+|I_{\alpha+1}(z)-I_{\alpha+2}(z)|.
\]
On the other hand, since $\alpha+1>-1/2$, we have
\[
0< I_{\alpha+1}(z)-I_{\alpha+2}(z)<2(\alpha+2)\f{I_{\alpha+2}(z)}{z}<2(\alpha+2)\f{I_{\alpha+1}(z)}{z}, \ \ \ z>0.
\]
See for example \cite{Na}.

Consequently,
\begin{equation}
	\label{eq5s-Inu}
	|I_\alpha(z)-I_{\alpha+1}(z)|<(4\alpha +6)\f{I_{\alpha+1}(z)}{z}, \ \ \ \alpha>-1, z>0.
\end{equation}

From \eqref{eq6s-Inu}, \eqref{eq5s-Inu} and \eqref{eq1-ptxy}, for $\alpha>-1$ we immediately obtain the following, for $t>0$ and $x,y \in (0,\vc)$, 
\begin{equation}
\label{eq6-Inu}
p_t^\alpha(x,y)-p_t^{\alpha+2}(x,y)= 2(\alpha+1)\f{1-r}{2r^{1/2}xy}p_t^{\alpha+1}(x,y) 
\end{equation}
and
\begin{equation}
	\label{eq5-Inu}
	|p_t^\alpha(x,y)-p_t^{\alpha+1}(x,y)|\lesi \f{1-r}{2r^{1/2}xy}p_t^{\alpha+1}(x,y), 
\end{equation}
where $r=e^{-4t}$.

\medskip

\subsection{The case $n=1$}	In this case,  we have $\rho_\nu(x) =\f{1}{16}\min\{1/x, x\}$ if $\nu>-1/2$ and $\rho_\nu(x) =\f{1}{16}\min\{1/x, 1\}$ if $\nu=-1/2$.

We will consider $\nu>-1/2$ and $\nu=-1/2$ separately.

\bigskip

\noindent \textbf{We first consider the case $\nu=-1/2$.}

We recall the following result in \cite[Proposition 3.2]{B}.
\begin{prop}[\cite{B}]
	\label{prop- derivative heat kernel nu = -1/2}
	Let $\nu=-1/2$. Then for any  $k, \ell \in \mathbb{N}$ and $N>0$, 
	\begin{equation}
		\label{eq- partial alpha x nu = -1/2}
		|x^\ell \partial^kp^{\nu}_t(x,y)|\lesi \f{e^{-t/4}}{t^{(\ell+k+1)/2}}\exp\Big(-\f{|x-y|^2}{ct}\Big) \Big[1+\f{\sqrt t}{\rho_\nu(x)}+\f{\sqrt t}{\rho_\nu(x)}\Big]^{-N}
	\end{equation}
	for all $t>0$ and $x,y>0$.
\end{prop}

The above estimate together with the expressions $\delta_\nu = \partial_x +x$ and $\delta_\nu^* = -\partial_x +x$ implies the following estimates.
\begin{prop}\label{prop-delta k p x y nu = -1/2}
	Let $\nu=-1/2$. For  $k\in \mathbb N$ and $N>0$, we have
	\begin{equation*}
		\begin{aligned}
			|\delta^k_\nu  p_t^\nu(x,y)|&\lesi  \f{1}{t^{(k+1)/2}}\exp\Big(-\f{|x-y|^2}{ct}\Big)\Big(1+\f{\sqrt t}{\rho_\nu(x)}+\f{\sqrt t}{\rho_\nu(y)}\Big)^{-N}
		\end{aligned}
	\end{equation*}
	for all $t>0$ and all $x, y\in (0,\vc)$.
\end{prop}

\begin{prop}\label{prop-gradient x y nu = -1/2}
	Let $\nu>-1/2$. For $k,j \in \mathbb N$, we have
	\begin{equation*}
		\begin{aligned}
			|\partial^k \delta_\nu^j p_t^\nu(x,y)|&\lesi   \f{1}{t^{(j+k+1)/2}}\exp\Big(-\f{|x-y|^2}{ct}\Big)\Big(1+\f{\sqrt t}{\rho_\nu(x)}+\f{\sqrt t}{\rho_\nu(y)}\Big)^{-N}
		\end{aligned}
	\end{equation*}
	for all $t>0$ and all $x, y\in (0,\vc)$.
\end{prop}

\begin{prop}\label{prop-gradient x y dual delta nu=-1/2}
	Let $\nu=-1/2$. Then for any $j, k\in \mathbb N$ and $N>0$,
	\[
	|\mathcal L_\nu^j (\delta_\nu^*)^kp^{\nu}_t(x,y)|\lesi \f{1}{t^{(k+2j+1)/2}}\exp\Big(-\f{|x-y|^2}{ct}\Big)\Big(1+\f{\sqrt t}{\rho_\nu(x)}+\f{\sqrt t}{\rho_\nu(y)}\Big)^{-N}
	\]
	for all $t>0$ and $x,y\in (0,\vc)$.
\end{prop}

\bigskip

\noindent \textbf{We now consider the case $\nu>-1/2$.} From now on, unless otherwise specified,  
we assume that $\nu>-1/2$. Before coming the the kernel estimates, we need the following simple identities, which can be verified by induction easily,for $k\in \mathbb N$,
\begin{equation}\label{eq- delta k xf}
	\delta^{k}_\nu \Big[xf(x)\Big]= k\delta^{k-1}_\nu  f(x) + x\delta^{k}_\nu  f(x), 
\end{equation}
and
\begin{equation}\label{eq- delta k}
	\delta_\nu^{k}=\Big(\delta_{\nu+1}+\f{1}{x}\Big)^{k} =\delta_{\nu+1}^k +\f{k}{x}\delta^{k-1}_{\nu+1}.
\end{equation}

In this case, note that $\displaystyle \rho_\nu(x)=\tfrac{1}{16}\min \{x,x^{-1}\}$ for $x\in (0,\vc)$. We have the following result (see \cite[Proposition 3.3]{B}).
\begin{prop}[\cite{B}]\label{prop-heat kernel}
	Let $\nu >-1/2$. Then we have
	\begin{equation}
	\label{eq-ptxy rho}
	\begin{aligned}
	p_t^\nu(x,y)\lesi \f{e^{-t/2}}{\sqrt t}\exp\Big(-\f{|x-y|^2}{ct}\Big)\Big(1+\f{\sqrt t}{\rho_\nu(x)} +\f{\sqrt t}{\rho_\nu(y)}\Big)^{-(\nu+1/2)}
	\end{aligned}
	\end{equation}
	for all $t>0$ and all $x,  y\in (0,\vc)$.
\end{prop}

\bigskip

Next we will estimate the derivatives associated with $\mathcal L_\nu$ of the heat kernel. To do this, note that from \eqref{eq1-ptxy} we can rewrite $p_t^\nu(x,y)$ as
$$
p_t^\nu(x,y)=\Big(\f{2\sqrt r}{1-r}\Big)^{1/2}\Big(\f{2r^{1/2}}{1-r}xy\Big)^{\nu+1/2}\exp\Big(-\f{1+r}{2(1-r)}(x^2+y^2)\Big)\Big(\f{2r^{1/2}}{1-r}xy\Big)^{-\nu}I_\nu\Big(\f{2r^{1/2}}{1-r}xy\Big),
$$
where $r=e^{-4t}$.

Setting 
$$
H_\nu(r;x,y)=\Big(\f{2r^{1/2}}{1-r}xy\Big)^{-\nu}I_\nu\Big(\f{2r^{1/2}}{1-r}xy\Big),
$$
then
\begin{equation}\label{eq-new formula of heat kernel}
p_t^\nu(x,y)=\Big(\f{2r^{1/2}}{1-r}\Big)^{1/2}\Big(\f{2r^{1/2}}{1-r}xy\Big)^{\nu+1/2}\exp\Big(-\f{1+r}{2(1-r)}(x^2+y^2)\Big)H_\nu(r;x,y). 
\end{equation}
From \eqref{eq-new formula of heat kernel}, applying the chain rule and \eqref{eq4-Inu},
\begin{equation}\label{eq- chain rule}
	\begin{aligned}
		\partial_x p_t^\nu(x,y)&=  \f{\nu+1/2}{x}p_t^\nu(x,y)- \f{1+r}{1-r}xp_t^\nu(x,y) + \f{2r^{1/2}}{1-r}y p_t^{\nu+1}(x,y),
	\end{aligned}
\end{equation}
which, in combination with the fact $\displaystyle \delta_\nu = \partial_x + x - \f{\nu+1/2}{x}$, implies that
\begin{equation}\label{eq- chain rule 2}
	\begin{aligned}
		\delta_\nu p_t^\nu(x,y)&=  xp_t^\nu(x,y)- \f{1+r}{1-r}xp_t^\nu(x,y) + \f{2r^{1/2}}{1-r}y p_t^{\nu+1}(x,y)\\
		&= -\f{2r}{1-r}xp_t^\nu(x,y)+\f{2r^{1/2}}{1-r}y p_t^{\nu+1}(x,y)
	\end{aligned}
\end{equation}

\begin{prop}
	\label{prop- prop 1} For $m\in \mathbb N$, we have
	\begin{equation}
		\label{eq- delta k pt k ge 1}
		\Big| \Big(\f{x}{\sqrt t}\Big)^k\delta_\nu^m p_t^\nu(x,y)\Big|\lesi_{\nu,k,m}   \f{1}{t^{(m+1)/2}}\exp\Big(-\f{|x-y|^2}{ct}\Big)\Big(1+\f{\sqrt t}{\rho_\nu(x)}+\f{\sqrt t}{\rho_\nu(y)}\Big)^{-(\nu+1/2)}
	\end{equation}
	for all $\nu>-1/2$, $k\in \mathbb N$, $t\in (0,1)$ and $(x,y)\in D$, where 
	$$D:=\{(x,y): xy<t\} \cup\{(x,y): xy \ge t, y<x/2\}\cup\{(x,y): xy \ge t, y>2x\}.
	$$
	Consequently, for $m\in \mathbb N$,
	\[
	|\delta_\nu^m p_t^\nu(x,y)|\lesi_{\nu,k,m}   \f{1}{t^{(m+1)/2}}\exp\Big(-\f{|x-y|^2}{ct}\Big)\Big(1+\f{\sqrt t}{\rho_\nu(x)}+\f{\sqrt t}{\rho_\nu(y)}\Big)^{-(\nu+1/2)}
	\]
	for all $\nu>-1/2$, $k\in \mathbb N$, $t\in (0,1)$ and $(x,y)\in D$.
	
\end{prop}
\begin{proof}

We will prove the proposition by induction.

	\noindent $\bullet$  For $m = 0, 1$, the inequality 	\eqref{eq- delta k pt k ge 1} was proved in \cite[Proposition 3.4]{B}.
		
	\bigskip
	
	\noindent $\bullet$ Assume that  \eqref{eq- delta k pt k ge 1} holds true for $m = 0,1,\ldots,  \ell$ for some $ \ell\in \mathbb N$, i.e, for $m=0,1,\ldots,  \ell$, we have
		\begin{equation}
		\label{eq- delta k pt k ge 1 hypothesis}
		\Big| \Big(\f{x}{\sqrt t}\Big)^k\delta_\nu^m p_t^\nu(x,y)\Big|\lesi_{\nu,k}   \f{1}{t^{(m+1)/2}}\exp\Big(-\f{|x-y|^2}{ct}\Big)\Big(1+\f{\sqrt t}{\rho_\nu(x)}+\f{\sqrt t}{\rho_\nu(y)}\Big)^{-(\nu+1/2)}
	\end{equation}
	for all $\nu>-1/2$, $k\in \mathbb N$, $t\in (0,1)$ and $(x,y)\in D$.
	
	We need to show that 
	\eqref{eq- delta k pt k ge 1} is also true for $\ell+1$. 	
	Indeed,  from \eqref{eq- chain rule 2} we have
	\[
	\delta_\nu^{\ell+1} p_t^\nu(x,y)=  -\f{2r}{1-r}\delta_\nu^{\ell}[xp_t^\nu(x,y)] + \f{2r^{1/2}}{1-r}y \delta_\nu^{\ell} p_t^{\nu+1}(x,y).
	\]
	Moreover, since $r = e^{-4t}$ we have $1+r\sim r\sim 1$ and $1-r\sim t$ as $t\in (0,1)$. This and \eqref{eq- delta k pt k ge 1 hypothesis} imply that
	\[
	\begin{aligned}
		\Big(\f{x}{\sqrt t}\Big)^k|\delta_\nu^{\ell+1} p_t^\nu(x,y)|
		&\lesi    \f{1}{t}\Big(\f{x}{\sqrt t}\Big)^k|\delta^\ell_\nu[xp_t^\nu(x,y)]| + \f{y}{t}\Big(\f{x}{\sqrt t}\Big)^k| \delta^\ell_\nu p_t^{\nu+1}(x,y)|\\
		&=: E_1 +E_2
	\end{aligned}
	\]
	for all $\nu>-1/2$, $k\in \mathbb N$, $t\in (0,1)$ and $(x,y)\in D$.

	For the term $E_1$, using \eqref{eq- delta k xf} and \eqref{eq- delta k pt k ge 1 hypothesis},
	\[
	\begin{aligned}
		E_1 &\lesi \f{1}{t}\Big(\f{x}{\sqrt t}\Big)^kx|\delta^\ell_\nu p_t^\nu(x,y)|+\f{1}{t}\Big(\f{x}{\sqrt t}\Big)^k|\delta^{\ell-1}_\nu p_t^\nu(x,y)|\\
		&\lesi \f{1}{t^{(\ell+2)/2}}\exp\Big(-\f{|x-y|^2}{ct}\Big)\Big(1+\f{\sqrt t}{\rho_\nu(x)}+\f{\sqrt t}{\rho_\nu(y)}\Big)^{-(\nu+1/2)}
	\end{aligned}
	\]
	as desired.

	For the second term $E_2$, using \eqref{eq- delta k} we obtain
	\[
	\begin{aligned}
		E_2&\lesi \f{y}{t}\Big(\f{x}{\sqrt t}\Big)^k\Big[|\delta_{\nu+1}^\ell p_t^{\nu+1}(x,y)|+\f{1}{x}|\delta_{\nu+1}^{\ell-1} p_t^{\nu+1}(x,y)|\Big].
	\end{aligned}
	\]
Applying \eqref{eq- delta k pt k ge 1 hypothesis} to estimate each term on the right hand side of the above inequality, we obtain
	\[
	\begin{aligned}
		\f{y}{t}\Big(\f{x}{\sqrt t}\Big)^k |\delta_{\nu+1}^\ell p_t^{\nu+1}(x,y)|&\le \f{|x-y|}{t}\Big(\f{x}{\sqrt t}\Big)^k|\delta_{\nu+1}^\ell p_t^{\nu+1}(x,y)|+\f{1}{\sqrt t}\Big(\f{x}{\sqrt t}\Big)^{k+1}|\delta_{\nu+1}^\ell p_t^{\nu+1}(x,y)|\\
		&\lesi \f{1}{t^{(\ell +2)/2}}\exp\Big(-\f{|x-y|^2}{ct}\Big)\Big(1+\f{\sqrt t}{\rho_\nu(x)}+\f{\sqrt t}{\rho_\nu(y)}\Big)^{-(\nu+1/2)}
	\end{aligned}
	\]
	and
	\[
	\begin{aligned}
		\f{y}{t}\Big(\f{x}{\sqrt t}\Big)^k\f{1}{x} |\delta_{\nu+1}^{\ell-1} p_t^{\nu+1}(x,y)|&\le \f{|x-y|}{t}\Big(\f{x}{\sqrt t}\Big)^k\f{1}{x}|\delta_{\nu+1}^{\ell-1} p_t^{\nu+1}(x,y)|+\f{1}{t}\Big(\f{x}{\sqrt t}\Big)^{k}|\delta_{\nu+1}^{\ell-1} p_t^{\nu+1}(x,y)|\\
		&\lesi \f{1}{t^{(\ell +1)/2}}\f{1}{x}\exp\Big(-\f{|x-y|^2}{ct}\Big)\Big(1+\f{\sqrt t}{\rho_\nu(x)}+\f{\sqrt t}{\rho_\nu(y)}\Big)^{-(\nu+3/2)}\\
		& \ \ \ +\f{1}{t^{(\ell +2)/2}}\exp\Big(-\f{|x-y|^2}{ct}\Big)\Big(1+\f{\sqrt t}{\rho_\nu(x)}+\f{\sqrt t}{\rho_\nu(y)}\Big)^{-(\nu+1/2)}\\
		& \lesi \f{1}{t^{(\ell +2)/2}}\exp\Big(-\f{|x-y|^2}{ct}\Big)\Big(1+\f{\sqrt t}{\rho_\nu(x)}+\f{\sqrt t}{\rho_\nu(y)}\Big)^{-(\nu+1/2)},
	\end{aligned}
	\]	
	where in the last inequality we use the fact that
	\begin{equation}\label{eq- x rho x < 1}
	\f{1}{x}\Big(1+\f{\sqrt t}{\rho_\nu(x)}+\f{\sqrt t}{\rho_\nu(y)}\Big)^{-1}\le \f{\rho_\nu(x)}{x\sqrt t } \lesi \f{1}{\sqrt t},
	\end{equation}
	since $\rho_\nu(x)\le x$.
	
	Hence,
	\[
	E_2\lesi \f{1}{t^{(\ell+2)/2}}\exp\Big(-\f{|x-y|^2}{ct}\Big)\Big(1+\f{\sqrt t}{\rho_\nu(x)}+\f{\sqrt t}{\rho_\nu(y)}\Big)^{-(\nu+1/2)}.
	\]
	The estimates of $E_1$ and $E_2$  ensure \eqref{eq- delta k pt k ge 1} for $\ell +1$.
	
	This completes our proof.
\end{proof}

\begin{prop}
	\label{prop- prop 2} For $m \in \mathbb N$ we have
	\begin{equation}
		\label{eq- delta k pt k ge 1 the case t > 1}
		\Big| x^k\delta_\nu^m p_t^\nu(x,y)\Big|\lesi_{\nu,k,m}   \f{e^{-t/2^{m+2}}}{t^{(m+1)/2}}\exp\Big(-\f{|x-y|^2}{ct}\Big)\Big(1+\f{\sqrt t}{\rho_\nu(x)}+\f{\sqrt t}{\rho_\nu(y)}\Big)^{-(\nu+1/2)}
	\end{equation}
	for all $\nu>-1/2$, $k\in \mathbb N$, $t\ge 1$ and $x,y>0$.
	
	Consequently, for $m \in \mathbb N$ we have
	\begin{equation*}
		\Big| \delta_\nu^m p_t^\nu(x,y)\Big|\lesi_{\nu,k}   \f{1}{t^{(m+1)/2}}\exp\Big(-\f{|x-y|^2}{ct}\Big)\Big(1+\f{\sqrt t}{\rho_\nu(x)}+\f{\sqrt t}{\rho_\nu(y)}\Big)^{-(\nu+1/2)}
	\end{equation*}
	for all  $\nu>-1/2$, $t\ge 1$ and $x,y>0$.
\end{prop}

\begin{proof}
	We will prove the proposition by induction.
	
	\noindent $\bullet$ For $m=0,1$, the estimate \eqref{eq- delta k pt k ge 1 the case t > 1} was proved in \cite[Proposition 3.5]{B}.
	
	\noindent $\bullet$ Assume that \eqref{eq- delta k pt k ge 1 the case t > 1} holds true for $m=0,1,\ldots,\ell$ for some $\ell\in \mathbb N$, i.e., for $m=0,1,\ldots,\ell$ we have
	\begin{equation}
		\label{eq- delta k pt k ge 1 the case t > 1 hypothesis}
		\Big| x^k\delta_\nu^m p_t^\nu(x,y)\Big|\lesi_{\nu,k,m}   \f{e^{-t/2^{m+2}}}{t^{(m+1)/2}}\exp\Big(-\f{|x-y|^2}{ct}\Big)\Big(1+\f{\sqrt t}{\rho_\nu(x)}+\f{\sqrt t}{\rho_\nu(y)}\Big)^{-(\nu+1/2)}
	\end{equation}
	for all $\nu>-1/2$, $k\in \mathbb N$, $t\ge 1$ and $x,y>0$.

	We need to verify the estimate for $m=\ell+1$. From \eqref{eq- chain rule 2} we have
	\[
	\delta_\nu^{\ell+1} p_t^\nu(x,y)=  -\f{2r}{1-r}\delta_\nu^{\ell} [xp_t^\nu(x,y)] + \f{2r^{1/2}}{1-r}y \delta_\nu^{\ell} p_t^{\nu+1}(x,y),
	\]
	which implies that
	\[
	\begin{aligned}
		x^k|\delta_\nu^{\ell+1} p_t^\nu(x,y)|&\lesi   x^k|\delta_\nu^{\ell} [xp_t^\nu(x,y)]| + yx^k| \delta_\nu^{\ell} p_t^{\nu+1}(x,y)|,
	\end{aligned}
	\]
	since $1+r\sim 1-r\sim  1$ and $r\le 1$ as $t\ge 1$.

	This, together with \eqref{eq- delta k xf} and \eqref{eq- delta k}, further implies that
	\[
	\begin{aligned}
		x^k|\delta_\nu^{\ell+1} p_t^\nu(x,y)|&\lesi   x^k |\delta_\nu^{\ell-1}p_t^\nu(x,y)| + x^{k+1}|\delta_\nu^{\ell} p_t^\nu(x,y)| + yx^k| \delta_{\nu+1}^{\ell}p_t^{\nu+1}(x,y)|+ \f{yx^k}{x}| \delta_{\nu+1}^{\ell-1}p_t^{\nu+1}(x,y)|,
	\end{aligned}
	\]
	Using \eqref{eq- delta k pt k ge 1 the case t > 1 hypothesis},
	\[
	\begin{aligned}
		x^k |\delta_\nu^{\ell-1}p_t^\nu(x,y)| + x^{k+1}|\delta_\nu^{\ell} p_t^\nu(x,y)|&\lesi \Big[\f{e^{-t/2^{\ell+1}}}{ t^{\ell/2}} + \f{e^{-t/2^{\ell+2}}}{t^{(\ell+1)/2}}\Big]\exp\Big(-\f{|x-y|^2}{ct}\Big)\Big(1+\f{\sqrt t}{\rho_\nu(x)}+\f{\sqrt t}{\rho_\nu(y)}\Big)^{-(\nu+1/2)}\\
		&\lesi \f{e^{-t/2^{\ell+3}}}{t^{(\ell+2)/2}}\exp\Big(-\f{|x-y|^2}{ct}\Big)\Big(1+\f{\sqrt t}{\rho_\nu(x)}+\f{\sqrt t}{\rho_\nu(y)}\Big)^{-(\nu+1/2)}.
	\end{aligned}
	\]
	For the next term $yx^k| \delta_{\nu+1}^{\ell}p_t^{\nu+1}(x,y)|$ we write
	\begin{equation}\label{eq- estimates related to y}
	yx^k| \delta_{\nu+1}^{\ell}p_t^{\nu+1}(x,y)|\le |x-y|x^k| \delta_{\nu+1}^{\ell}p_t^{\nu+1}(x,y)|+x^{k+1}| \delta_{\nu+1}^{\ell}p_t^{\nu+1}(x,y)|
	\end{equation}
	and then arguing similarly, we also obtain
	\[
	\begin{aligned}
	yx^k| \delta_{\nu+1}^{\ell}p_t^{\nu+1}(x,y)|&\lesi \f{e^{-t/2^{\ell+3}}}{t^{(\ell+2)/2}}\exp\Big(-\f{|x-y|^2}{ct}\Big)\Big(1+\f{\sqrt t}{\rho_{\nu+1}(x)}+\f{\sqrt t}{\rho_{\nu+1}(y)}\Big)^{-(\nu+1/2)}\\
	&\sim \f{e^{-t/2^{\ell+3}}}{t^{(\ell+2)/2}}\exp\Big(-\f{|x-y|^2}{ct}\Big)\Big(1+\f{\sqrt t}{\rho_\nu(x)}+\f{\sqrt t}{\rho_\nu(y)}\Big)^{-(\nu+1/2)},
		\end{aligned}
	\]
	where in the last inequality we used $\rho_{\nu+1}(x)=\rho_\nu(x)$ for all $x\in (0,\vc)$.

	For the last term, we write
	\[
	\f{yx^k}{x}| \delta_{\nu+1}^{\ell-1}p_t^{\nu+1}(x,y)|\le \f{|x-y|x^k}{x}| \delta_{\nu+1}^{\ell-1}p_t^{\nu+1}(x,y)|+yx^k| \delta_{\nu+1}^{\ell-1}p_t^{\nu+1}(x,y)|.
	\]
	Since the term $yx^k| \delta_{\nu+1}^{\ell-1}p_t^{\nu+1}(x,y)|$ can be estimated similarly to \eqref{eq- estimates related to y}, we need only to take case of the first term on the RHS of the above inequality. Indeed, by using  \eqref{eq- delta k pt k ge 1 the case t > 1 hypothesis} and the fact $\rho_{\nu+1}(x)=\rho_\nu(x)$ for all $x\in (0,\vc)$, we have
	\[
	\begin{aligned}
		\f{|x-y|x^k}{x}| \delta_{\nu+1}^{\ell-1}p_t^{\nu+1}(x,y)|&\lesi \f{1}{x}\f{e^{-t/2^{\ell+1}}}{t^{\ell/2}}\exp\Big(-\f{|x-y|^2}{ct}\Big)\Big(1+\f{\sqrt t}{\rho_\nu(x)}+\f{\sqrt t}{\rho_\nu(y)}\Big)^{-(\nu+3/2)}\\
		&\lesi \f{e^{-t/2^{\ell+3}}}{t^{(\ell+2)/2}}\exp\Big(-\f{|x-y|^2}{ct}\Big)\Big(1+\f{\sqrt t}{\rho_\nu(x)}+\f{\sqrt t}{\rho_\nu(y)}\Big)^{-(\nu+1/2)},
	\end{aligned}
	\]
where in the last inequality we used \eqref{eq- x rho x < 1}.
	
	This completes our proof.
\end{proof}

\bigskip

\begin{prop}
	\label{prop 3}  For $m \in \mathbb N\setminus \{0\}$ we have
	\begin{equation}
		\label{eq- delta k pt k}
		| \delta_\nu^m p_t^\nu(x,y)|+\f{x}{t}|\delta_\nu^{m-1}[p_t^\nu(x,y)-p_t^{\nu+1}(x,y)]|\lesi   \f{1}{t^{(m+1)/2}}\exp\Big(-\f{|x-y|^2}{ct}\Big)\Big(1+\f{\sqrt t}{\rho_\nu(x)}+\f{\sqrt t}{\rho_\nu(y)}\Big)^{-(\nu+1/2)}
	\end{equation}
	for all $\nu>-1/2$, $t\in (0,1)$ and $(x,y)$ satisfying $xy\ge t$ and $x/2\le y\le 2x$.
\end{prop}
\begin{proof}
In this situation, we have $x\sim y$ and $x\gtrsim \sqrt t$.

We will prove the proposition by induction. Note that in this case $1-r\sim t$ and $r<1$ as $t\in (0,1)$, where $r=e^{-4t}$.

	\noindent $\bullet$ For $m=1$, the estimate for $\delta_\nu^m p_t^\nu(x,y)$ was proved in \cite{B}. For the second term, using \eqref{eq5-Inu} and Proposition \ref{prop-heat kernel},
		\[
		\begin{aligned}
			\f{x}{t}|p_t^\nu(x,y)-p_t^{\nu+1}(x,y)|&\lesi \f{x}{t}\f{t}{xy} |p_t^{\nu+1}(x,y)|\\
			&\lesi \f{1}{t}\exp\Big(-\f{|x-y|^2}{ct}\Big)\Big(1+\f{\sqrt t}{\rho_\nu(x)}+\f{\sqrt t}{\rho_\nu(y)}\Big)^{-(\nu+1/2)},
		\end{aligned}
		\]
		provided that  $xy\ge t$ and $x/2\le y\le 2x$.
		
		This proves the inequality for the case $m=1$.
		
		\medskip

	\noindent $\bullet$ Assume that the estimate holds true for $m=1,\ldots, \ell$ for some $\ell\ge 1$, i.e., for $m=1,\ldots, \ell$,
	\begin{equation}
		\label{eq- delta k pt k hypothesis}
		| \delta_\nu^m p_t^\nu(x,y)|+\f{x}{t}|\delta_\nu^{m-1}[p_t^\nu(x,y)-p_t^{\nu+1}(x,y)]|\lesi   \f{1}{t^{(m+1)/2}}\exp\Big(-\f{|x-y|^2}{ct}\Big)\Big(1+\f{\sqrt t}{\rho_\nu(x)}+\f{\sqrt t}{\rho_\nu(y)}\Big)^{-(\nu+1/2)}
	\end{equation}
	for all $\nu>-1/2$, $t\in (0,1)$ and all $x,y$ satisfying $xy\ge t$ and $x/2\le y\le 2x$.

	We now prove the estimate for $\ell+1$. We first show that 
	\[
	 \f{x}{t}|\delta_\nu^{\ell}[p_t^\nu(x,y)-p_t^{\nu+1}(x,y)]|\lesi   \f{1}{t^{(\ell+2)/2}}\exp\Big(-\f{|x-y|^2}{ct}\Big)\Big(1+\f{\sqrt t}{\rho_\nu(x)}+\f{\sqrt t}{\rho_\nu(y)}\Big)^{-(\nu+1/2)}
	\]
		for all $\nu>-1/2$, $t\in (0,1)$ and all $x,y $ satisfying $xy\ge t$ and $x/2\le y\le 2x$.

	From \eqref{eq- chain rule 2},
	\begin{equation}\label{eq-del pt new}
		\begin{aligned}
			\delta_\nu p_t^\nu(x,y)&=-\f{2r}{1-r}x[p_t^\nu(x,y)-p_t^{\nu+1}(x,y)]-\f{2r}{1-r}xp_t^{\nu+1}(x,y)+\f{2r^{1/2}}{1-r}y p_t^{\nu+1}(x,y)\\	 	 
			&=-\f{2r}{1-r}x[p_t^\nu(x,y)-p_t^{\nu+1}(x,y)]-\Big[\f{2r}{1-r}-\f{2r^{1/2}}{1-r}\Big]xp_t^{\nu+1}(x,y)+\f{2r^{1/2}}{1-r}(y-x) p_t^{\nu+1}(x,y)\\
			&=-\f{2r}{1-r}x[p_t^\nu(x,y)-p_t^{\nu+1}(x,y)]+ \f{2r^{1/2}}{1+r^{1/2}} xp_t^{\nu+1}(x,y)+\f{2r^{1/2}}{1-r}(y-x) p_t^{\nu+1}(x,y).
		\end{aligned}
	\end{equation}
Using this and \eqref{eq- delta k} to write
	\[
	\begin{aligned}
	\f{x}{t}|\delta_\nu^{\ell}[p_t^\nu(x,y)-p_t^{\nu+1}(x,y)]| &=\f{x}{t}|\delta_\nu^{\ell-1}[\delta_\nu p_t^\nu(x,y)-\delta_\nu p_t^{\nu+1}(x,y)]|\\
		&=\f{x}{t}|\delta_\nu^{\ell-1}[\delta_\nu p_t^\nu(x,y)-\delta_{\nu+1} p_t^{\nu+1}(x,y) -\f{1}{x}p_t^{\nu+1}(x,y)]|\\
		&\lesi \f{x}{t}|\delta_\nu^{\ell-1}[\delta_\nu p_t^\nu(x,y)-\delta_{\nu+1} p_t^{\nu+1}(x,y)| +\f{x}{t}|\delta_\nu^{\ell-1}\big[\f{1}{x}p_t^{\nu+1}(x,y)\big]|\\
		&=:E_1+E_2.
	\end{aligned}
	\]
	For $E_1$, using \eqref{eq-del pt new} and \eqref{eq6-Inu}, we have
	\[
	\begin{aligned}
		\delta_\nu p_t^\nu(x,y)-\delta_{\nu+1} p_t^{\nu+1}(x,y)=&-\f{2r}{1-r}x[p_t^\nu(x,y)-p_t^{\nu+2}(x,y)]+ \f{2r^{1/2}}{1+r^{1/2}} x[p_t^{\nu+1}(x,y)-p_t^{\nu+2}(x,y)]\\
		& \ \ \ +\f{2r^{1/2}}{1-r}(y-x) [p_t^{\nu+1}(x,y)-p_t^{\nu+2}(x,y)]\\
		=&-2(\nu+2)\f{r^{1/2}}{y} p_t^{\nu+1}(x,y)+ \f{2r^{1/2}}{1+r^{1/2}} x[p_t^{\nu+1}(x,y)-p_t^{\nu+2}(x,y)]\\
		& \ \ \ +\f{2r^{1/2}}{1-r}(y-x) [p_t^{\nu+1}(x,y)-p_t^{\nu+2}(x,y)].
	\end{aligned}
	\]
	This, together with the facts that $x\sim y\gtrsim \sqrt t$, $r\le 1 $ and $1-r \sim t$, implies that 
	\[
	\begin{aligned}
		E_1&\lesi \f{1}{t}|\delta_\nu^{\ell-1}p_t^{\nu+1}(x,y)|+\f{x}{t}|\delta_\nu^{\ell-1}\big[x(p_t^{\nu+1}(x,y)-p_t^{\nu+2}(x,y))\big]|\\
		& \ \ \ +\f{x}{t^2}|\delta_\nu^{\ell-1}\big[(x-y)(p_t^{\nu+1}(x,y)-p_t^{\nu+2}(x,y))\big]|.
	\end{aligned}
	\]
	Using \eqref{eq- delta k},  \eqref{eq- delta k pt k hypothesis} and the fact $\rho_{\nu+1}(x)=\rho_\nu(x)$ for all $x\in (0,\vc)$,
	\[
	\begin{aligned}
		\f{1}{t}|\delta_\nu^{\ell-1}p_t^{\nu+1}(x,y)|&\lesi \f{1}{t}|\delta_{\nu+1}^{\ell-1}p_t^{\nu+1}(x,y)|+\f{1}{tx}|\delta_{\nu+1}^{\ell-2}p_t^{\nu+1}(x,y)|\\
		&\lesi \f{1}{t}|\delta_{\nu+1}^{\ell-1}p_t^{\nu+1}(x,y)|+\f{1}{t^{3/2}}|\delta_{\nu+1}^{\ell-2}p_t^{\nu+1}(x,y)|\\
		&\lesi  \f{1}{t^{(\ell+2)/2}}\exp\Big(-\f{|x-y|^2}{ct}\Big)\Big(1+\f{\sqrt t}{\rho_\nu(x)}+\f{\sqrt t}{\rho_\nu(y)}\Big)^{-(\nu+1/2)},
	\end{aligned}
	\]
	as long as $xy\ge t$ and $x/2\le y\le 2x$.
	
	By \eqref{eq- delta k} and \eqref{eq- delta k xf}, we have
	\[
	\begin{aligned}
		\f{x}{t}|\delta_\nu^{\ell-1}&\big[x(p_t^{\nu+1}(x,y)-p_t^{\nu+2}(x,y))\big]|\\
		&\lesi \f{x}{t}|\delta_{\nu+1}^{\ell-1}\big[x(p_t^{\nu+1}(x,y)-p_t^{\nu+2}(x,y))\big]|+\f{x}{t}\f{1}{x}|\delta_{\nu+1}^{\ell-2}\big[x(p_t^{\nu+1}(x,y)-p_t^{\nu+2}(x,y))\big]|\\
		&\lesi \f{x^2}{t}|\delta_{\nu+1}^{\ell-1}\big[(p_t^{\nu+1}(x,y)-p_t^{\nu+2}(x,y))\big]|+\f{x}{t}|\delta_{\nu+1}^{\ell-2}\big[(p_t^{\nu+1}(x,y)-p_t^{\nu+2}(x,y))\big]|\\
		&+\f{x}{t} |\delta_{\nu+1}^{\ell-2}\big[(p_t^{\nu+1}(x,y)-p_t^{\nu+2}(x,y))\big]|+\f{x}{t}\f{1}{x} |\delta_{\nu+1}^{\ell-3}\big[(p_t^{\nu+1}(x,y)-p_t^{\nu+2}(x,y))\big]|.	
	\end{aligned}
	\]
	This, together with \eqref{eq- delta k pt k hypothesis} and the fact $\rho_{\nu+1}(x)=\rho_\nu(x)\lesi \min \{x, 1/x\}$, implies that
	\[
	\begin{aligned}
		\f{x}{t}|\delta_\nu^{\ell-1}&\big[x(p_t^{\nu+1}(x,y)-p_t^{\nu+2}(x,y))\big]|
		&\lesi \f{1}{t^{(\ell+2)/2}}\exp\Big(-\f{|x-y|^2}{ct}\Big)\Big(1+\f{\sqrt t}{\rho_\nu(x)}+\f{\sqrt t}{\rho_\nu(y)}\Big)^{-(\nu+1/2)}.
	\end{aligned}
	\]
	Similarly,
	\[
	\begin{aligned}
		\f{x}{t^2}|\delta_\nu^{\ell-1}\big[(x-y)(p_t^{\nu+1}(x,y)-p_t^{\nu+2}(x,y))\big]|
		&\lesi \f{1}{t^{(\ell+2)/2}}\exp\Big(-\f{|x-y|^2}{ct}\Big)\Big(1+\f{\sqrt t}{\rho_\nu(x)}+\f{\sqrt t}{\rho_\nu(y)}\Big)^{-(\nu+1/2)}.
	\end{aligned}
	\]
	Hence,
	\[
	E_1\lesi \f{1}{t^{(\ell+2)/2}}\exp\Big(-\f{|x-y|^2}{ct}\Big)\Big(1+\f{\sqrt t}{\rho_\nu(x)}+\f{\sqrt t}{\rho_\nu(y)}\Big)^{-(\nu+1/2)}. 
	\]
	
	For the term $E_2$, by \eqref{eq- delta k} we have
	\[
	\begin{aligned}
		E_2 \lesi  \f{x}{t} \Big|\delta^{\ell-1}_{\nu+1}\Big[\f{1}{x}p_t^{\nu+1}(x,y)\Big]\Big|+ \f{1}{t} \Big|\delta^{\ell-2}_{\nu+1}\Big[\f{1}{x}p_t^{\nu+1}(x,y)\Big]\Big|.
	\end{aligned}
	\]

	We now use the  formula
	\[
	\delta^k_\nu\Big(\f{1}{x}f\Big)=\sum_{i=0}^k \f{c_i}{x^{k-i+1}} \delta_\nu^i f,
	\]
	where $c_i$ are constants, and the fact that $x\gtrsim \sqrt t$ to obtain
	\[
	\begin{aligned}
		E_2 &\lesi \sum_{i=0}^{\ell-1}\f{1}{x^{(\ell-1)-i+1}}\f{x}{t} |\delta^{i}_{\nu+1} p_t^{\nu+1}(x,y)|+ \sum_{i=0}^{\ell-2}\f{1}{tx^{(\ell-2)-i+1}} \delta^{\ell-2}_{\nu+1}\Big[\f{1}{x}p_t^{\nu+1}(x,y)\Big]\\
		&\lesi \sum_{i=0}^{\ell-1}\f{1}{t^{(\ell-i+1)/2}}  |\delta^{i}_{\nu+1} p_t^{\nu+1}(x,y)|+ \sum_{i=0}^{\ell-2}\f{1}{t^{(\ell-i+1)/2}} \delta^{\ell-2}_{\nu+1}\Big[\f{1}{x}p_t^{\nu+1}(x,y)\Big]\\
		&\lesi \f{1}{t^{(\ell+2)/2}}\exp\Big(-\f{|x-y|^2}{ct}\Big)\Big(1+\f{\sqrt t}{\rho_\nu(x)}+\f{\sqrt t}{\rho_\nu(y)}\Big)^{-(\nu+1/2)},
	\end{aligned}
	\]
where in the last inequality we used  \eqref{eq- delta k pt k hypothesis}.	

Hence, we have proved that for $m=1,\ldots, \ell$,
\begin{equation}\label{eq- different p}
 \f{x}{t}|\delta_\nu^{m}[p_t^\nu(x,y)-p_t^{\nu+1}(x,y)]|\lesi   \f{1}{t^{(m+2)/2}}\exp\Big(-\f{|x-y|^2}{ct}\Big)\Big(1+\f{\sqrt t}{\rho_\nu(x)}+\f{\sqrt t}{\rho_\nu(y)}\Big)^{-(\nu+1/2)}
\end{equation}
for all $\nu>-1/2$, $t\in (0,1)$ and $(x,y)$ satisfying $xy\ge t$ and $x/2\le y\le 2x$.

\bigskip

We now take care of $\delta_\nu^\ell p_t^\nu(x,y)$. From \eqref{eq-del pt new}, we have
	\[
	 \begin{aligned}
		\delta^{\ell+1}_\nu p_t^\nu(x,y) 
		&=-\f{2r}{1-r}\delta^{\ell}_\nu \big[x(p_t^\nu(x,y)-p_t^{\nu+1}(x,y))\big]+ \f{2r^{1/2}}{1+r^{1/2}} \delta^{\ell}_\nu \big[xp_t^{\nu+1}(x,y)]\\
		& \ \ \ \ +\f{2r^{1/2}}{1-r}\delta^{\ell}_\nu \big[(y-x) p_t^{\nu+1}(x,y)\big],
	\end{aligned}
	\]
	which implies that
	\[
	\begin{aligned}
		|\delta^{\ell+1}_\nu p_t^\nu(x,y)| 
		&\lesi \f{1}{t}\big|\delta^{\ell}_\nu \big[x(p_t^\nu(x,y)-p_t^{\nu+1}(x,y))\big]\big|+  \big|\delta^{\ell}_\nu \big[xp_t^{\nu+1}(x,y)]\big|+\f{1}{t}\big|\delta^{\ell}_\nu \big[(y-x) p_t^{\nu+1}(x,y)\big]\big|\\
		&=:F_1+F_2+F_3.
	\end{aligned}
	\]
	Applying  \eqref{eq- delta k xf}, \eqref{eq- delta k} and $t\in (0,1)$, we obtain
	\[
	\begin{aligned}
		F_2+F_3	&\le x|\delta_\nu^\ell p_t^{\nu+1}(x,y)|+|\delta_\nu^{\ell-1} p_t^{\nu+1}(x,y)|+ \f{|x-y|}{t}|\delta_\nu^\ell p_t^{\nu+1}(x,y)|+\f{1}{t}|\delta_\nu^{\ell-1} p_t^{\nu+1}(x,y)|  \\
		&\le x|\delta_\nu^\ell p_t^{\nu+1}(x,y)| + \f{|x-y|}{t}|\delta_\nu^\ell p_t^{\nu+1}(x,y)|+\f{1}{t}|\delta_\nu^{\ell-1} p_t^{\nu+1}(x,y)|  \\
		&\le x|\delta_{\nu+1}^\ell p_t^{\nu+1}(x,y)|+ |\delta_{\nu+1}^{\ell-1} p_t^{\nu+1}(x,y)| \\
		&+\f{|x-y|}{t}|\delta_{\nu+1}^\ell p_t^{\nu+1}(x,y)|+\f{|x-y|}{t}\f{1}{x}|\delta_{\nu+1}^{\ell-1} p_t^{\nu+1}(x,y)|\\
		&+\f{1}{t}|\delta_{\nu+1}^{\ell-1} p_t^{\nu+1}(x,y)| +\f{1}{xt}|\delta_{\nu+1}^{\ell-2} p_t^{\nu+1}(x,y)|. 
	\end{aligned}
	\]
	Using \eqref{eq- delta k pt k hypothesis}, \eqref{eq- x rho x < 1}  and the fact $\rho_{\nu+1}(x)=\rho_\nu(x)\lesi 
	\min\{x, 1/x\}$, we obtain 
	\[
	F_2+F_3\lesi \f{1}{t^{(\ell+2)/2}}\exp\Big(-\f{|x-y|^2}{ct}\Big)\Big(1+\f{\sqrt t}{\rho_\nu(x)}+\f{\sqrt t}{\rho_\nu(y)}\Big)^{-(\nu+1/2)}.
	\]
	It remains to estimate $F_1$. Using \eqref{eq- delta k xf}, we have
	\[
	\begin{aligned}
		F_1&\lesi \f{x}{t}|\delta_\nu^{\ell}[p_t^\nu(x,y)-p_t^{\nu+1}(x,y)]|+\f{1}{t}|\delta_\nu^{\ell-1}[p_t^\nu(x,y)-p_t^{\nu+1}(x,y)]|\\
		&\lesi \f{x}{t}|\delta_\nu^{\ell}[p_t^\nu(x,y)-p_t^{\nu+1}(x,y)]|+\f{1}{\sqrt t}\f{x}{t}|\delta_\nu^{\ell-1}[p_t^\nu(x,y)-p_t^{\nu+1}(x,y)]|\\
			&\lesi \f{1}{t^{(\ell+2)/2}}\exp\Big(-\f{|x-y|^2}{ct}\Big)\Big(1+\f{\sqrt t}{\rho_\nu(x)}+\f{\sqrt t}{\rho_\nu(y)}\Big)^{-(\nu+1/2)},
	\end{aligned}
	\]
	as desired, where in the last inequality we used \eqref{eq- different p}.
	
	This completes our proof.
\end{proof}

From Propositions \ref{prop- prop 1}, \ref{prop- prop 2} and \ref{prop 3}, we have: 
\begin{prop}\label{prop-delta k p x y} 
	Let $\nu>-1/2$. For  $k\in \mathbb N$, we have
	\begin{equation}
		\label{eq- delta k pt}
		\begin{aligned}
			|\delta^k_\nu  p_t^\nu(x,y)|&\lesi  \f{1}{t^{(k+1)/2}}\exp\Big(-\f{|x-y|^2}{ct}\Big)\Big(1+\f{\sqrt t}{\rho_\nu(x)}+\f{\sqrt t}{\rho_\nu(y)}\Big)^{-(\nu+1/2)}
		\end{aligned}
	\end{equation}
	for all $t>0$ and all $x, y\in (0,\vc)$.
\end{prop}

\begin{prop}\label{prop-gradient x y}
	Let $\nu>-1/2$. For $k,j \in \mathbb N$, we have
	\begin{equation}
		\label{eq-partial k delta x pt}
		\begin{aligned}
			|\partial^k \delta_\nu^j p_t^\nu(x,y)|&\lesi  \Big[ \f{1}{\rho_\nu(x)^k}+\f{1}{t^{k/2}}\Big] \f{1}{t^{(j+1)/2}}\exp\Big(-\f{|x-y|^2}{ct}\Big)\Big(1+\f{\sqrt t}{\rho_\nu(x)}+\f{\sqrt t}{\rho_\nu(y)}\Big)^{-(\nu+1/2)}
		\end{aligned}
	\end{equation}
	for all $t>0$ and all $x, y\in (0,\vc)$.
\end{prop}
\begin{proof}
	For the sake of simplicity, we only prove the proposition for $j=0$ since the proof in the general case $j\ge 1$ can be done similarly. We will do this by induction.
	
$\bullet$	For $k=1$, we have
\[
\delta_\nu = \partial_x + x -\f{\nu+1/2}{x},
\]
which, together with the fact $\rho_\nu(x)\le \min\{x, 1/x\}$, implies that
\[
\begin{aligned}
	|\partial  p_t^\nu(x,y)| &\lesi  |\delta_\nu p_t^\nu(x,y)| + xp_t^\nu(x,y)+\f{1}{x}p_t^\nu(x,y)\\
	&\lesi |\delta_\nu p_t^\nu(x,y)| + \f{1}{\rho_\nu(x)}p_t^\nu(x,y)\\
	&\lesi \Big[\f{1}{\sqrt t}+\f{1}{\rho_\nu(x)}\Big]\f{1}{\sqrt t}\exp\Big(-\f{|x-y|^2}{ct}\Big)\Big(1+\f{\sqrt t}{\rho_\nu(x)}+\f{\sqrt t}{\rho_\nu(y)}\Big)^{-(\nu+1/2)},
\end{aligned}
\]
where in the last inequality we used Propositions \ref{prop-heat kernel} and \ref{prop-delta k p x y}.

$\bullet$	Assume that \eqref{eq-partial k delta x pt} holds true for $k=1,\ldots, \ell$ for some $\ell \in \mathbb N$, i.e., for $k=1,\ldots, \ell$, we have
\begin{equation}
	\label{eq-partial k delta x pt hypothesis}
	\begin{aligned}
		|\partial^k  p_t^\nu(x,y)|&\lesi   \Big[ \f{1}{\rho_\nu(x)^k}+\f{1}{t^{k/2}}\Big]  \f{1}{\sqrt t}\exp\Big(-\f{|x-y|^2}{ct}\Big)\Big(1+\f{\sqrt t}{\rho_\nu(x)}+\f{\sqrt t}{\rho_\nu(y)}\Big)^{-(\nu+1/2)}
	\end{aligned}
\end{equation}
	for all $t>0$ and all $x, y\in (0,\vc)$.

We now claim the estimate for $k=\ell+1$. It is easy to see that 
\[
\delta_\nu^{\ell+1} =\Big(\partial_x + x -\f{\nu+1/2}{x}\Big)^{\ell+1} = \partial_x^{\ell+1} + \sum_{i+j\le \ell+1  \atop j\le \ell}\Big[a_{ij}x^i \partial^j_x +b_{ij}\f{1}{x^i} \partial^j_x\Big],
\]
where $a_{ij}$ and $b_{ij}$ are constants.

It follows that 
\[
\partial_x^{\ell+1}=\delta_\nu^{\ell+1} - \sum_{i+j\le \ell+1  \atop j\le \ell}\Big[a_{ij}x^i \partial^j_x +b_{ij}\f{1}{x^i} \partial^j_x\Big].
\]
This, together with Proposition \ref{prop-delta k p x y}, \eqref{eq-partial k delta x pt hypothesis} and the fact that $\rho(x)\lesi \min\{x,x^{-1}\} =\min\{1, x,x^{-1}\}\le 1$, implies that 
\[
\begin{aligned}
	|\partial^{\ell+1}  p_t^\nu(x,y)|&\lesi   \Big[ \f{1}{t^{(\ell+1)/2}} +\sum_{i+j\le \ell+1  \atop j\le \ell}\f{1}{\rho_\nu(x)^i}\Big(\f{1}{t^{j/2}}+\f{1}{\rho_\nu(x)^j}\Big)\Big] \\
	& \ \ \ \times  \f{1}{\sqrt t}\exp\Big(-\f{|x-y|^2}{ct}\Big)\Big(1+\f{\sqrt t}{\rho_\nu(x)}+\f{\sqrt t}{\rho_\nu(y)}\Big)^{-(\nu+1/2)}.
\end{aligned}
\]
Since $\rho_\nu(x)\lesi 1$, we have
\[
\sum_{i+j\le \ell+1  \atop j\le \ell}\f{1}{\rho_\nu(x)^i}\Big(\f{1}{t^{j/2}}+\f{1}{\rho_\nu(x)^j}\Big)\lesi \f{1}{\rho_\nu(x)^{\ell+1}}  + \f{1}{t^{(\ell+1)/2}}.
\]
Therefore,
\[
	|\partial^{\ell+1}  p_t^\nu(x,y)|\lesi   \Big[\f{1}{\rho_\nu(x)^{\ell+1}}  + \f{1}{t^{(\ell+1)/2}}\Big]  \f{1}{\sqrt t}\exp\Big(-\f{|x-y|^2}{ct}\Big)\Big(1+\f{\sqrt t}{\rho_\nu(x)}+\f{\sqrt t}{\rho_\nu(y)}\Big)^{-(\nu+1/2)},
\]
which ensures \eqref{eq-partial k delta x pt} for $k=\ell+1$.

This completes our proof.

\end{proof}

\begin{prop}\label{prop-gradient x y dual delta}
	Let $\nu>-1/2$. Then for any $m, k,\ell\in \mathbb N$ with $\ell\ge k+2m$ we have
	\[
	|\mathcal L_\nu^m (\delta_\nu^*)^kp^{\nu+\ell}_t(x,y)|\lesi \f{1}{t^{(k+2m+1)/2}}\exp\Big(-\f{|x-y|^2}{ct}\Big)\Big(1+\f{\sqrt t}{\rho_\nu(x)}+\f{\sqrt t}{\rho_\nu(y)}\Big)^{-(\nu+1/2)}
	\]
	for all $t>0$ and $x,y\in (0,\vc)$.
\end{prop}
\begin{proof}
	We have
	\[
	\begin{aligned}
		-\delta_\nu^* &=\partial_x  - x+\frac{1}{x}\Big(\nu + \f{1}{2}\Big)\\
		&=\partial_x  + x-\frac{1}{x}\Big(\nu+\ell + \f{1}{2}\Big) -2x+\f{1}{x}(2\nu+\ell+1)\\
		&=\delta_{\nu+\ell}-2x+\f{1}{x}(2\nu+\ell+1), 
	\end{aligned}
	\]
	which implies that
	\[
	\begin{aligned}
		\delta^*_\nu &=-\delta_{\nu+\ell} +2x - \f{1}{x}\Big(2\nu+\ell + 1\Big).
	\end{aligned}
	\]
	Similarly,
	\[
		\delta_\nu = \delta_{\nu+\ell}  + \f{\ell}{x}.
	\]
	This, together with the facts that $\delta_{\nu+\ell}(fg)=f'g + f\delta_{\nu+\ell}g$ and $\mathcal L_\nu = \delta^*_\nu\delta_\nu +2(\nu+1)$, implies
	\[
	\mathcal L_\nu^m(\delta^*_\nu )^k=\sum_{i,j,h\ge 0 \atop i+j+h\le 2m+k}c_{ijh} \f{x^i}{x^j} \delta^h_{\nu+\ell},
	\]
	where $c_{ijh}$ are constants.

	Hence,
	\[
	\begin{aligned}
		|\mathcal L_\nu^m(\delta^*_\nu )^kp^{\nu+\ell}_t(x,y)|\lesi \sum_{i,j,h\ge 0 \atop i+j+h\le 2m+ k} \f{x^i}{x^j} |\delta^h_{\nu+\ell}p^{\nu+\ell}_t(x,y)|.
	\end{aligned}
	\]

By Proposition \ref{prop-delta k p x y} and the facts $\ell\ge 2m+k$ and  $\rho_\nu(x)= \rho_{\nu+\ell}(x)\sim \min\{x,1/x\}\le 1$, we further obtain
\[
\begin{aligned}
	|\mathcal L_\nu^m(\delta^*_\nu )^kp^{\nu+\ell}_t(x,y)|&\lesi \sum_{i,j,h\ge 0 \atop i+j+h\le 2m+k} \f{x^i}{x^j} \f{e^{-t/4}}{t^{(h+1)/2}}\exp\Big(-\f{|x-y|^2}{ct}\Big)\Big(1+\f{\sqrt t}{\rho_\nu(x)}+\f{\sqrt t}{\rho_\nu(y)}\Big)^{-(\nu+\ell+1/2)}\\
	&\lesi \sum_{i,j,h\ge 0 \atop i+j+h\le 2m+ k} \f{x^i}{x^j}\Big(\f{\sqrt t}{\rho_\nu(x)}\Big)^{-(i+j)} \f{e^{-t/4}}{t^{(h+1)/2}}\exp\Big(-\f{|x-y|^2}{ct}\Big)\Big(1+\f{\sqrt t}{\rho_\nu(x)}+\f{\sqrt t}{\rho_\nu(y)}\Big)^{-(\nu +1/2)}\\
		&\lesi \sum_{i,j,h\ge 0 \atop i+j+h\le 2m+ k}  \f{1}{t^{(i+j)/2}} \f{e^{-t/4}}{t^{(h+1)/2}}\exp\Big(-\f{|x-y|^2}{ct}\Big)\Big(1+\f{\sqrt t}{\rho_\nu(x)}+\f{\sqrt t}{\rho_\nu(y)}\Big)^{-(\nu +1/2)}\\
		&\lesi  \f{1}{t^{(k+2m+1)/2}}\exp\Big(-\f{|x-y|^2}{ct}\Big)\Big(1+\f{\sqrt t}{\rho_\nu(x)}+\f{\sqrt t}{\rho_\nu(y)}\Big)^{-(\nu +1/2)}.
\end{aligned}
\]
	This completes our proof.
\end{proof}

\bigskip

\subsection{The case $n\ge 2$} For  $\nu\in [-1/2,\vc)^n$, we define
\begin{equation}\label{eq- nu min}
\nu_{\min} =  \min\{\nu_j: j\in \mathcal J_\nu\}
\end{equation}
with the convention $\inf \emptyset = \vc$, where $\mathcal J_\nu=\{j: \nu_j>-1/2\}$. In what follows, by $A\le B^{-\vc}$ we mean for every $N>0$, $A\lesi_N B^{-N}$.

We recall that in what follows, $\rho_\nu$ is the critical  function defined in \eqref{eq- critical function}.

From \eqref{eq- prod ptnu} and \eqref{eq- equivalence of rho}, the following propositions are just direct consequences of Propositions \ref{prop- derivative heat kernel nu = -1/2}, \ref{prop-heat kernel}, \ref{prop-delta k p x y}, \ref{prop-gradient x y}  and \ref{prop-gradient x y dual delta} and the fact $\rho_{\nu+e_j}(x)\le \rho_{\nu}(x)$. 


\begin{prop}
	\label{prop- delta k pt d>2} Let $\nu\in [-1/2,\vc)^n$ and $\ell\in \mathbb N$. Then we have
	\begin{equation*}
		| \delta_\nu^\ell  p_t^\nu(x,y)|\lesi   \f{1}{t^{(n+\ell)/2}}\exp\Big(-\f{|x-y|^2}{ct}\Big)\Big(1+\f{\sqrt t}{\rho_\nu(x)}+\f{\sqrt t}{\rho_\nu(y)}\Big)^{-(\nu_{\min}+1/2)}
	\end{equation*}
	for $t>0$ and $x,y\in \mathbb R^n_+$.
\end{prop}

\begin{prop}\label{prop-gradient x y d>2}
		Let $\nu\in [-1/2,\vc)^n$. For $k,j \in \mathbb N^n$, we have
		\begin{equation}
			\label{eq-partial k delta x pt n>1}
			\begin{aligned}
				|\partial^k \delta_\nu^j p_t^\nu(x,y)|&\lesi  \Big[ \f{1}{\rho_\nu(x)^{|k|}}+\f{1}{t^{|k|/2}}\Big] \f{1}{t^{(|j|+1)/2}}\exp\Big(-\f{|x-y|^2}{ct}\Big)\Big(1+\f{\sqrt t}{\rho_\nu(x)}+\f{\sqrt t}{\rho_\nu(y)}\Big)^{-(\nu_{\min}+1/2)}
			\end{aligned}
		\end{equation}
		for all $t>0$ and all $x,y\in \mathbb R^n_+$.
	\end{prop}
	
For $\nu\in [-1/2,\vc)^n$ and $\ell=(\ell_1,\ldots, \ell_n)\in \mathbb N^n$, define $\ell_\nu = (\ell_\nu^1, \ldots, \ell_\nu^n)$, where
\begin{equation}\label{eq-ell nu}
\ell_\nu^j=\begin{cases}
	\ell_j, \ \ & \nu_j >-1/2\\
	0, \ \ &\nu_j=-1/2
\end{cases}
\end{equation}
for all $j=1,\ldots, n$.

\begin{prop}
	\label{prop- delta k pt d>2 dual delta} Let $\nu\in [-1/2,\vc)^n$. 	 Then for any $m \in \mathbb N$ and $k\in \mathbb N^n$ with $\ell\ge k+2\vec{m}$, where $\vec{m}=(m,\ldots, m)\in \mathbb R^n$, we have
	\[
	|\mathcal L_\nu^m (\delta_\nu^*)^kp^{\nu+\ell_\nu}_t(x,y)|\lesi \f{1}{t^{(k+2m+1)/2}}\exp\Big(-\f{|x-y|^2}{ct}\Big)\Big(1+\f{\sqrt t}{\rho_\nu(x)}+\f{\sqrt t}{\rho_\nu(y)}\Big)^{-(\nu+1/2)}
	\]
	for all $t>0$ and $x,y\in \mathbb R^n_+$.
\end{prop}
\bigskip

\section{Hardy spaces associated to the Laguerre operator and its duality}
In this section, we develop the theory of Hardy spaces adapted to the Laguerre operator $\mathcal{L}_\nu$. We begin by introducing the Hardy space via an atomic decomposition involving a new local atom related to the critical function $\rho_\nu$. Next, we characterize the Hardy space using the maximal function associated with the heat semigroup of $\mathcal{L}_\nu$. As a counterpart to the Hardy space, we also study a new BMO type space, which we establish as the dual of the Hardy space. The novelty of our theory lies in its extension to the full range $p \in (0,1]$.

\subsection{Hardy spaces associated to the Laguerre operator}

 We first introduce a new atomic Hardy space related to the critical function.

\begin{defn}\label{def: rho atoms}
	Let $p\in (0,1]$, $\nu\in [-1/2,\vc)^n$ and $\rho_\nu$ be the function as in \eqref{eq- critical function}.  A function $a$ is called a  $(p,\rho_\nu)$-atom associated to the ball $B(x_0,r)$ if
	\begin{enumerate}[{\rm (i)}]
		\item ${\rm supp}\, a\subset B(x_0,r)$;
		\item $\|a\|_{L^\vc}\leq |B(x_0,r)|^{-1/p}$;
		\item $\displaystyle \int a(x)x^\alpha dx =0$ for all multi-indices $\alpha$ with $|\alpha| \le n(1/p-1)$ if $r< \rho_\nu(x_0)$.
	\end{enumerate}
\end{defn}

If $\rho_\nu(x_0) =1$, the $(p,\rho_\nu)$-atom coincides with the local atoms defined in \cite{G}. Therefore, we can view a $(p,\rho_\nu)$ atom as a local atom associated to the critical function $\rho_\nu$.
Let $p\in (0,1]$. We  say that $f=\sum_j	\lambda_ja_j$ is an atomic $(p,\rho_\nu)$-representation if
$\{\lambda_j\}_{j=0}^\infty\in l^p$, each $a_j$ is a $(p,\rho_\nu)$-atom,
and the sum converges in $L^2(\mathbb{R}^n_+)$. The space $H^{p}_{\rho_\nu}(\mathbb{R}^n_+)$ is then defined as the completion of
\[
\left\{f\in L^2(\mathbb{R}^n_+):f \ \text{has an atomic
	$(p,\rho_\nu)$-representation}\right\}
\]
under the norm given by
$$
\|f\|_{H^{p}_{\rho_\nu}(\mathbb{R}^n_+)}=\inf\Big\{\Big(\sum_j|\lambda_j|^p\Big)^{1/p}:
f=\sum_j \lambda_ja_j \ \text{is an atomic $(p,\rho_\nu)$-representation}\Big\}.
$$

\begin{rem}
	\label{rem1}
	For a $(p,\rho_\nu)$ atom $a$ associated with the ball $B(x_0,r)$, if $r> \rho_\nu(x_0)$, by the argument used in the proof of Theorem 3.2 in \cite{YZ}, it can be verified that we can decompose $a =\sum_{j=1}^N\lambda_ja_j$ for some $N\in \mathbb N$, where $\sum_j |\lambda_j|^p \sim 1$ and $a_j$ is a $(p,\rho_\nu)$ atom associated with the ball $B(x_j,r_j)$ with $r_j=\rho(x_j)$ for each $j$. For this reason, in Definition \ref{def: rho atoms}, we might assume that $r\le \rho(x_0)$ for any $(p,\rho_\nu)$ atom associated to $B(x_0,r)$.	
\end{rem}

We next introduce the Hardy space associated with the Laguerre operator $\mathcal L_\nu$. For $f\in L^2(\mathbb{R}^n_+)$ we define 
\[
\mathcal M_{\mathcal L_\nu}f(x) = \sup_{t>0}|e^{-t^2\mathcal L_\nu}f(x)|
\]
for all $x\in \mathbb R^n_+$.\\

The maximal Hardy spaces associated to $\mathcal L_\nu$ are defined as follows.
\begin{defn}
	Given $p\in (0,1]$,  the Hardy space $H^p_{\mathcal L_\nu}(\mathbb{R}^n_+)$ is defined as the completion of
	\[
	\{f
	\in L^2(\mathbb{R}^n_+): \mathcal M_{\mathcal L_\nu}f \in L^p(\mathbb{R}^n_+)\}
	\]
	under the norm given by
	$$
	\|f\|_{H^p_{\mathcal L_\nu}(\mathbb{R}^n_+)}=\|\mathcal M_{\mathcal L_\nu}f\|_{L^p(\mathbb{R}^n_+)}.
	$$
	
\end{defn}
We have the following result.
\begin{thm}\label{mainthm2s}
	Let $\nu\in [-1/2,\vc)^n$ and $\gamma_\nu =  \nu_{\min}+1/2 $, where $\nu_{\min}$ is defined by \eqref{eq- nu min}. For $p\in (\f{n}{n+\gamma_\nu},1]$, we have
	\[
	H^{p}_{\rho_\nu}(\mathbb{R}^n_+)\equiv H^p_{\mathcal L_\nu}(\mathbb{R}^n_+)
	\]
	with equivalent norms.
\end{thm}
If all $\nu_j$ are large, then we have the larger range of $p$, i.e., $\f{n}{n+\gamma_\nu}$ is close to 0. Note that the critical function $\rho_\nu$ defined in \eqref{eq- critical function} does not satisfy the property \eqref{criticalfunction} as in \cite[Definition 2.2]{YZ}. In addition,  the heat kernel of $\mathcal L_\nu$ does not the H\"older continuity condition, which played an essential roles in \cite{YZ, BDK}. Consequently, the approaches outlined in \cite{YZ, BDK} (including \cite{CFYY}) cannot be directly applied in our specific context. This necessitates the development of new ideas and a novel approach.\\

It is worth noting that the equivalence between the atomic Hardy space $H^p_{\rho_\nu}(\mathbb R^n_+)$ and the Hardy space $H^p_{\mathcal L_\nu}(\mathbb{R}^n_+)$ was established in \cite{Dziu} for the specific case of $n=1, p=1$, and $\nu > -1/2$. However, the approach employed in \cite{Dziu} is not  applicable to higher-dimensional cases.\\


We now give the proof of Theorem \ref{mainthm2s}. 
\begin{proof}[Proof of Theorem \ref{mainthm2s}:]

	\bigskip
	
	\underline{We first prove that   $H^{p}_{{\rm at},\rho_\nu}(\mathbb{R}^n_+)\hookrightarrow H^p_{\mathcal L_\nu}(\mathbb{R}^n_+)$}.  Let $a$ be a $(p,\rho_\nu)$-atom associated with a ball $B:=B(x_0,r)$. By Remark \ref{rem1}, we might assume that $r\le \rho_\nu(x_0)$. We need to verify that 
	\[
	\|\mathcal M_{\mathcal L_\nu} a\|_p \lesi 1.
	\]
	To do this, we write
	\[
	\begin{aligned}
		\|\mathcal M_{\mathcal L_\nu} a\|_p&\lesi \|\mathcal M_{\mathcal L_\nu} a\|_{L^p(4B)} +\|\mathcal M_{\mathcal L_\nu} a\|_{L^p((4B)^c)}\\
		&\lesi E_1 + E_2.
	\end{aligned}
	\]
	By H\"older's inequality and the $L^2$-boundedness of $\mathcal M_{\mathcal L_\nu}$,
	\[
	\begin{aligned}
		\|\mathcal M_{\mathcal L_\nu} a\|_{L^p(4B)}&\lesi |4B|^{1/p-1/2} \|\mathcal M_{\mathcal L_\nu} a\|_{L^2(4B)}\\
		 &\lesi |4B|^{1/p-1/2} \|a\|_{L^2(B)}\\
		 &\lesi 1.
	\end{aligned}
	\]
	It remains to take care of the second term $E_2$. We now consider two cases.
	\medskip
	
	\textbf{Case 1: $r=\rho_\nu(x_0)$.} By Proposition \ref{prop-heat kernel}, for $x\in (4B)^c$,
	\[
	\begin{aligned}
		|\mathcal M_{\mathcal L_\nu} a(x)|\lesi \sup_{t>0}\int_{B}\f{1}{t^{n/2}}\exp\Big(-\f{|x-y|^2}{ct}\Big)\Big(\f{\rho_\nu(y)}{\sqrt t}\Big)^{\gamma_\nu}|a(y)|dy,
	\end{aligned}
	\]
where $\gamma_\nu=\nu_{\min}+1/2$.

	By Lemma \ref{lem-critical function}, $\rho_\nu(y)\sim \rho_\nu(x_0)$ for $y\in B$. This, together with the fact that $|x-y|\sim |x-x_0|$ for $x\in (4B)^c$ and $y\in B$, further imply
	\[
	\begin{aligned}
		\mathcal M_{\mathcal L_\nu} a(x)&\lesi \sup_{t>0}\int_{B}\f{1}{t^{n/2}}\exp\Big(-\f{|x-x_0|^2}{ct}\Big)\Big(\f{\rho_\nu(x_0)}{\sqrt t}\Big)^{\gamma_\nu}|a(y)|dy\\
		&\lesi \Big(\f{\rho_\nu(x_0)}{|x-x_0|}\Big)^{\gamma_\nu}\f{1}{|x-x_0|^n}\|a\|_1\\
		&\lesi \Big(\f{r}{|x-x_0|}\Big)^{\gamma_\nu}\f{1}{|x-x_0|^n}|B|^{1-1/p}
	\end{aligned}
	\]
	Hence we obtain
	\[
	\begin{aligned}
		\|\mathcal M_{\mathcal L_\nu} a\|_{L^p((4B)^c)}&\lesi |B|^{1-1/p} \Big[\int_{(4B)^c}\Big(\f{r}{|x-x_0|}\Big)^{p\gamma_\nu}\f{1}{|x-x_0|^{np}}dx\Big]^{1/p}\\
		&\lesi 1,
	\end{aligned}
	\]
	as long as $\f{n}{n+\gamma_\nu}<p\le 1$.
	
	\bigskip
	

\textbf{Case 2: $r<\rho_\nu(x_0)$.} Using the cancellation property $\int a(x) x^\alpha dx= 0$ for all $|\alpha|\le \lfloor n(1/p-1)\rfloor=:N_p$ and the Taylor expansion, we have
	\[
	\begin{aligned}
		\sup_{t>0} |e^{-t\mathcal L_\nu}a(x)|&= \sup_{t>0}\Big|\int_{B}[p_t^\nu(x,y)-p_t^\nu(x,x_0)]a(y)dy\Big|\\
		&= \sup_{t>0}\Big|\int_{B}\Big[p_t^\nu(x,y)-\sum_{|\alpha|\le N_p}\f{\partial^\alpha p_t^\nu(x,x_0)}{\alpha!}(y-x_0)^\alpha\Big]a(y)dy\Big|\\
		&= \sup_{t>0}\Big|\int_{B} \sum_{|\alpha|= N_p+1}\f{\partial^\alpha p_t^\nu(x,x_0+\theta(y-x_0))}{\alpha!}(y-x_0)^\alpha a(y)dy\Big|
	\end{aligned}
	\]
	for some $\theta \in (0,1)$.
	
	Since $y\in B$ we have $\rho_\nu(x_0+\theta(y-x_0))\sim \rho_\nu(x_0)$ and $|x-[x_0+\theta(y-x_0)]|\sim |x-x_0|$ for all $x\in (4B)^c$, $y\in B$ and $\theta \in (0,1)$, by  Proposition \ref{prop-gradient x y} we obtain, for $x\in (4B)^c$,
	\[
	\begin{aligned}
		\sup_{t>0} |e^{-t\mathcal L_\nu}a(x)|&\lesi  \sup_{t>0}\int_{B}\Big(\f{|y-y_0|}{\sqrt t}+\f{|y-y_0|}{\rho_\nu(x_0)}\Big)^{N_p+1}\f{1}{t^{n/2}}\exp\Big(-\f{|x-y|^2}{ct}\Big)\Big(1+\f{\sqrt t}{\rho_\nu(x_0)}\Big)^{-\gamma_\nu} |a(y)|dy\\
		&\sim  \sup_{t>0}\int_{B}\Big(\f{r}{\sqrt t}+\f{r}{\rho_\nu(x_0)}\Big)^{N_p+1}\f{1}{t^{n/2}}\exp\Big(-\f{|x-x_0|^2}{ct}\Big)\Big(1+\f{\sqrt t}{\rho_\nu(x_0)}\Big)^{-\gamma_\nu} \|a\|_1\\
		&\lesi   \sup_{t>0}  \Big(\f{r}{\sqrt t}\Big)^{N_p+1} \f{1}{t^{n/2}}\exp\Big(-\f{|x-x_0|^2}{ct}\Big) \|a\|_1\\
		& \ \ +   \sup_{t>0}  \Big(\f{r}{\rho_\nu(x_0)}\Big)^{N_p+1} \f{1}{t^{n/2}}\exp\Big(-\f{|x-x_0|^2}{ct}\Big)\Big(1+\f{\sqrt t}{\rho_\nu(x_0)}\Big)^{-\gamma_\nu} \|a\|_1\\	
		&= F_1 + F_2.
	\end{aligned}
	\]
	For the term $F_1$, it is straightforward to check that 
	\[
	\begin{aligned}
		F_1&\lesi \Big(\f{r}{|x-x_0|}\Big)^{N_p+1}\f{1}{|x-x_0|^n}|B|^{1-1/p}\\
		&\lesi \Big(\f{r}{|x-x_0|}\Big)^{(N_p+1)\wedge \gamma_\nu}\f{1}{|x-x_0|^n}|B|^{1-1/p},	
		\end{aligned}
	\]
where in the last inequality we used the fact $r\le |x-x_0|$ for $x\in (4B)^c$.
	
	For $F_2$, since $r<\rho_\nu(x_0)$, we have
	\[
	\begin{aligned}
		F_2&\lesi \sup_{t>0} \Big(\f{r}{\rho_\nu(x_0)}\Big)^{(N_p+1)\wedge \gamma_\nu}\f{1}{t^{n/2}}\exp\Big(-\f{|x-x_0|^2}{ct}\Big)\Big(\f{\rho_\nu(x_0)}{\sqrt t}\Big)^{(N_p+1)\wedge \gamma_\nu} \|a\|_1\\
		&\lesi \Big(\f{r}{|x-x_0|}\Big)^{^{(N_p+1)\wedge \gamma_\nu}}\f{1}{|x-x_0|^n}|B|^{1-1/p}.
	\end{aligned}
	\]
	 Taking this and the estimate of $F_1$ into account, we obtain
	 \[
	 \sup_{t>0} |e^{-t\mathcal L_\nu}a(x)|\lesi \Big(\f{r}{|x-x_0|}\Big)^{^{(N_p+1)\wedge \gamma_\nu}}\f{1}{|x-x_0|^n}|B|^{1-1/p}.
	 \]
	 Therefore,
	 \[
	 |\mathcal M_{\mathcal L_\nu} a(x)| \lesi \Big(\f{r}{|x-x_0|}\Big)^{^{(N_p+1)\wedge \gamma_\nu}}\f{1}{|x-x_0|}|B|^{1-1/p},
	 \]
	 which implies that
	 \[
	 \|\mathcal M_{\mathcal L_\nu} a\|_p\lesi 1,
	 \]
	 as long as $\f{n}{n+\gamma_\nu}< p\le 1$.

	 This completes the proof of the first step.

	 \bigskip
	 
	\underline{It remains to show that $H^p_{\mathcal L_\nu}(\mathbb{R}^n_+)\hookrightarrow  H^{p}_{\rho_\nu}(\mathbb{R}^n_+)$.} To do this, let $f\in H^p_{\mathcal L_\nu}(\mathbb{R}^n_+)\cap L^2(\Rn_+)$. By Theorem \ref{mainthm1-Hardy} we can write $f$ as a linear decomposition of $(p,N)_{\mathcal L_\nu}$ atoms with $N>n(\f{1}{p}-1)$. Assume that $a$ is a $(p,N)_{\mathcal L_\nu}$ atom associated with a ball $B$. If $r_B\ge \rho_\nu(x_B)$, then $a$ is also a $(p,\rho_\nu)$ atom. Therefore, it suffices to prove that $a\in H^p_{\rho_\nu}(\Rn_+)$ in the case $r_B<\rho_\nu(x_B)$. To do this, we recall that $a$ satisfies the following properties
	 \begin{enumerate}[{\rm (i)}]
	 	\item  $a=\mathcal L_\nu^N b$;
	 	\item $\supp \mathcal L_\nu^{k}b\subset B, \ k=0, 1, \dots, M$;
	 	\item $\|\mathcal L_\nu^{k}b\|_{L^\vc(\mathbb{R}^n_+)}\leq
	 	r_B^{2(N-k)}|B|^{-\f{1}{p}},\ k=0,1,\dots,N$.
	 \end{enumerate}
	 We first claim that for any multi-index $\alpha$ with $|\alpha|<N$, we have 
	 \begin{equation}\label{eq-integral of a}
	 	\Big|\int (x-x_B)^\alpha a(x)\,dx\Big| \le |B|^{1-\f{1}{p}} r_B^{|\alpha|} \Big(\f{r_B}{\rho_B}\Big)^{N-|\alpha|}.
	 	\end{equation} 
	 Indeed, from (i) we have
	 \[
	 \begin{aligned}
	 	\Big|\int (x-x_B)^\alpha a(x)\,dx\Big|&=\Big|\int_B (x-x_B)^\alpha \mathcal L_\nu^N b (x)\,dx\Big|\\
	 	&=\Big|\int_B \mathcal L_\nu^N(x-x_B)^\alpha  b (x)\,dx\Big|\\
	 	&\le  \int_B |\mathcal L_\nu^N(x-x_B)^\alpha| |b (x)|\,dx.
	 \end{aligned}
	 \]
	 For $x\in B$, using Lemma \ref{lem-critical function} it can be verified that 
	 \[
	 \begin{aligned}
	 	|\mathcal L_\nu^N(x-x_B)^\alpha|&\lesi  \sum_{j=0}^{|\alpha|}r_B^{|\alpha| -j} \f{1}{\rho(x_B)^{2N-j}}\\
	 	&\lesi \f{1}{\rho(x_B)^{2N-|\alpha|}}.
	\end{aligned}
	 \]
	  Hence,
	  \[
	  \begin{aligned}
	  	\Big|\int (x-x_B)^\alpha a(x)\,dx\Big|&\lesi \f{1}{\rho(x_B)^{2N-|\alpha|}}\|b\|_{1}\\
	  	&\lesi |B|^{1-1/p} \f{r_B^{2N}}{\rho(x_B)^{2N-|\alpha|}} = |B|^{1-1/p} r_B^{|\alpha|} \Big(\f{r_B}{\rho_B}\Big)^{2N-|\alpha|}. 
	  \end{aligned}
	  \]
This confirms \eqref{eq-integral of a}.

We now  prove that	$a\in H^p_{\rho_\nu}(\Rn_+)$ as long as $r_B<\rho_\nu(x_B)$. Recall that $S_0(B)=B, S_j(B)=2^{j}B\setminus 2^{j-1}B, j\ge 1$. Set $\omega = \lfloor n(1/p-1)\rfloor$. Let $\mathcal{V}_j$ be the span of the polynomials $\big\{(x-x_B)^\alpha\big\}_{|\alpha|\le \om }$ on $S_j(B)$ corresponding to the inner product space given by
	 $$ \ip{f,g}_{\mathcal{V}_j}:=\fint_{S_j(B)} f(x)g(x)\,dx.$$
	 Let $\{u_{j,\alpha}\}_{|\alpha|\le \om }$ be an orthonormal basis for $\mathcal{V}_j$ obtained via the Gram--Schmidt process applied to $\big\{(x-x_B)^\alpha\big\}_{|\alpha|\le \om }$
	 which, through homogeneity and uniqueness of the process, gives
	 \begin{align}\label{eq:mol1}
	 	u_{j,\alpha}(x)=\sum_{|\beta|\le \om } \lambda_{\alpha,\beta}^j (x-x_B)^\beta,
	 \end{align}
	 where for each $|\alpha|, |\beta|\le \om $ we have
	 \begin{align}\label{eq:mol2}
	 	|u_{j,\alpha}(x)|\le C\qquad\text{and}\qquad |\lambda_{\alpha,\beta}^j |\lesi (2^j r_B)^{-|\beta|}. 
	 \end{align}
	 Let $\{v_{j,\alpha}\}_{|\alpha|\le \om }$ be the dual basis of $\big\{(x-x_B)^\alpha\big\}_{|\alpha|\le \om }$ in $\mathcal{V}_j$; 
	 that is, it is the unique collection of polynomials such that
	 \begin{align}\label{eq:mol3} \ip{v_{j,\alpha}, (\cdot-x_B)^\beta}_{\mathcal V_j}=\delta_{\alpha,\beta}, \qquad  |\alpha|, |\beta|\le \om . \end{align}
	 Then we have
	 \begin{align}\label{eq:mol5}
	 	\Vert v_{j,\alpha}\Vert_\infty \lesi (2^jr_B)^{-|\alpha|},\qquad \forall \; |\alpha|\le \om .
	 \end{align}
	 Now let $P:={\rm proj}_{\mathcal{V}_0} (a)$ be the orthogonal projection of $a$ onto $\mathcal{V}_0$. Then  we have
	 \begin{align}\label{eq:mol6}
	 	P = \sum_{|\alpha|\le \om } \ip{a, u_{0,\alpha}}_{\mathcal V_0} u_{0,\alpha} = \sum_{|\alpha|\le \om } \ip{a,(\cdot-x_B)^\alpha}_{\mathcal V_0} v_{0,\alpha}.
	 \end{align}
	 Let $j_0\in \mathbb N$ such that $2^{j_0}r_B\ge \rho_{\nu}(x_B)>2^{j_0-1}r_B$. Then we write
	 \[
	 \begin{aligned}
	 	P =& \sum_{|\alpha|\le \om } \ip{a,(\cdot-x_B)^\alpha}_{\mathcal V_0} v_{0,\alpha}\\
	 	&=\sum_{|\alpha|\le \om }\sum_{j=0}^{j_0-2} \ip{a,(\cdot-x_B)^\alpha} \Big[\f{v_{j,\alpha}}{|S_j(B)|}-\f{v_{j+1,\alpha}}{|S_{j+1}(B)|}\Big] + \ip{a,(\cdot-x_B)^\alpha} \f{v_{j_0-1,\alpha}}{|S_{j_0-1}(B)|}.
	 \end{aligned}
	 \]
	 Hence, we can decompose
	 \[
	 \begin{aligned}
	 	a &= (a-P) +\sum_{|\alpha|\le \om }\sum_{j=0}^{j_0-2} \ip{a,(\cdot-x_B)^\alpha} \Big[\f{v_{j,\alpha}}{|S_j(B)|}-\f{v_{j+1,\alpha}}{|S_{j+1}(B)|}\Big] + \sum_{|\alpha|\le \om}\ip{a,(\cdot-x_B)^\alpha}\f{v_{j_0-1,\alpha}}{|S_{j_0-1}(B)|}\\
	 	&=a_1 +\sum_{|\alpha|\le \om }\sum_{j=0}^{j_0-2}a_{2,j,\alpha} +\sum_{|\alpha|\le \om }a_{3,\alpha}.
	 \end{aligned}
	 \]
	 
	 Let us now outline the important properties of the functions in the above decomposition. 
	 For $a_1$ we observe that for all $|\alpha|\le \om $,
	 \begin{align}\label{eq:mol8}
	 	\supp a_1 \subset B, &&
	 	\int a_1(x)(x-x_B)^\alpha dx=0, &&
	 	\Vert a_1\Vert_{L^\vc}\lesi |B|^{-1/p}.
	 \end{align}
	 Note that the property 
	 \[
	 \int a_1(x)(x-x_B)^\alpha dx=0 \ \text{for all $|\alpha|\le \om$}
	 \]
	 implies that 
	 \[
	 \int a_1(x)x^\alpha dx=0 \ \text{for all $|\alpha|\le \om$}.
	 \]
Hence, in this case $a_1$ is a $(p,\rho_\nu)$ atom.

	 
	 Next, for $a_{2,j,\alpha}$, it is obvious that $\supp a_{2,j,\alpha} \subset 2^{j+1}B$. In addition, from \eqref{eq:mol3}, we have
	 \[
	 \int a_{2,j,\alpha}(x)(x-x_B)^\beta dx=0 \ \text{for all $|\beta|\le \om $},
	 \]
	 which implies
	 \[
	 \int a_{2,j,\alpha}(x)x^\beta dx=0 \ \text{for all $|\beta|\le \om $}.
	 \]
	 We now estimate the size of $a_{2,j,\alpha}$. Using \eqref{eq-integral of a}, we have
	 \[
	 \begin{aligned}
	 	\|a_{2,j,\alpha}\|_{\vc}&\lesi (2^jr_B)^{-|\alpha|} |2^jB|^{-1} \Big|\int a(x)(x-x_B)^\alpha dx\Big| \\
	 	&\lesi |2^jB|^{-1 -|\alpha|/n} |B|^{1-1/p}r_B^{|\alpha|}\Big(\f{r_B}{\rho_{\nu}(x_B)}\Big)^{2N-|\alpha|}\\
	 	&\lesi 2^{-j(2N+n-n/p)} |2^jB|^{-1/p}.
	 \end{aligned}
	 \]
	 This means that $a_{2,j,\alpha}$ is a multiple of a $(p,\rho_\nu)$ atom, which further implies that
	 \[
	 \Big\|\sum_{|\alpha|\le \om }\sum_{j=0}^{j_0-2}a_{2,j,\alpha}\Big\|_{H^p_{\rho_\nu}}(\Rn_+)\lesi 1,
	 \]
	 as long as $N>n(1/p-1)$.
	 
 Next, for $a_{3,\alpha}$ we first have $\supp a_{3,\alpha} \subset 2^{j_0}B$. Moreover,  using \eqref{eq-integral of a} again, we obtain
	 \[
	 \begin{aligned}
	 	\|a_{2,j,\alpha}\|_{\vc}&\lesi (2^{j_0}r_B)^{-|\alpha|} |2^{j_0}B|^{-1} \Big|\int a(x)(x-x_B)^\alpha dx\Big| \\
	 	&\lesi |2^{j_0}B|^{-1 -|\alpha|/n} |B|^{1-1/p}r_B^{|\alpha|}\Big(\f{r_B}{\rho_{\nu}(x_B))}\Big)^{2N-|\alpha|}\\
	 	&\lesi 2^{-j_0(2N+n-n/p)} |2^{j_0}B|^{-1/p}\\
	 	&\lesi |2^{j_0}B|^{-1/p},
	 \end{aligned}
	 \]
	 as long as $N>n(1/p-1)$.
	 
	 It follows that $a_{3,\alpha}$ is a $(p,\rho_\nu)$ atom and $\|a_{3,\alpha}\|_{H^p_\rho(\Rn_+)}\lesi 1$.
	 
	 This completes our proof.
\end{proof}

\bigskip

		
		

\bigskip

\begin{prop}\label{prop-equivalence L +2}
		Let $\nu\in [-1/2,\vc)^n$ and $\ell\in \mathbb N^n$. Let $\ell_\nu$ be defined by \eqref{eq-ell nu}. Then for $\f{n}{n+\gamma_\nu}<p\le 1$ we have
		\[
		H^p_{\mathcal L_{\nu+\ell_\nu}}(\mathbb R^n_+) = H^p_{\mathcal L_\nu}(\mathbb R^n_+).
		\]
\end{prop}
\begin{proof}
	The proposition is a direct consequence of Theorem \ref{mainthm2s} and the facts that $\rho_{\nu}=\rho_{\nu+\ell_\nu}$ and $\gamma_\nu\le \gamma_{\nu+\ell_\nu}$.
\end{proof}


	

\subsection{Campanato spaces associated to the Laguerre operator $\mathcal L_\nu$}

For the duality of the Hardy space, we introduce the Campanato space associated to the critical function $\rho_\nu$. Let $\mathcal P_M$ 
denote the set of all polynomials of degree at most $M$. 
For any $g \in L^1_{\rm loc} (\Rn)$ and any ball $B \subset \Rn$, we denote $P_B^M g$ the \textit{minimizing polynomial}
of $g$ on the ball $B$ with homogeneous degree  $\leq M$, which means that $P_B^M g$ is the unique polynomial $P \in \mathcal{P}_M$
such that,
\begin{align} \label{minipoly}
	\int_B [g(x)-P(x)]x^\alpha dx  =0 \quad \mbox{for every } |\alpha|\le M.
\end{align}
It is known that if $g$ is locally integrable, then $P^M_B g$ uniquely exists (see \cite{JTW}). Recall that for any polynomial $P$ of degree at most $M$, any ball $B$ and any multi-index $\alpha$, we have
\begin{equation}\label{eq- polynomial}
\sup_{x\in B}|\partial_x^\alpha P(x)|\lesi_{\alpha,M} \sup_{x\in B}|P(x)| \ \ \ \text{and} \ \ \ \sup_{x\in B}|P(x)|\lesi_M \fint_B |P(x)|dx.
\end{equation}
which, together with Lemma 4.1 in \cite{Lu}, implies that
\begin{equation}\label{eq1- PBM}
	\sup_{x\in B}|\partial_x^\alpha P_B^Mg(x)|\lesi \f{r_B^{-|\alpha|}}{|B|}\int_B|g(x)|dx.
\end{equation}

We define the local Campanato spaces associated to critical functions as follows.
\begin{defn}\label{defn-BMO}
	Let $\nu\in [-1/2,\vc)^n$. Let $\rho_\nu$ be the critical function as in \eqref{eq- critical function}. 
	Let $s \geq 0$, $1 \leq q \leq \infty$ and $M \in \mathbb{N}_0$.
	The local Campanato space $BMO^{s,q,M}_{\rho_\nu}(\Rn_+)$ associated to $\rho_\nu$ is defined to
	be the space of all locally $L^q$ functions $f$ on $\Rn_+$ such that
	\begin{equation*}
		\begin{split}
			\|f\|_{BMO^{s, q,M}_{\rho_\nu}(\Rn_+)}&:=\sup_{\substack{B: \  {\rm balls}\\ r_B < \rho_{\nu}(x_B)}}  \left[ \frac{1}{|B|^{s/n}}
			\left(\frac{1}{|B|}\int_B|f(x)-P^M_Bf(x)|^q dx \right)^{1/q} \right] \\
			& \quad\quad +\sup_{\substack{B: \ {\rm balls}\\ r_B \geq \rho_{\nu}(x_B)} } \left[ \frac{1}{|B|^{s/n}}
			\left(\frac{1}{|B|}\int_B|f(x)|^q dx \right)^{1/q} \right] <\infty,
		\end{split}
	\end{equation*}
	with the usual modification when $q =\infty$. Here, $x_B$ and $r_B$ denote the center and radius of the ball $B$, respectively.
\end{defn}
Note that for each ball $B\subset \Rn_+$ if $r_B\le \rho_\nu(x_B)$, then $B$ is identical to its extension $B_{\Rn}$ to $\Rn$ defined by $B_{\Rn}:=\{y\in \Rn: |y-x_B|<r_B\}$. Therefore, in the above definition the polynomial $P_B^Mf$ is well-defined. 

In  Proposition \ref{independent of q} below, we will show that the function spaces in the above definition are independent of 
the choices of $1\le r\le \vc, M\ge \lfloor s\rfloor$ for $s>0$ and with $1\le r<\vc, M\ge 0$ for $s=0$, where $\lfloor s\rfloor$ denotes the greatest integer 
less than or equal to $s$. Hence, we define the space $BMO^s_{\rho_\nu}(\Rn_+)$ as any space $BMO^{s,r,M}_{\rho_\nu}(\Rn_+)$ 
with $1\le r\le \vc, M\ge \lfloor s\rfloor$ for $s>0$ and with $1\le r<\vc, M\ge 0$ for $s=0$.

We have the following result regarding the duality of the Hardy space $H^p_{\mathcal L_\nu}(\mathbb R^n_+)$.
\begin{thm} \label{dual}  Let $\nu\in [-1/2,\vc)^n$ and $\gamma_\nu =  \nu_{\min}+1/2$, where $\nu_{\min}$ is given in \eqref{eq- nu min}. Let $\rho_\nu$ be defined
	as in \eqref{eq- critical function}. Then for  $p \in (\f{n}{n+\gamma_\nu},1]$, we have
	\begin{align*}
		\big(H^p_{\mathcal L_\nu} (\mathbb{R}^n_+)\big)^\ast = BMO^{s}_{\rho_{\nu}} (\mathbb{R}^n_+),
		\quad \mbox{where } s: = n (1/p-1).
	\end{align*}
\end{thm}

An equivalent norm
of the Campanato space is as follows (see Theorem 1 in \cite{JTW}):
For any locally $L^q$ function $u$,
\begin{align*}
	\|u\|_{BMO^{s, q,M}(\Rn)} &\sim \sup_{ B: {\rm balls} }   \left[\frac{1}{|B|^{s/n}}
	\left(\frac{1}{|B|}\int_B|u(x)-P_B^M u(x)|^q dx \right)^{1/q} \right].
\end{align*}

We first recall the following result concerning independence of $q$ for the spaces $BMO^{s,q,M}(\Rn)$.
See  Theorem 3 in \cite{JTW}.

\begin{prop} \label{coincidence of global campanato}
	\begin{enumerate}
		[\rm (a)]\item If $s > 0$, $1 \leq q ,r\leq \infty$ and $M_1, M_2 \in \mathbb{N}$ such that $M_1, M_2 \geq \lfloor s\rfloor$, then $BMO^{s,q,M_1}(\Rn) = BMO^{s,r,M_2}(\Rn)$ with equivalent norms;
		\item If $1 \leq q, r < \infty$ and $M_1, M_2 \in \mathbb{N}_0$, then $\mathcal{C}^{0,q,M_1}(\Rn) = \mathcal{C}^{0,r,M_2}(\Rn)$.
	\end{enumerate}
\end{prop}

In what follows, for a ball $B$ in $\Rn_+$ let $B_{\Rn}$ be the extension of $B$ to $\Rn$, i.e., $B_{\Rn}:=\{y\in \Rn: |x-x_B|<r_B\}$.

\begin{lem}\label{lem-new defn BMO}
	Let $\nu\in [-1/2,\vc)^n$. Let $\rho$ be the critical function as in \eqref{eq- critical function}. 
	Let $s \geq 0$, $1 \leq q \leq \infty$ and $M \in \mathbb{N}_0$. Then
	\[
	\|f\|_{BMO^{s, q,M}_{\rho_\nu}(\Rn_+)}\sim \|\tilde f\|_{BMO^{s, q,M}(\Rn)} +\sup_{\substack{B: \ {\rm balls}\\ r_B \geq \rho_{\nu}(x_B)} } \left[ \frac{1}{|B|^{s/n}}
	\left(\frac{1}{|B|}\int_B|f(x)|^q dx \right)^{1/q} \right],
	\]
	where $\tilde f$ is the zero extension of $f$ to $\Rn$.
\end{lem}
\begin{proof}
	It is straightforward that
	\[
	\|f\|_{BMO^{s, q,M}_{\rho_\nu}(\Rn_+)}\le \|\tilde f\|_{BMO^{s, q,M}(\Rn)} +\sup_{\substack{B: \ {\rm balls}\\ r_B \geq \rho_{\nu}(x_B)} } \left[ \frac{1}{|B|^{s/n}}
	\left(\frac{1}{|B|}\int_B|f(x)|^q dx \right)^{1/q} \right].
	\]
It remains to show that 
	\[
	\|\tilde f\|_{BMO^{s, q,M}(\Rn)} +\sup_{\substack{B: \ {\rm balls}\\ r_B \geq \rho_{\nu}(x_B)}}  \left[ \frac{1}{|B|^{s/n}}
	\left(\frac{1}{|B|}\int_B|f(x)|^q dx \right)^{1/q} \right]\lesi \|f\|_{BMO^{s, q,M}_{\rho_\nu}(\Rn_+)}.
	\]
	Recall that 
	\[
	\|\tilde f\|_{BMO^{s, q,M}(\Rn)}\sim  \sup_{ B_{\Rn}: \ {\rm balls \, in \, \Rn} } \left[ \frac{1}{|B_{\Rn}|^{s/n}}
	\left(\frac{1}{|B_{\Rn}|}\int_{B_{\Rn}}|\tilde f(x) -P^M_{B_{\Rn}}\tilde f(x)|^q dx \right)^{1/q} \right]
	\]
	Note that for any ball $B_{\Rn}$ in $\Rn$, we have $P^M_{B_{\Rn}}\tilde f =0$ if $B_{\Rn} \cap \Rn_+ = \emptyset$. 
	This, together with \eqref{eq1- PBM}, implies that
	\[
	\begin{aligned}
		\|\tilde f\|_{BMO^{s, q,M}(\Rn)}&\sim  \sup_{ B_{\Rn}\cap \Rn_+\ne \emptyset } \left[ \frac{1}{|B_{\Rn}|^{s/n}}
		\left(\frac{1}{|B_{\Rn}|}\int_{B_{\Rn}}|\tilde f(x) -P^M_{B_{\Rn}}\tilde f(x)|^q dx \right)^{1/q} \right]\\
		&\lesi \sup_{ B_{\Rn}\cap \Rn_+\ne \emptyset } \left[ \frac{1}{|B_{\Rn}|^{s/n}}
		\left(\frac{1}{|B_{\Rn}|}\int_{B_{\Rn}}|\tilde f(x)|^q dx \right)^{1/q} \right] \\
		&\lesi \sup_{ B_{\Rn}\cap \Rn_+\ne \emptyset } \left[ \frac{1}{|B_{\Rn}|^{s/n}}
\left(\frac{1}{|B_{\Rn}|}\int_{B_{\Rn}\cap \Rn_+}|f(x)|^q dx \right)^{1/q} \right] \\
		&\sim \sup_{\substack{B_{\Rn}\cap \Rn_+\ne \emptyset\\ x_B\in \Rn_+ }}  \left[ \frac{1}{|B_{\Rn}|^{s/n}}
		\left(\frac{1}{|B_{\Rn}|}\int_{B_{\Rn}\cap \Rn_+}|f(x)|^q dx \right)^{1/q} \right]+ \\ 
		& \sup_{\substack{B_{\Rn}\cap \Rn_+\ne \emptyset\\ x_B\notin \Rn_+ }}  \left[ \frac{1}{|B_{\Rn}|^{s/n}}
		\left(\frac{1}{|B_{\Rn}|}\int_{B_{\Rn}\cap \Rn_+}|f(x)|^q dx \right)^{1/q} \right]\\
\end{aligned}
	\]
It is obvious that 	
$$
\begin{aligned}
	\sup_{\substack{B_{\Rn}\cap \Rn_+\ne \emptyset\\ x_B\in \Rn_+ }} &\left[ \frac{1}{|B_{\Rn}|^{s/n}}
	\left(\frac{1}{|B_{\Rn}|}\int_{B_{\Rn}\cap \Rn_+}|f(x)|^q dx \right)^{1/q} \right]\\
	 &\lesi \sup_{\substack{B_{\Rn}\cap \Rn_+\ne \emptyset\\ x_B\in \Rn_+ }} \left[ \frac{1}{|B_{\Rn}\cap \Rn_+|^{s/n}}
	\left(\frac{1}{|B_{\Rn}\cap \Rn_+|}\int_{B_{\Rn}\cap \Rn_+}|f(x)|^q dx \right)^{1/q} \right] \\
	&\lesi \|f\|_{BMO^{s, q,M}_{\rho_\nu}(\Rn_+)}.
\end{aligned}
$$	
	
For all balls $B_{\Rn}$ such that $B_{\Rn}\cap \Rn_+\ne \emptyset$ and $ x_B\notin \Rn_+$, take a $x_0\in B_{\Rn}\cap \Rn_+\ne \emptyset$. 
Then we have $B_{\Rn}\subset B_{\Rn}(x_0,2r_B)$. Consequently, for such a ball $B_{\Rn}$, 
$$
\begin{aligned}
	&\left[ \frac{1}{|B_{\Rn}|^{s/n}}
	\left(\frac{1}{|B_{\Rn}|}\int_{B_{\Rn}\cap \Rn_+}|f(x)|^q dx \right)^{1/q} \right]\\
	&\lesi \ \left[ \frac{1}{|B_{\Rn}(x_0,2r_B)\cap \Rn_+|^{s/n}}
	\left(\frac{1}{|B_{\Rn}(x_0,2r_B)\cap \Rn_+|}\int_{B_{\Rn}(x_0,2r_B)\cap \Rn_+}|f(x)|^q dx \right)^{1/q} \right] \\
	&\lesi \|f\|_{BMO^{s, q,M}_{\rho_\nu}(\Rn_+)},
\end{aligned}
$$	
which implies that
\[
	\sup_{\substack{B_{\Rn}\cap \Rn_+\ne \emptyset\\ x_B\notin \Rn_+ }} \left[ \frac{1}{|B_{\Rn}|^{s/n}}
\left(\frac{1}{|B_{\Rn}|}\int_{B_{\Rn}\cap \Rn_+}|f(x)|^q dx \right)^{1/q} \right]\lesi \|f\|_{BMO^{s, q,M}_{\rho_\nu}(\Rn_+)}.
\]

	This completes our proof.
\end{proof}

\begin{prop} \label{independent of q}
	Let $\rho_\nu$ be the function in \eqref{eq- critical function}. Then we have:
	\begin{enumerate}
		[\rm (i)]\item If $s > 0$, $1 \leq q, r \leq \infty$ and $M_1, M_2 \in \mathbb{N}$ such that $M_1, M_2 \geq \lfloor s\rfloor$, then $BMO^{s,q,M_1}_{\rho_\nu}(\Rn_+) = BMO^{s,r,M_2}_{\rho_\nu}(\Rn_+)$;
		\item If $1 \leq q ,r < \infty$ and $M_1,M_2 \in \mathbb{N}_0$, then $\mathcal{C}^{0,q,M_1}_{\rho_\nu}(\Rn_+) = \mathcal{C}^{0,r,M_2}_{\rho_\nu}(\Rn_+)$.
	\end{enumerate}
\end{prop}
\begin{proof}
	We only give the proof of part (i) since the proof of (ii) is similar.
	
	To prove part (i) it suffices to prove that
	$BMO^{s,q,M}_{\rho_\nu}(\Rn_+) = BMO^{s,1,M}_{\rho_\nu }(\Rn_+)$
	for all $1 \leq q \leq \infty$. By Lemma \ref{lem-new defn BMO} and H\"older's inequality,  $BMO^{s,q,M_2}_{\rho_\nu}(\Rn_+) \subset BMO^{s,1,M_1}_{\rho_\nu }(\Rn_+)$. Hence it remains to prove that   $BMO^{s,1,M_1}_{\rho_\nu}(\Rn_+) \subset BMO^{s,q,M_2}_{\rho_\nu}(\Rn_+)$.

	To see the latter, let $f\in BMO^{s,1,M_1}_{\rho_\nu}(\Rn_+)$. Denote by $\tilde f$ the zero extension of $f$ to $\Rn$. Note that, by part (i) of
	Proposition \ref{coincidence of global campanato}, Lemma \ref{lem-new defn BMO}, \eqref{eq1- PBM} and the fact $|B|\sim |B_{\Rn}|$ for any ball $B$ in $\Rn_+$, we have
	\begin{align*}
		\|f\|_{BMO^{s, q,M_2}_{\rho_\nu}(\Rn_+)} &\sim \|\tilde f\|_{BMO^{s, q,M_2}(\Rn)}  +\sup_{\substack{B: {\rm balls\, in\, \Rn_+}\\ r_B \geq \rho_{\nu}(x_B)} } \left\{ \frac{1}{|B|^{s/n}}
		\left[\frac{1}{|B|}\int_B|f(x)|^q dx \right]^{1/q} \right\}\\
		& \leq \|\tilde f\|_{BMO^{s,1,M}(\Rn)} + \sup_{\substack{B: {\rm balls \ in\ } \Rn_+\\ r_B \geq \rho_{\nu}(x_B)} } \left\{ \frac{1}{|B|^{s/n}}
		\left[\frac{1}{|B|}\int_B|\tilde f(x) -P_{B_{\Rn}}^{M_2} \tilde f(x)|^q dx \right]^{1/q} \right\}\\
		&\quad + \sup_{\substack{B: {\rm balls \ in\ } \Rn_+\\ r_B \geq \rho_{\nu}(x_B)} } \left\{ \frac{1}{|B|^{s/n}}
		\left[\frac{1}{|B|}\int_B|P_{B_{\Rn}}^{M_2}\tilde  f(x)|^q dx \right]^{1/q} \right\}\\
		&\leq \|\tilde f\|_{BMO^{s,1,M}(\Rn)} + \|\tilde f\|_{BMO^{s,q,M}(\Rn)}  + \sup_{\substack{B: {\rm balls \ in\ } \Rn_+\\ r_B \geq \rho_{\nu}(x_B)} } \left\{ \frac{1}{|B_{\Rn}|^{s/n}}
		\left[\sup_{x \in B_{\Rn}}|P_{B_{\Rn}}^{M_2} \tilde f(x)| \right] \right\}\\
		&\lesi 2\|\tilde f\|_{BMO^{s,1,M_1}(\Rn)}  + \sup_{\substack{B: {\rm balls\, in\, \Rn_+}\\ r_B \geq \rho_{\nu}(x_B)} } \left\{ \frac{1}{|B_{\Rn}|^{s/n}}
		\left[\frac{1}{|B_{\Rn}|}\int_{B_{\Rn}}|\tilde f(x)| dx \right] \right\}\\
		&\sim 2\|\tilde f\|_{BMO^{s,1,M_1}(\Rn)}  + \sup_{\substack{B: {\rm balls\, in\, \Rn_+}\\ r_B \geq \rho_{\nu}(x_B)} } \left\{ \frac{1}{B}
		\left[\frac{1}{|B|}\int_{B}|f(x)| dx \right] \right\}\\
		& \lesi  \|f\|_{BMO^{s,1,M_1}_{\rho_\nu}(\Rn_+)}.
	\end{align*}
	This implies that $BMO^{s,1,M_1}_{\rho_\nu}(\Rn_+) \subset BMO^{s,q,M_2}_{\rho_\nu}(\Rn_+)$ and completes the proof.
\end{proof}

\textit{Due to Proposition \ref{independent of q}, we will define the space $BMO^s_{\rho_\nu}(\Rn_+)$ as any space
	 $BMO^{s,r,M}_{\rho_\nu}(\Rn_+)$ with $1\le r\le \vc, M\ge \lfloor s\rfloor$ for $s>0$ and with $1\le r<\vc, M\ge 0$ for $s=0$.}

In order to prove Theorem \ref{dual}, we need the following technical results. 

Since  $\mathcal L_\nu$ is a non-negative self-adjoint operator  satisfying the Gaussian upper bound (see Proposition \ref{prop- delta k pt d>2}), it is well-known that the kernel $K_{\cos(t\sqrt{\mathcal L_\nu})}(\cdot,\cdot)$ of $\cos(t\sqrt{\mathcal L_\nu})$ satisfies 
\begin{equation}\label{finitepropagation}
	{\rm supp}\,K_{\cos(t\sqrt{\mathcal L_\nu})}(\cdot,\cdot)\subset \{(x,y)\in \Rn_+\times \Rn_+:
	|x-y|\leq t\}.
\end{equation}
See for example \cite{Sikora}.

As a consequence of \cite[Lemma 3]{Sikora}, we have:
\begin{lem}\label{lem:finite propagation}
	Let $\nu\in [-1/2,\vc)^n$. Let $\varphi\in C^\vc_0(\mathbb{R})$ be an even function with {\rm supp}\,$\varphi\subset (-1, 1)$ and $\displaystyle \int \varphi =2\pi$. Denote by $\Phi$ the Fourier transform of $\varphi$.  Then for any $k\in \mathbb N$ the kernel $K_{(t^2\mathcal L_\nu)^k\Phi(t\sqrt{\mathcal L_\nu})}$ of $(t^2\mathcal L_\nu)^k\Phi(t\sqrt{\mathcal L_\nu})$ satisfies 
	\begin{equation}\label{eq1-lemPsiL}
		\displaystyle
		{\rm supp}\,K_{(t^2\mathcal L_\nu)^k\Phi(t\sqrt{\mathcal L_\nu})}\subset \{(x,y)\in \Rn_+\times \Rn_+:
		|x-y|\leq t\},
	\end{equation}
	and
	\begin{equation}\label{eq2-lemPsiL}
		|K_{(t^2\mathcal L_\nu)^k\Phi(t\sqrt{\mathcal L_\nu})}(x,y)|\lesi  \f{1}{t^n}
	\end{equation}
	for all $x,y \in \Rn_+$ and $t>0$.
\end{lem}

\begin{lem} \label{03}
	Let $s\ge 0$, $\nu\in [-1/2,\vc)^n$ and $\rho_\nu$ be as in \eqref{eq- critical function}. Let $\Phi$ be defined as in Lemma \ref{lem:finite propagation}. 
	Then for $k>\f{s}{2}$, there exists $C>0$ such that for all $f\in BMO^{s}_{\rho_\nu}(\mathbb{R}^n_+)$,
	\begin{equation}
		\label{eq1-Carleson}
		\sup_{B: \, {\rm balls}} \f{1}{|B|^{2s/n +1}} \int_0^{r_B}\int_B |(t^2\mathcal L_\nu)^k  \Phi(t\sqrt{\mathcal L_\nu})f(x)|^2\f{dxdt}{t}\le  C\|f\|_{BMO^{s}_{\rho_\nu}(\mathbb{R}^n_+)}.
	\end{equation}
\end{lem}
\begin{proof} The lemma was proved in \cite{B}  for the case $s\in (0,1)$. However, in this case the proof is simpler since we can replace $P_Mf$ by $\displaystyle f_B:=\f{1}{|B|}\int_B f$ in Definition \ref{defn-BMO}. 
	
Let $B$ be an arbitrary ball in $\mathbb{R}^n_+$. We consider two cases.
	
	\medskip
	\textbf{Case 1: $r_B\ge  \rho_\nu(x_B)/4$}. Due to Lemma \ref{lem:finite propagation},  
	\[
	\begin{aligned}
		\f{1}{|B|^{2s/n +1}} \int_0^{r_B}\int_B & |(t^2\mathcal L_\nu)^k  \Phi(t\sqrt{\mathcal L_\nu})f(x)|^2\f{dx dt}{t}\\
		&= \f{1}{|B|^{2s/n +1}} \int_0^{r_B}\int_{B} |(t^2\mathcal L_\nu)^k  \Phi(t\sqrt{\mathcal L_\nu})(f\chi_{4B})(x)|^2\f{dx dt}{t}.
	\end{aligned}
	\]
	Note that by the spectral theory, the Littlewood-Paley square function
	\[
	g \mapsto \left(\int_0^{\infty} |(t^2\mathcal L_\nu)^k\Phi(t\sqrt{\mathcal L_\nu})g|^2\f{dt}{t}\right)^{1/2}
	\]
	is bounded on $L^2(\mathbb{R}^n_+)$. Hence,
	\[
	\begin{aligned}
		\f{1}{|B|^{2s/n +1}} \int_0^{r_B}\int_{B} |(t^2\mathcal L_\nu)^k  \Phi(t\sqrt{\mathcal L_\nu})(f\chi_{3B})(x)|^2\f{dx dt}{t}&\lesi  \f{1}{|B|^{2s/n +1}} \|f\chi_{4B}\|_{2}^2\\
		&\lesi  \|f\|_{BMO^s_{\rho_\nu}(\mathbb{R}^n_+)}^2.
	\end{aligned}
	\]

	\medskip
	
	\textbf{Case 2: $r_B< \rho_\nu(x_B)/4$}.  In this case, we write
	\[
	\begin{aligned}
		\f{1}{|B|^{2s/n +1}} \int_0^{r_B}&\int_B |(t^2\mathcal L_\nu)^k \Phi(t\sqrt{\mathcal L_\nu}) f(x)|^2\f{dx dt}{t}\\
		&\le \f{1}{|B|^{2s/n +1}} \int_0^{r_B}\int_{B} |(t^2\mathcal L_\nu)^k  \Phi(t\sqrt{\mathcal L_\nu})(f - P^M_{4B}f)(x)|^2\f{dx dt}{t}\\
		&\ \ \ \ \ + \f{1}{|B|^{2s/n +1}} \int_0^{r_B}\int_{B} |(t^2\mathcal L_\nu)^k  \Phi(t\sqrt{\mathcal L_\nu})P^M_{4B}f(x)|^2\f{dx dt}{t}.
	\end{aligned}
	\]
	Similarly to Case 1, we have
	\[
	\begin{aligned}
		\f{1}{|B|^{2s/n +1}} \int_0^{r_B}\int_{B} &|(t^2\mathcal L_\nu)^k \Phi(t\sqrt{\mathcal L_\nu})(f - P^M_{4B}f)(x)|^2\f{dx dt}{t}\\
		&=\f{1}{|B|^{2s/n +1}} \int_0^{r_B}\int_{B} |(t^2\mathcal L_\nu)^k  \Phi(t\sqrt{\mathcal L_\nu})\big[(f - P^M_{4B}f)\chi_{4B}\big](x)|^2\f{dx dt}{t}\\
		&\lesi \f{1}{|B|^{2s/n +1}}\|(f - P^M_{4B}f)\chi_{4B}\|_{2}^2\\
		&\lesi  \|f\|_{BMO^s_{\rho_\nu}(\mathbb{R}^n_+)}^2.
	\end{aligned}
	\]
	Hence it remains to show that
	\begin{equation}\label{eq123}
		\f{1}{|B|^{2s/n +1}} \int_0^{r_B}\int_{B} |(t^2\mathcal L_\nu)^k  \Phi(t\sqrt{\mathcal L_\nu})P^M_{4B}f(x)|^2\f{dx dt}{t}\lesi   \|f\|_{BMO^s_{\rho_\nu}(\mathbb{R}^n_+)}^2.
	\end{equation}

By Lemma \ref{lem:finite propagation}, we have, for $x\in B$ and $t<r_B$,
	\begin{equation}\label{eq0-Carleson proof}
		\begin{aligned}
			|(t^2\mathcal L_\nu)^k\Phi(t\sqrt{\mathcal L_\nu})P^M_{4B}f(x)| &=\Big| \int_{2B} t^{2k}\varphi(t\sqrt \mathcal L_\nu)(x,y) \mathcal L_\nu^k(P^M_{4B}f)(y) dy\Big|\\
			&\lesi t^{2k} \sup_{z\in 4B}|\mathcal L_\nu^k(P^M_{4B}f)(z)|.
		\end{aligned}
	\end{equation}
		Setting $B_\rho = B(x_B,\rho_\nu(x_B))$, then we have
	\[
	\begin{aligned}
		\sup_{z\in 4B}|\mathcal L_\nu^k(P^M_{4B}f)(z)|
		&\lesi  \sup_{z\in 4B} |\mathcal L_\nu^k[P_{4B}^M f(z)-P_{B_\rho}^M f(z)]|+ r_B^{-2k}\sup_{z\in 4B} |\mathcal L_\nu^kP_{B_\rho}^M f(z)|\\
		&=: E_1 + E_2.
	\end{aligned}
	\]	
	For the term $E_2$, using the expression
	\[
	\mathcal L_\nu^k=\Biggl\{\sum_{i=1}^n \left[-\frac{\partial^2}{\partial x_i^2} + x_i^2 + \frac{1}{x_i^2}(\nu_i^2 - \frac{1}{4})\right]\Biggr\}^k,
	\]
	the fact $\rho_\nu(x)\lesi \min\{1/|x|, x_i: i\in \mathcal J_\nu\}$, \eqref{eq- polynomial} and Lemma \ref{lem-critical function} we have
\[
\begin{aligned}
	E_2 &\lesi  \sup_{z\in B_\rho} |\mathcal L_\nu^k P_{B_\rho}^M f(z)| \lesi \rho_\nu(x_B)^{-2k}\sup_{z\in B_\rho} | P_{B_\rho}^M f(z)|.
\end{aligned}
\]
This, together with  \eqref{eq1- PBM}, Lemma \ref{lem-new defn BMO} and Lemma \ref{independent of q}, further implies that
	\[
	\begin{aligned}
		E_2 
		&\lesi  \rho_\nu(x_B)^{-2k} \fint_{B_\rho}|P_{B_\rho}^M f(z)|d(z)\\
		&\lesi  \rho_\nu(x_B)^{-2k} \fint_{B_\rho}|f(z)-P_{B_\rho}^M f(z)|dz +\rho_\nu(x_B)^{-2k} \fint_{B_\rho}|f(z)|dz\\
		&\lesi  \rho_\nu(x_B)^{-2k} |B_\rho|^{s/n}\|f\|_{BMO_\rho^{s}(\Rn_+)}\\
		&\lesi  r_B^{s-2k}\|f\|_{BMO_\rho^{s}(\Rn_+)}.
	\end{aligned}
	\]
	To estimate $E_1$, let $k_0\in \mathbb N$ such that $2^{k_0}r_{B}<\rho_\nu(x_B)\le 2^{k_0+1}r_{B}$. Then we have
\[
\begin{aligned}
	E_1 &\lesi    \left[\sum_{j=2}^{k_0-1}\sup_{z\in 4B} |\mathcal L_\nu^k [P_{2^{j}B}^M f(z)-P_{2^{j+1}B}^M f(z)]| +\sup_{z\in 4B} |\mathcal L_\nu^k [P_{2^{k_0}B}^M f(z)-P_{B_\rho}^M f(z)]|\right]\\
	&\lesi  \left[\sum_{j=2}^{k_0-1}\sup_{z\in 2^jB} |\mathcal L_\nu^k [P_{2^{j}B}^M f(z)-P_{2^{j+1}B}^M f(z)]| +\sup_{z\in 2^{k_0}B} |\mathcal L_\nu^k [P_{2^{k_0}B}^M f(z)-P_{B_\rho}^M f(z)]|\right].
\end{aligned}
\]
	
	Using the expression
	\[
	\mathcal L_\nu^k=\Biggl\{\sum_{i=1}^n \left[-\frac{\partial^2}{\partial x_i^2} + x_i^2 + \frac{1}{x_i^2}(\nu_i^2 - \frac{1}{4})\right]\Biggr\}^k,
	\]
	the fact that $\rho_\nu(x)\lesi \min\{1/|x|, x_i: i\in \mathcal J_\nu\}$, \eqref{eq- polynomial}, Lemma \ref{lem-critical function} and the fact that $r_B\le \rho(x_B)/4$, we obtain
	\[
	\begin{aligned}
		E_1 
		&\lesi   \sum_{j=2}^{k_0-1}(2^jr_B)^{-2k}\sup_{z\in 2^jB} |P_{2^{j}B}^M f(z)-P_{2^{j+1}B}^M f(z)| +\rho(x_B)^{-2k}\sup_{z\in 2^{k_0}B} |P_{2^{k_0}B}^M f(z)-P_{B_\rho}^M f(z)| \\
		&\lesi \sum_{j=2}^{k_0-1}(2^jr_B)^{-2k}\fint_{2^jB} |P_{2^{j}B}^M f(z)-P_{2^{j+1}B}^M f(z)|dz +\rho(x_B)^{-2k}\fint_{2^{k_0}B} |P_{2^{k_0}B}^M f(z)-P_{B_\rho}^M f(z)|dz \\
		&=: E_{11}+E_{12}.
	\end{aligned}
	\]
For $E_{11}$, we have
	\[
	\begin{aligned}
		E_{11}	&\lesi \sum_{j=2}^{k_0-1}(2^jr_B)^{-2k}\Big[\fint_{2^{j}B}|f(z)- P_{2^{j}B}^M f(z)| dz+\fint_{2^{j}B}|f(z)- P_{2^{j+1}B}^M f(z)| dz\Big]\\
		&\lesi \sum_{j=2}^{k_0-1}(2^jr_B)^{-2k}|2^jB|^{s/n} \|f\|_{BMO_{\rho_\nu}^{s}(\Rn_+)}\\
		&\lesi r_B^{s-2k} \|f\|_{BMO_{\rho_\nu}^{s}(\Rn_+)},
	\end{aligned}
	\]
	as long as $k>s/2$.
	
	Similarly,
	\[
	E_{12}\lesi r_B^{s-2k} \|f\|_{BMO_{\rho_\nu}^{s}(\Rn_+)}.
	\]
		Collecting the estimates of $E_2$, $E_{11}$ and $E_{12}$,
	\[
	\begin{aligned}
		\sup_{z\in 4B}|\mathcal L_\nu^kP^M_{4B}f(z)| &\lesi   r_B^{s-2k} \|f\|_{BMO_{\rho_\nu}^{s}(\Rn_+)}.
	\end{aligned}
	\]
	
	This, along with \eqref{eq0-Carleson proof}, yields
	\begin{equation}
		\begin{aligned}
			&	\f{1}{|B|^{2s/n +1}} \int_0^{r_B}\int_B |(t^2\mathcal L_\nu)^k\varphi(t\sqrt{\mathcal L_\nu})f(x)|^2\f{dxdt}{t}\\
			&\quad \lesi \|f\|^2_{BMO_{\rho_\nu}^{s}(\Rn_+)} \f{r_B^{2s-4k}}{|B|^{2s/n +1}} \int_0^{r_B}\int_B t^{4k}  \f{dxdt}{t}\\
			&\quad \lesi  \|f\|^2_{BMO_{\rho_\nu}^{s}(\Rn_+)},
		\end{aligned}
	\end{equation}
	as along as $k> s/2$.

	This completes our proof.

\end{proof}

Arguing similarly to the proof of \cite[Lemma 4.9]{B}, we have:
\begin{lem} \label{inte}
	Let $s\ge 0$, $\nu\in [-1/2,\vc)^n$ and $\rho_\nu$ be as in \eqref{eq- critical function}.  If $f \in BMO^s_{\rho_\nu}(\mathbb{R}^n_+)$ and $\sigma > s$, then
	\begin{align} \label{integrable}
		\int_{\mathbb R^n_+}\frac{ |f(x)| }{ (1+ |x|) ^{n + \sigma}}dx <\infty.
	\end{align}
\end{lem}


\begin{lem} \label{07}
	Let $\nu\in [-1/2,\vc)^n$, $\rho_\nu$ be as in \eqref{eq- critical function} and $\Phi$ be as in Lemma \ref{lem:finite propagation}. Let  $\f{n}{n+\gamma_\nu}<p\le 1$ and $s=n(1/p-1)$, where $\gamma_\nu=\nu_{\min}+1/2$. 
	Then  for every $f \in BMO^{s}_{\rho_\nu} (\mathbb{R}^n_+)$ and every $(p,\rho_\nu)$-atom $a$,
	\begin{align} \label{abc}
		\int_{\mathbb R^n_+} f (x) {a(x)}dx  =  C(k)\int_{\mathbb R^n_+ \times (0,\infty)}(t^2\mathcal L_\nu )^k\Phi(t\sqrt{\mathcal L_\nu})f(x) { t^2\mathcal L_\nu e^{-t^2\mathcal L_\nu}a(x)} \frac{dx dt}{t},
	\end{align}
	where $C(k) = \Big[\displaystyle \int_0^\vc z^{k}\Phi(\sqrt z)e^{-z}  dz\Big]^{-1}$. 
\end{lem}
\begin{proof}
	\cite[Lemma 4.11]{B} proved the result for $e^{-t^2\mathcal L_\nu}$ instead of $\Phi(t\sqrt{\mathcal L_\nu})$. While the proof in this setting follows the spirit of \cite[Lemma 4.11]{B}, some major modifications are required. For the readers' convenience, we provide the details here.
	
	  Set
	\begin{align*}
		F(x,t) := (t^2\mathcal L_\nu )^k\Phi(t\sqrt{\mathcal L_\nu})f(x) \quad \mbox{and} \quad G(x,t):= t^2\mathcal L_\nu e^{-t^2\mathcal L_\nu}a(x).
	\end{align*}
	It follows  from Lemma \ref{03}, Theorem \ref{mainthm1-Hardy} and Theorem \ref{mainthm2s} that
	\begin{align*}
		\|\mathcal{C}_p (F)\|_{L^\infty(\mathbb R^n_+)} \leq C \|f\|_{BMO^s_{\rho_\nu}(\mathbb R^n_+)}
	\end{align*}
	and
	\begin{align*}
		\|\mathcal{A}(G)\|_{L^p(\mathbb R^n_+)} = \|S_L a\|_{L^p(\mathbb R^n_+)} \leq C.
	\end{align*}
	Hence by Lemma \ref{05} (iii) the integral
	\begin{align*}
		\int_{\mathbb R^n_+ \times (0,\infty)} F(x,t)G(x,y)\frac{dx dt}{t}
	\end{align*}
	is well-defined. 
	
	On the other hand, for each $t>0$,
	\begin{equation} \label{int1}
		\begin{split}
			\int_{\mathbb R^n_+ \times (0,\infty)} F(x,t)G(x,y)\frac{dx dt}{t} &= \int_{\mathbb R^n_+ \times (0,\infty)}f(x) {(t^2\mathcal L_\nu )^k\Phi(t\sqrt{\mathcal L_\nu}) t^2\mathcal L_\nu e^{-t^2\mathcal L_\nu}a(x)} \frac{dx dt}{t}\\
			&= \int_{\mathbb R^n_+ }f(x) \int_0^\vc {(t^2\mathcal L_\nu )^k\Phi(t\sqrt{\mathcal L_\nu}) t^2\mathcal L_\nu e^{-t^2\mathcal L_\nu}a(x)}\f{dt}{t} dx.
		\end{split}
	\end{equation}

We will show that 
	\begin{align} \label{jus2}
		\sup_{\substack{\varepsilon, N>0}}\left|\int_\varepsilon^N
		(t^2\mathcal L_\nu )^k\Phi(t\sqrt{\mathcal L_\nu}) t^2\mathcal L_\nu e^{-t^2\mathcal L_\nu}a(x)\f{dt}{t}\right| \leq c_{a} (1+|x|)^{-(n +\gamma_\nu)}.
	\end{align}
	
	To do this, we first prove that
		\begin{align} \label{jus1}
		\sup_{t>0}\left|(t^2\mathcal L_\nu )^k\Phi(t\sqrt{\mathcal L_\nu}) t^2\mathcal L_\nu e^{-t^2\mathcal L_\nu}a(x) \right| \leq c_{a} (1 + |x|)^{-(n+\gamma_\nu)}.
	\end{align}
		In what follows we denote by $K_{K(\mathcal L_\nu)}(x,y)$ the kernel of the functional calculus $F(\mathcal L_\nu)$ for a bounded function $F: [0,\vc)\to \mathbb C$. By Proposition \ref{prop- delta k pt d>2}, \eqref{eq1-lemPsiL} and \eqref{eq2-lemPsiL}, we have
	\begin{equation}\label{eq- kernel of pt Phi}
		\begin{aligned}
		\left|K_{(t^2\mathcal L_\nu)^k\Phi(t\sqrt{\mathcal L_\nu}) (t^2 \mathcal L_\nu)^k e^{-t^2 \mathcal L_\nu}}(x,y) \right|
		& = \left|\int_{\Rn_+} p_{t^2}^\nu(x,y)K_{(t^2\mathcal L_\nu)^{k+1}\Phi(t\sqrt{\mathcal L_\nu})}(y,z) dz \right| \\
		&  \le \int_{B(y,t)} \frac{C_k}{t^{2n}} \exp\left( -\frac{|x-z|^2}{ct^2} \right) \left(1+\frac{t}{\rho_\nu(x)}\right)^{-\gamma_\nu}dz \\
		& \le  \frac{c_k}{t^n} \exp\left( -\frac{|x-y|^2}{ct^2} \right)\left(1+\frac{t}{\rho_\nu(x)}\right)^{-\gamma_\nu}.
	\end{aligned}
\end{equation}

	Suppose $a$ is a $(p,\rho)$-atom associated to the ball $B:=B(x_B,r_B)$ with $r_B\le \rho_\nu(x_B)$. If $|x|\le 2(1+|x_B|)$, then
	\begin{equation}\label{small}
		\begin{split}
			\left|(t^2\mathcal L_\nu)^k\Phi(t\sqrt{\mathcal L_\nu}) (t^2 \mathcal L_\nu)^k e^{-t^2 \mathcal L_\nu}a(x) \right| &\leq C_{k} \|a\|_{\vc}
			\int_{\Rn_+}  \frac{1}{t^n} \exp\left( -\frac{|x-y|^2}{ct^2} \right) dy \\
			& \leq c_{k}  \|a\|_{\vc}\\
			& \leq c_{k,a}  (1+|x_B|)^{-(n+\gamma_\nu)}\\
			&\le c_{k,a} (1+|x|)^{-(n+\gamma_\nu)}.
		\end{split}
	\end{equation}
	If $|x|\ge 2(1+|x_B|)\ge 2r_B$, then $|x-z|\sim |x|$ for $z\in B$. This, together with \eqref{eq- kernel of pt Phi} and the fact $\rho_\nu(x)\le \f{1}{|x|}$, yields that for $|x|\ge 2(1+|x_B|)\ge 2r_B$,
	\begin{equation} \label{large}
		\begin{aligned}
			\left|(t^2\mathcal L_\nu)^k\Phi(t\sqrt{\mathcal L_\nu}) (t^2 \mathcal L_\nu)^k e^{-t^2 \mathcal L_\nu}a(x) \right| 	&\leq C_{k} \int_{B} \frac{1}{t^n} \exp\left( -\frac{|x|^2}{ct^2} \right)\left(\frac{t}{\rho_\nu(x)}\right)^{-\gamma_\nu} |a(y)| dy \\
			&\leq c_{k}  \frac{1}{t^n} \exp\left( -\frac{|x|^2}{ct^2} \right) \left(\frac{1}{t|x|}\right)^{\gamma_\nu} \|a\|_{1} \\
			&\leq c_{k,a}  |x|^{-(n+\gamma_\nu)}\sim C_{k,a}  (1+|x|)^{-(n+\gamma_\nu)}.
		\end{aligned}
	\end{equation}
	Combining \eqref{small} and \eqref{large} yields \eqref{jus1}.

	We now turn to prove \eqref{jus2}. To do this, we define, for $\lambda \in [0, \infty)$
	\begin{align*}
		\eta(\lambda) = \int_\lambda^\infty t^{2k}\Phi(t)t^{2}e^{-t^2}\frac{dt}{t}.
	\end{align*}
	By the Fundamental Theorem of Calculus, one can check that $\eta$ is smooth on $[0,\infty)$ and decays rapidly at infinity.
	Moreover,
	\begin{align*}
		\lim_{\lambda \rightarrow 0^+}  \eta^{(2\ell +1)} (\lambda) =0 \quad \text{for all $\ell \in \mathbb{N}_0$}.
	\end{align*}
	Hence $\eta$ can be extended to an even Schwartz function on $\mathbb{R}$.  We claim that the kernel of $\eta(t\sqrt{L})$ satisfies
	\begin{align} \label{claim}
		\left| K_{\eta(t\sqrt{L})}(x,y)\right| \leq  \frac{C_{N}}{t^Q} \left(1+\frac{|x-y|^2}{t} \right)^{-N} \left(1+\frac{t}{\rho_\nu(x)} \right)^{-\gamma_\nu}
	\end{align}
	for arbitrary $N$.
	Indeed, let $m >n/2+\gamma_\nu$ be sufficiently large (depending on $N$), using Proposition \ref{prop-heat kernel} and the formula
	\begin{align*}
		(I + t^2\mathcal L_\nu)^{-m} =\frac{1}{\Gamma(m)} \int_0^\infty e^{-ut^2\mathcal L_\nu}e^{-u}u^{m-1}du,
	\end{align*}
	one can show that
	\begin{align*}
		\left|K_{(I + t^2\mathcal L_\nu)^{-m}} (x,y) \right| \leq \frac{C_N}{t^n}\left(1 +\frac{|x-y|}{t} \right)^{-N} \left(1 + \frac{t}{\rho_\nu(x)} \right)^{-\gamma_\nu}
	\end{align*}
	for arbitrary $N>0$.

	It follows that
	\begin{align*}
		&\hspace{0.5cm}\left|K_{\eta(t\sqrt{L})}(x,y) \right| \\
		& \leq  \int_{\Rn_+}  \left|K_{(I + t^2\mathcal L_\nu)^{-m}} (x,z)\left| K_{(I + t^2\mathcal L_\nu)^{m}\eta(t\sqrt{L}) }(z,y)\right|dz \right|dz \\
		&\leq C_{a} \frac{1}{t^Q} \int_{\Rn_+} \left| K_{(I + t^2\mathcal L_\nu)^{m}\eta(t\sqrt{L}) }(x,z)\right|
		\left(1 +\frac{|x-z|}{t} \right)^{-N} \left(1 + \frac{t}{\rho_\nu(x)} \right)^{-\gamma_\nu}dz   \\
		& \leq C_{a} \frac{1}{t^n} \left(1 +\frac{|x-y|}{t} \right)^{-N} \left(1 +\frac{t}{\rho_\nu(x)}\right)^{-\gamma_\nu}
		\int_{\Rn_+} \left| K_{(I + t^2\mathcal L_\nu)^{m}\eta(t\sqrt{L}) }(z,y)\right|\left(1 +\frac{|y-z|}{t} \right)^{N} dz.
	\end{align*}
	A standard argument as in the proof of \cite[Theorem 3.4]{PK} shows that there exists a sufficiently large integer $m$ (depending on $N$) such that
	\begin{align*}
		\int_{\Rn_+} \left| K_{(I + t^2\mathcal L_\nu)^{m}\eta(t\sqrt{L}) }(z,y)\right|\left(1 +\frac{|y-z|}{t} \right)^{N} dz \leq  C.
	\end{align*}
	Hence
	\begin{align*}
		\left|K_{\eta(t\sqrt{L})}(x,y) \right|
		\lesi \frac{1}{t^n} \left(1 +\frac{|x-y|}{t} \right)^{-N} \left(1 +\frac{t}{\rho_\nu(x)}\right)^{-\gamma_\nu}
	\end{align*}
	which gives \eqref{claim}.
	
	With the estimate \eqref{claim} on hand, we can proceed as in the proof of \eqref{jus1} to obtain
	\begin{align*}
		\sup_{t>0} |\eta(t\sqrt{L})a(x)| \leq C_a (1 + |x|)^{-(n +\gamma_\nu)}.
	\end{align*}
	Consequently,
	\begin{align*}
		\sup_{\substack{\varepsilon, N>0}}\left|\int_\varepsilon^N
		(t^2\mathcal L_\nu)^k\Phi(t\sqrt{\mathcal L_\nu}) (t^2 \mathcal L_\nu)^k e^{-t^2 \mathcal L_\nu}a(x)\right|
		&   = \sup_{\substack{\varepsilon, N>0}}\left|\eta(\varepsilon \sqrt{L})a(x)-\eta(N\sqrt{L})a(x) \right | \\
		& \leq   \sup_{\varepsilon>0}\left| \eta(\varepsilon \sqrt{L})a(x)\right | + \sup_{N>0}\left|\eta(N\sqrt{L})a(x)\right |\\
		&\leq  C_a (1 + |x|)^{-(n+\gamma_\nu)}.
	\end{align*}
	This verifies \eqref{jus2}.

This, together with Lemma \ref{inte}, implies that 
\[
\int_{\mathbb R^n_+ }\Big|f(x) \int_0^\vc {(t^2\mathcal L_\nu )^k\Phi(t\sqrt{\mathcal L_\nu}) t^2\mathcal L_\nu e^{-t^2\mathcal L_\nu}a(x)}\f{dt}{t}\Big| dx<\vc.
\]
On the other hand, by the spectral theory we have
\begin{align*}
	 \int_0^\vc
	(t^2\mathcal L_\nu )^k\Phi(t\sqrt{\mathcal L_\nu}) t^2\mathcal L_\nu e^{-t^2\mathcal L_\nu}a  \frac{dt}{t} = C( k) a \quad \mbox{in } L^2(\mathbb{R}^n_+).
\end{align*}

Therefore,  \eqref{abc} follows from \eqref{int1}.

This completes our proof.
	
\end{proof}

We are ready to give the proof of Theorem \ref{dual}.

\begin{proof} [Proof of  Theorem \ref{dual}:] Fix $\f{n}{n+\gamma_\nu}<p\le 1$, $s=n(1/p-1)$ and  $M\in \mathbb N$ such that $M\ge \lfloor s \rfloor$. Note that the proof for $\f{n}{n+1}<p\le 1$ was given in \cite[Theorem 1.5]{B}. In our case, $p$ might take all values between $0$ and $1$ and hence we need some modifications.

Firstly, let $f \in BMO^s_{\rho_\nu}(\mathbb{R}^n_+)$ and let $a$ be a  $(p, \rho_\nu)$-atoms. Then
	by Lemma \ref{07}, Lemma \ref{05} and Lemma \ref{03}, we have
	\begin{align*}
		\left|\int_{\Rn_+} f(x) {a(x)}dx  \right|
		& =  \left| \int_{\Rn_+ \times (0,\infty)}(t^2\mathcal L_\nu)^k\Phi(t\sqrt{\mathcal L_\nu}) f(x) { t^2\mathcal L_\nu e^{-t^2\mathcal L_\nu}a(x)} \frac{dx dt}{t} \right| \\
		&\leq \left( \f{1}{|B|^{2s/n +1}} \int_0^{r_B}\int_B |(t^2\mathcal L_\nu)^k\Phi(t\sqrt{\mathcal L_\nu})f(x)|^2\f{dx dt}{t}\right) \\
		&\qquad \times \left\| \left( \iint_{\Gamma(x)}|t^2\mathcal L_\nu e^{-t^2\mathcal L_\nu}a(y)|^2 \frac{dy dt}{t^{n+1}}\right)^{1/2}\right\|_{L^p(\mathbb{R}^n_+)} \\
		&\lesssim \|f\|_{BMO^s_{\rho_\nu}(G)}.
	\end{align*}
	
	This proves $BMO^s_{\rho_\nu}(\mathbb{R}^n_+) \subset (H^p_{\mathcal L_\nu}(\mathbb{R}^n_+))^\ast$.

	\medskip 
	
	It remains to show that  $(H^p_{\mathcal L_\nu}(\mathbb{R}^n_+))^\ast \subset BMO^s_{\rho_\nu}(\mathbb{R}^n_+) $. To do this, let $\{\psi_\xi\}_{\xi \in \mathcal{I}}$ and $\{B(x_\xi,\rho_\nu(x_{\xi}))\}_{\xi \in \mathcal{I}}$ as in Corollary \ref{cor1}. 
	Set $B_\xi:= B(x_\xi, \rho_\nu(x_\xi))$ and we will claim that for any $f \in L^2(\mathbb{R}^n_+)$ and $\xi \in \mathcal{I}$, we have $\psi_\xi f \in H^p_{\mathcal L_\nu} (\Rn_+)$ and
	\begin{align} \label{eq:L2}
		\|\psi_{\xi}f\|_{H^p_{\mathcal L_\nu}(\mathbb{R}^n_+)} \leq C \left|B_\xi\right|^{\frac{1}{p}-\frac{1}{2}} \|f\|_{L^2(\mathbb{R}^n_+)}.
	\end{align}
	It suffices to prove that 
	\[
	\left\|\sup_{t>0}\big|e ^{-t^2 \mathcal L_\nu}(\psi_{\xi}f)\big| \right\|_{p}\lesssim  |B_\xi|^{\frac{1}{p}-\frac{1}{2}}\|f\|_{2}.
	\]
	Indeed,  by H\"{o}lder's inequality,
	\begin{equation} \label{eq:B}
		\begin{split}
			\left\|\sup_{t>0}\big|e ^{-t^2 \mathcal L_\nu}(\psi_{\xi}f)\big| \right\|_{L^p(4B_\xi)}& \leq |4B_\xi|^{\frac{1}{p}-\frac{1}{2}}
			\left\| \sup_{t>0}\big|e ^{-t^2 \mathcal L_\nu}(\psi_{\xi}f)\big| \right\|_{L^2(4B_\xi)}  \\
			& \lesssim  |B_\xi|^{\frac{1}{p}-\frac{1}{2}}\|f\|_{L^2(\mathbb{R}^n_+)}.
		\end{split}
	\end{equation}
	If $x \in \mathbb R^n_+\backslash 4B_\xi$, then  applying Proposition \ref{prop- delta k pt d>2}, Lemma \ref{lem-critical function} and H\"older's inequality, we get
	\begin{align*}
		\big|e ^{-t^2 \mathcal L_\nu}(\psi_{\xi}f)(x)\big|
		& \lesi \int_{B_\xi} \left( \frac{t}{\rho_\nu(y)}\right)^{-\gamma_\nu} \frac{1}{t^n} \exp\left( -\frac{|x-y|^2}{ct^2}\right)|f(y)|dy \\
		& \sim \int_{B_\xi} \left(  \frac{t}{\rho_\nu(x_\xi)}\right)^{-\gamma_\nu} \frac{1}{t^n} \exp\left( -\frac{|x-y|^2}{ct^2}\right)|f(y)|dy \\
		& \lesi \frac{\rho_\nu(x_\xi)^{\gamma_\nu}}{|x-x_\xi|^{n+\gamma_\nu}}\int_{B_\xi}  |f(y)|dy \\			
		&\lesi \frac{\rho_\nu(x_\xi)^{\gamma_\nu}}{|x-x_\xi|^{n+\gamma_\nu}} |B_\xi|^{\frac{1}{2}} \|f\|_{L^2(\mathbb{R}^n_+)},
	\end{align*}
	which implies that
	\begin{equation} \label{eq:BC}
		\begin{split}
			\left\|\sup_{t>0}\big|e ^{-t^2 \mathcal L_\nu}(\psi_{\xi}f)\big| \right\|_{L^p(\Rn_+ \backslash 4B_\xi)} \lesi  |B_\xi|^{\frac{1}{p}-\frac{1}{2}}\|f\|_{L^2(\mathbb{R}^n_+)},
		\end{split}
	\end{equation}
	as long as $\f{n}{n+\gamma_\nu}<p\le 1$.

	Combining \eqref{eq:B} and \eqref{eq:BC} yields \eqref{eq:L2}.
	
	Assume that $\ell \in (H^p_{\mathcal L_\nu}(\mathbb{R}^n_+))^\ast$.
	For each index $\xi \in \mathcal{I}$ we define
	\begin{align*}
		\ell_{\xi} f := \ell(\psi_{\xi} f), \quad f \in L^2(\mathbb{R}^n_+).
	\end{align*}
	By \eqref{eq:L2},
	\begin{align*}
		|\ell_{\xi}(f)| \leq C \|\psi_{\xi}f\|_{H^p_{\mathcal L_\nu}(\mathbb{R}^n_+)} \leq C|B_\xi|^{\frac{1}{p}-\frac{1}{2}}\|f\|_{L^2(\mathbb{R}^n_+)}.
	\end{align*}
	Hence there exists $g_{\xi} \in L^{2}(B_\xi)$ such that
	\begin{align*}
		\ell_{\xi}(f) = \int_{B_\xi} f(x)g_{\xi}(x)dx ,
		\quad f \in L^2 (\mathbb{R}^n_+).
	\end{align*}

	We define $g = \sum_{\xi \in \mathcal{I}}1_{B_{\xi}} g_{\xi}$. Then, if $f =\sum_{i=1}^k \lambda_j a_j$, where $k \in \mathbb{N}$,
	$\lambda_i \in \mathbb{C}$, and $a_i$ is a $(p,\rho_\nu)$-atom, $i =1, \cdots, k$, we have
	\begin{align*}
		\ell(f) = \sum_{i=1}^k \lambda_i  \ell (a_i)
		& =  \sum_{i=1}^k \lambda_i \sum_{\xi \in \mathcal{I}}\ell_\xi (a_i) \\
		&= \sum_{i=1}^k \lambda_i \sum_{\xi \in \mathcal{I}} \int_{B_\xi} a_i(x)g_\xi (x)dx \\
		&= \int_{{\Rn_+}} f(x)g(x)dx .
	\end{align*}

	Suppose that $B = B(x_B, r_B) \in \mathbb R^n_+$ with $r_B < \rho_\nu(x_B)$, and $0 \not\equiv f \in L_0^2(B)$, that is,
	$f \in L^2(\mathbb{R}^n_+)$ such that $\supp f \subset B$ and $\displaystyle \int_B f(x)x^\alpha dx  =0$ for all $\alpha$ with $|\alpha|\le M$. Arguing similarly to the proof of \eqref{eq:L2},
	\[
	\|f\|_{H^p_{\mathcal L_\nu}(\mathbb{R}^n_+) } \lesi \|f\|_{L^2} |B|^{1/p -1/2}.
	\]
	 Hence
	\begin{align*}
		|\ell(f)| = \left| \int_B fg\right| \leq \|\ell\|_{(H^p_{\mathcal L_\nu}(\mathbb{R}^n_+))^\ast} \|f\|_{H^p_{\mathcal L_\nu}(\mathbb{R}^n_+) } \leq C \|\ell\|_{(H^p_{\mathcal L_\nu}(\mathbb{R}^n_+))^\ast} \|f\|_{L^2} |B|^{1/p -1/2}.
	\end{align*}
	From this we conclude that $g \in (L_0^2(B))^\ast$ and
	\[
	\|g\|_{(L_0^2(B))^\ast} \leq C \|\ell\|_{(H^p_{\mathcal L_\nu}(\mathbb{R}^n_+))^\ast}|B|^{1/p -1/2}.
	\]
Using the fact that (see Folland and Stein \cite[p. 145]{FS})
	\[
	\|g\|_{(L_0^2(B))^\ast}=\inf_{P\in \mathcal P_M} \|g-P\|_{L^{2}(B)},
	\]
we obtain	
	\begin{align} \label{31}
		\sup_{\substack{B: ball \\ r_B <\rho_\nu(x_B)}}|B|^{1/2-1/p} \inf_{P\in \mathcal P_M} \|g-P\|_{L^{2}(B)} \leq C \|\ell\|_{(H^p_{\mathcal L_\nu}(\mathbb{R}^n_+))^\ast} .
	\end{align}
	Moreover, if $B$ is a ball with $r_B \ge \rho_\nu(x_B)$, and $f \in L^2(\mathbb{R}^n_+)$ such that $f \not\equiv 0$
	and $\supp f \subset B$, then similarly to \eqref{eq:L2},
	\[
	\|f\|_{H^p_{\mathcal L_\nu}(\mathbb{R}^n_+)}\lesi |B|^{1/2 -1/p} \|f\|_{L^2}^{-1}.
	\] 
	Hence,
	\begin{align*}
		|\ell(f)| = \Big|\int_{\Rn_+} fg \Big| \leq C\|\ell\|_{(H^p_{\mathcal L_\nu}(\mathbb{R}^n_+))^\ast } \|f\|_{H^p_{\mathcal L_\nu}(\mathbb{R}^n_+)}
		\leq   C\|\ell\|_{(H^p_{\mathcal L_\nu}(\mathbb{R}^n_+))^\ast } |B|^{1/p-1/2} \|f\|_{L^2}.
	\end{align*}
	Hence
	\begin{align} \label{32}
		\sup_{\substack{B: {\rm ball} \\ r_B \ge \rho_\nu(x_B)}}|B|^{1/2 -1/p} \|g\|_{L^{2}(B)} \leq C \|\ell\|_{(H^p_{\mathcal L_\nu}(\mathbb{R}^n_+))^\ast} .
	\end{align}
	From \eqref{31} and \eqref{32} it follows that $g \in BMO^s_{\rho_{\nu}} (\mathbb{R}^n_+)$ and
	\begin{align*}
		\|g\|_{BMO^s_{\rho_{\nu}} (\mathbb{R}^n_+)} \leq C  \|\ell\|_{H^p_{\mathcal L_\nu}(\mathbb{R}^n_+))^\ast}, \quad \mbox{where } s =n(1/p-1).
	\end{align*}

	The proof of Theorem \ref{dual} is thus complete.
\end{proof}

\section{Boundedness of Riesz transforms on Hardy spaces and Campanato spaces associated to $\mathcal L_\nu$}

In this section, we will study the boundedness of the higher-order Riesz transforms. We first show in Theorem \ref{thm-Riesz transform} below that the higher-order Riesz transforms are Calder\'on-Zygmund operators. Then, in Theorem \ref{thm- boundedness on Hardy and BMO} we will show that the  higher-order Riesz transforms are bounded on our new Hardy spaces and new BMO type spaces defined in Section 4.

\begin{thm}\label{thm-Riesz transform} Let $\nu\in [-1/2,\vc)^n$, $\gamma_\nu = \min\{1, \nu_{\min}+1/2\}$ and a multi-index $k=(k_1,\ldots, k_n)\in \mathbb N^n$. Denote by $\delta^k_\nu \mathcal L_\nu^{-|k|/2}(x,y)$ the associated kernel of $\delta^k_\nu \mathcal L_\nu^{-|k|/2}$. Then the Riesz transform $\delta^k_\nu \mathcal L_\nu^{-|k|/2}$ is a Calder\'on-Zygmund operator. That is,	$\delta^k_\nu \mathcal L_\nu^{-|k|/2}$ is bounded on $L^2(\Rn_+)$ nd its kernel satisfies the following estimates:
	\[
	|\delta^k_\nu \mathcal L_\nu^{-|k|/2}(x,y)|\lesi  \f{1}{|x-y|^n} , \ \ \ x\ne y
	\]
	and
	\[
	\begin{aligned}
		| \delta^k_\nu \mathcal L_\nu^{-|k|/2}(x,y)-\delta^k_\nu \mathcal L_\nu^{-|k|/2}(x,y')|&+| \delta^k_\nu \mathcal L_\nu^{-|k|/2}(y,x)-\delta^k_\nu \mathcal L_\nu^{-|k|/2}(y',x)|
		&\lesi \Big(\f{|y-y'|}{x-y}\Big)^{\gamma_\nu}\f{1}{|x-y|^n},
	\end{aligned}
	\]
	whenever $|y-y'|\le |x-y|/2$.
\end{thm}
We would like to emphasize that Theorem \ref{thm-Riesz transform} is new, even when $n=1$. In fact, in the case of $n=1$,   
the boundedness of the Riesz transform was explored in \cite{NS, Betancor, Betancor2}. In \cite{NS}, the Riesz operator was 
decomposed into local and global components. The local part emerged as a local Calder\'on-Zygmund operator, while the global part was 
effectively dealt with by using weighted Hardy’s inequalities. In \cite{Betancor, Betancor2}, the investigation leaned on the connections 
between the Riesz transform associated with the Laguerre operator and that associated with the Hermite operator on $\mathbb{R}$. It's important to note that the methods employed in \cite{NS, Betancor, Betancor2} do not imply the boundedness of the Riesz transform on weighted $L^p_w(\mathbb{R}_+)$ with $1 < p < \infty$ and $w \in A_p(\mathbb{R}_+)$.

Returning to the Riesz transform $\delta_\nu \mathcal L^{-1/2}\nu$, we successfully establish its boundedness on the Hardy space $H^p_{\rho_\nu}(\mathbb R^n_+)$ and the Campanato spaces $BMO^{s}_{\rho_{\nu}} (\mathbb{R}^n_+)$, thus providing a comprehensive account of the boundedness properties of the Riesz transform.

\begin{thm}\label{thm- boundedness on Hardy and BMO}
	Let $\nu\in [-1/2, \vc)^n$, $\gamma_\nu = \nu_{\min} + 1/2$ and $k\in \mathbb N^n$. Then for  $\f{n}{n+\gamma_\nu}<p\le 1$ and $s=n(1/p-1)$, we have
	\begin{enumerate}[{\rm (i)}]
		\item the Riesz transform $\delta^k_\nu \mathcal L_\nu^{-|k|/2}$ is bounded on $H^p_{\rho_\nu}(\mathbb{R}^n_+)$;
		\item the Riesz transform $\delta^k_\nu \mathcal L_\nu^{-|k|/2}$ is bounded on $BMO^s_{\rho_\nu}(\mathbb{R}^n_+)$.
	\end{enumerate} 
\end{thm}
In the specific case where $n=1, p=1$, and $\nu>-1/2$, the boundedness of the Riesz transform on $BMO^s_{\rho_\nu}(\mathbb{R}_+)$ with $s=0$ was established in \cite{ChaLiu}. The approach employed in this proof leveraged the structure of the half-line $\mathbb{R}_+$ and exploited the connection between the Riesz transform and the Hermite operator. In addition, item (i) in Theorem \ref{thm- boundedness on Hardy and BMO} is new, even in the case of $n=1$.

 We first prove that $\delta_{\nu_j}^k \mathcal L_\nu^{-|k|/2}$ is bounded on $L^2(\Rn_+)$.
\begin{prop}
	 Let $\nu\in [-1/2,\vc)^n$ and $k\in \mathbb N^k$ be a multi-index.  Then the Riesz transform $\delta^k_\nu \mathcal L_\nu^{-|k|/2}$ is bounded on $L^2(\Rn_+)$.
\end{prop}
\begin{proof}
We will prove the proposition by induction.

$\bullet$ For $|k|=1$, the $L^2$-boundedness of the Riesz transform $\delta_\nu^k \mathcal L_\nu^{-|k|/1}$ is quite straightforward. Indeed, for each $j=1,\ldots, n$, by \eqref{eq- delta and eigenvector} we have
\[
\begin{aligned}
	\delta_{\nu_j}\mathcal{L}_{\nu}^{-1/2}f& = \sum_{\alpha\in \mathbb{N}^{n}} (4|\alpha| + 2|\nu| + 2n)^{-1/2} \langle f, \varphi_{\alpha}^\nu\rangle \delta_{\nu_j}varphi_{\alpha}^\nu\\
	& = \sum_{\alpha\in \mathbb{N}^{n}} \f{-2\sqrt{\alpha_j}}{\sqrt{4|\alpha| + 2|\nu| + 2n}} \langle f, \varphi_{\alpha}^\nu\rangle \varphi_{\alpha-e_j}^{\nu+e_j},
\end{aligned} 
\]
which implies that
\[
\|\delta_{\nu_j}\mathcal{L}_{\nu}^{-1/2}f\|_2^2\le \sum_k | \langle f, \varphi_{k}^\nu\rangle|^2 =\|f\|_2^2. 
\]

$\bullet$ Assume that the Riesz transform  $\delta^k \mathcal L_\nu^{-|k|/1}$ is bounded on $L^2(\Rn_+)$ for all $|k|\le \ell$ for some $\ell\ge 1$. We need to show that $\delta^k \mathcal L_\nu^{-|k|/1}$ is bounded on $L^2(\Rn_+)$ for all $|k|=\ell+1$. Fix a multi-index $k=(k_1,\ldots, k_n)$ with $|k|=\ell+1$. If all indices $k_j$ is less than or equal to $1$ for $j=1,\ldots, n$, the $L^2$-boundedness follows by using the similar argument to  that of the case $|k|=1$. Hence, without loss of generality, we might assume that $k_1\ge 2$.  Then we consider two cases.

\textbf{Case 1: $\nu_1>-1/2$}. From \eqref{eq- delta and eigenvector}, it can be verified that 
\[
\delta_{\nu}^{k} \mathcal L_\nu^{-(\ell+1)/2}=\delta_{\nu}^{k-e_1} (\mathcal L_{\nu+e_1}+2)^{-\ell/2}\delta_{\nu_1} \mathcal L_\nu^{-1/2}.
\]
Since $\delta_{\nu_1} \mathcal L_\nu^{-1/2}$ is bounded on $L^2(\Rn_+)$, it suffices to prove $\delta_{\nu}^{k-e_1} (\mathcal L_{\nu+e_1}+2)^{-\ell/2}$ is bounded on $L^2(\Rn_+)$. To do this, we write
\[
\delta_{\nu}^{k-e_1}\mathcal L_{\nu+e_j}^{-\ell/2}  \mathcal L_{\nu+e_j}^{\ell/2}(\mathcal L_{\nu+e_j}+2)^{-\ell/2}.
\]
By the spectral theory, $\mathcal L_{\nu+e_1}^{\ell/2}(\mathcal L_{\nu+e_1}+2)^{-\ell/2}$ is bounded on $L^2(\Rn_+)$. We need only to show that $\delta_{\nu}^{k-e_1}\mathcal L_{\nu+e_1}^{-\ell/2}$ is bounded on $L^2(\Rn_+)$.

Notice that by \eqref{eq- delta k}
\begin{equation*} 
	\delta_{\nu_1}^{k_1-1}=\Big(\delta_{{\nu_1}+1}+\f{1}{x_1}\Big)^{k_1-1} =\delta_{{\nu_1}+1}^{k_1-1} +\f{k_1-1}{x_1}\delta^{k_1-2}_{{\nu_1}+1},
\end{equation*}
which implies
\[
\delta_{\nu}^{k-e_1}\mathcal L_{\nu+e_1}^{-\ell/2}=\delta_{\nu +e_1}^{k-e_1}\mathcal L_{\nu+e_1}^{-\ell/2} +\f{k_1-1}{x_1}\delta_{\nu+e_1}^{k-2e_1}\mathcal L_{\nu+e_1}^{-\ell/2}.
\]
The first operator $\delta_{\nu+e_1}^{k-e_1}\mathcal L_{\nu+e_1}^{-\ell/2}$ is bounded on $L^2(\Rn_+)$ by the inductive hypothesis. Hence, it suffices to prove that the operator 
$$
f\mapsto \f{k_1-1}{x_1}\delta_{\nu+e_1}^{k-2e_1}\mathcal L_{\nu+e_1}^{-\ell/2}f
$$
is bounded on $L^2(\Rn_+)$.

By \eqref{eq- prod ptnu}, Proposition \ref{prop-delta k p x y} and the fact $|k-2e_1|=\ell-1$, we have, for a fixed $\epsilon \in (0,1)$,
\[
\begin{aligned}
	\delta_{\nu+e_1}^{k-2e_1}\mathcal L_{\nu+e_1}^{-\ell/2}(x,y) &= \int_0^\vc t^{\ell/2}\delta_{\nu+e_1}^{k-2e_1} p_t^{\nu+e_1}(x,y) \f{dt}{t}\\
	&\lesi \int_0^\vc  \f{1 }{t^{(n-1)/2}}\exp\Big(-\f{|x-y|^2}{ct}\Big)\Big(1+\f{\sqrt t}{\rho_{\nu_1+1}(x_1)}\Big)^{-(\nu_{1}+3/2)} \f{dt}{t}\\
	&\lesi  \f{1}{|x-y|^{n-1}} \min\{\Big(\f{|x-y|}{\rho_{\nu_1+1}(x_1)}\Big)^{-\epsilon}, \Big(\f{|x-y|}{\rho_{\nu_1+1}(x_1)}\Big)^{-(\nu_1+3/2)}\},
\end{aligned}
\]
which, together with the fact $\rho_{\nu_1+1}(x_1)\lesi x_1$, implies
\[
\begin{aligned}
	\Big|\f{k_1-1}{x_1}\delta_{\nu +e_1}^{k_1-1}\mathcal L_{\nu+e_1}^{-\ell/2}f(y)\Big|&\lesi \int_{|x-y|\le \rho_{\nu_1+1}(x_1)}  \f{1}{\rho_{\nu_1+1}(x_1)|x-y|^{n-1}} \Big(\f{|x-y|}{\rho_{\nu_1+1}(x_1)}\Big)^{-\epsilon} |f(y)|dy\\
	& \ \ \ \ +\int_{|x-y|> \rho_{\nu_1+1}(x_1)} \f{1}{\rho_{\nu_1+1}(x_1)|x-y|^{n-1}} \Big(\f{|x-y|}{\rho_{\nu_1+1}(x_1)}\Big)^{-(\nu_1+3/2)} |f(y)|dy  \\
	&=:I_1 +I_2.
\end{aligned}
\]
It is easy to see that 
\[
\begin{aligned}
	I_1&\lesi \int_{|x-y|\le   \rho_{\nu_1+1}(x_1)}\Big(\f{|x-y|}{\rho_{\nu_1+1}(x_1)}\Big)^{1-\epsilon}\f{|f(y)|}{|x-y|^{n}}dy\\
	&\lesi \mathcal Mf(x)
\end{aligned}
\]
and
\[
\begin{aligned}
	I_2&\lesi \int_{|x-y|> \rho_{\nu_1+1}(x_1)} \f{|f(y)|}{|x-y|^{n}}\Big(\f{\rho_{\nu_1+1}(x_1)}{|x-y|}\Big)^{\nu_1+1/2} dy\\
	&\lesi \mathcal Mf(x),
\end{aligned}
\]
where $\mathcal M$ is the Hardy-Littlewood maximal function.

It follows that the operator $\f{k_1-1}{x_1}\delta_{\nu+e_1}^{k_1-1}\mathcal L_{\nu+e_1}^{-|k|/2}$ is bounded on $L^2(\Rn_+)$, which completes the proof in the case $\nu_1>-1/2$.

\medskip

\textbf{Case 2: $\nu_1=-1/2$}

We first have
$$
\delta_{\nu_1}^2 = (-\delta_{\nu_1}^* +2x)\delta_{\nu_1} = (-\mathcal L_{\nu_1}+1)+2x_1 \delta_{\nu_1},
$$
which implies
\[
\delta_\nu^{k} \mathcal L_\nu^{-(\ell+1)/2} =-\delta_{\nu}^{k-2e_1}\mathcal L_{\nu_1}  \mathcal L_\nu^{-(\ell+1)/2} +\delta_{\nu}^{k-2e_1}   \mathcal L_\nu^{-(\ell+1)/2} +\delta_\nu^{k-2e_1} \big[x_1 \delta_{\nu_1}\mathcal L_\nu^{-(\ell+1)/2}\big].
\]
For the first operator we have
\[
\delta_{\nu}^{k-2e_1}\mathcal L_\nu^{-(\ell-1)/2}\circ  \mathcal L_\nu^{(\ell-1)/2}\mathcal L_{\nu_1}  \mathcal L_\nu^{-(\ell+1)/2}.
\]
The operator $\delta_{\nu}^{k-2e_1}\mathcal L_\nu^{-(\ell-1)/2}$ is bounded on $L^2(\Rn_+)$ due to $|k-2e_1|=\ell-1< \ell$ and the inductive hypothesis, meanwhile the operator $\mathcal L_\nu^{(\ell-1)/2}\mathcal L_{\nu_1}  \mathcal L_\nu^{-(\ell+1)/2}=\mathcal L_{\nu_1}\mathcal L_{\nu}^{-1}$ is bounded on $L^2(\Rn_+)$ by using the similar argument to that of the case $|k|=1$ above and \eqref{eq-eignevalue eigenvector}. Hence, $\delta_{\nu}^{k-2e_1}\mathcal L_{\nu_1}  \mathcal L_\nu^{-(\ell+1)/2}$ is bounded on $L^2(\Rn_+)$.

For the second operator, we have
 $$
 \delta_{\nu}^{k-2e_1}   \mathcal L_\nu^{-(\ell+1)/2} = \delta_{\nu}^{k-2e_1}   \mathcal L_\nu^{-(\ell-1)/2}\circ \mathcal L_\nu^{-1}.
 $$
The operator $\delta_{\nu}^{k-2e_1}   \mathcal L_\nu^{-(\ell-1)/2}$ is bounded on $L^2(\Rn_+)$ due to $|k-2e_1|=\ell-1< \ell$ and the inductive hypothesis; and the operator $\mathcal L_\nu^{-1}$ is bounded on $L^2(\Rn_+)$ due to the spectral theory and the spectral gap of $\mathcal L_\nu$. Consequently, $\delta_{\nu}^{k-2e_1}   \mathcal L_\nu^{-(\ell+1)/2}$ is bounded on $L^2(\Rn_+)$.

For the third operator, we have
\[
\delta_\nu^{k-2e_1} \big[x_1 \delta_{\nu_1}\mathcal L_\nu^{-(\ell+1)/2}\big]=x_1\delta^{k-e_1}_\nu\mathcal L_\nu^{-(\ell+1)/2} + (k-2) \delta^{k-2e_1}_\nu\mathcal L_\nu^{-(\ell+1)/2}.
\]
The  operator $x_1\delta^{k-e_1}_\nu\mathcal L_\nu^{-(\ell+1)/2}$ is bounded on $L^2$ by using the argument similarly to the case $\nu\ne -1/2$; meanwhile the operator $\delta^{k-2e_1}_\nu\mathcal L_\nu^{-(\ell+1)/2}$ is exactly the second operator above and is also bounded on $L^2(\Rn_+)$. Therefore, the operator $\delta_\nu^{k-2e_1} \big[x_1 \delta_{\nu_1}\mathcal L_\nu^{-(\ell+1)/2}\big]$ is bounded on $L^2(\Rn_+)$.

This completes our proof.
\end{proof}

 The following results establish the kernel bounds for the Riesz transforms.
\begin{prop}\label{Prop-Riesz transform} Let $\nu\in [-1/2,\vc)^n$ and $k\in \mathbb N^k$ be a multi-index.  Denote by $\delta^k_\nu \mathcal L_\nu^{-|k|/2}(x,y)$ the kernel of $\delta^k_\nu \mathcal L_\nu^{-|k|/2}$. Then we have, for $x\ne y$ and $j=1,\ldots,n$,
		\begin{equation}\label{eq- R kernel 1}
		|\delta^k_\nu \mathcal L_\nu^{-|k|/2}(x,y)|\lesi  \f{1}{|x-y|^n} \Big(1+\f{|x-y|}{\rho_\nu(x)}+\f{|x-y|}{\rho_\nu(y)}\Big)^{-(\nu_{\min}+1/2)},
		\end{equation}
		\begin{equation}\label{eq- R kernel 2}
		|\partial_{y_j}\delta^k_\nu \mathcal L_\nu^{-|k|/2}(x,y)|\lesi \Big(\f{1}{\rho_\nu(y)}+ \f{1}{|x-y|}\Big)\f{1}{|x-y|^n} \Big(1+\f{|x-y|}{\rho_\nu(x)}+\f{|x-y|}{\rho_\nu(y)}\Big)^{-(\nu_{\min}+1/2)}
		\end{equation}
		and
		\begin{equation}\label{eq- R kernel 3}
		|\partial_{x_j}\delta^k_\nu \mathcal L_\nu^{-|k|/2}(x,y)|\lesi \Big(\f{1}{\rho_\nu(x)}+ \f{1}{|x-y|}\Big)\f{1}{|x-y|^n} \Big(1+\f{|x-y|}{\rho_\nu(x)}+\f{|x-y|}{\rho_\nu(y)}\Big)^{-(\nu_{\min}+1/2)}.
		\end{equation}
\end{prop}
\begin{proof}
	By Proposition \ref{prop- delta k pt d>2}, we have
	\[
	\begin{aligned}
		\delta^k_\nu \mathcal L_\nu^{-|k|/2}(x,y) &= \int_0^\vc t^{|k|/2}\delta^k_\nu p_t^\nu(x,y) \f{dt}{t}\\
		&\lesi \int_0^\vc  \f{1 }{t^{n/2}}\exp\Big(-\f{|x-y|^2}{ct}\Big)\Big(1+\f{\sqrt t}{\rho_\nu(x)}+\f{\sqrt t}{\rho_\nu(y)}\Big)^{-(\nu_{\min}+1/2)} \f{dt}{t}\\
		&\lesi  \f{1}{|x-y|^n} \Big(1+\f{|x-y|}{\rho_\nu(x)}+\f{|x-y|}{\rho_\nu(y)}\Big)^{-(\nu_{\min}+1/2)}.
	\end{aligned}
	\]
	Similarly, by using Proposition \ref{prop-gradient x y d>2} we obtain the second and the third estimates.

	This completes our proof.
\end{proof}

\begin{prop}\label{Prop-Riesz transform 2} Let $\nu\in [-1/2,\vc)^n$ and $k\in \mathbb N^k$ be a multi-index.  Denote by $\delta^k_\nu \mathcal L_\nu^{-|k|/2}(x,y)$ the kernel of $\delta^k_\nu \mathcal L_\nu^{-|k|/2}$. Then we have, for $t>0$ and $x\ne y$ and $j=1,\ldots,n$,
	\begin{equation}\label{eq- R kernel 1 v2}
		|\delta^k_\nu e^{-t\mathcal L_\nu}\mathcal L_\nu^{-|k|/2}(x,y)|\lesi  \f{1}{|x-y|^n} \Big(1+\f{|x-y|}{\rho_\nu(x)}+\f{|x-y|}{\rho_\nu(y)}\Big)^{-(\nu_{\min}+1/2)}
	\end{equation}
	and
	\begin{equation}\label{eq- R kernel 2 v2}
		|\partial_{x_j}\delta^k_\nu e^{-t\mathcal L_\nu}\mathcal L_\nu^{-|k|/2}(x,y)|\lesi \Big(\f{1}{\rho_\nu(x)}+ \f{1}{|x-y|}\Big)\f{1}{|x-y|^n} \Big(1+\f{|x-y|}{\rho_\nu(x)}+\f{|x-y|}{\rho_\nu(y)}\Big)^{-(\nu_{\min}+1/2)}.
	\end{equation}
\end{prop}
\begin{proof}
	By Proposition \ref{prop- delta k pt d>2}, we have
	\[
	\begin{aligned}
		\delta^k_\nu e^{-t\mathcal L_\nu}\mathcal L_\nu^{-|k|/2}(x,y) &= \int_0^\vc  u^{|k|/2} \delta^k_\nu p_{t+u}^\nu(x,y) \f{du}{u}\\
		&\lesi \int_0^\vc  \f{u^{|k|/2}}{(t+u)^{|k|/2}} \f{1 }{(t+u)^{n/2}}\exp\Big(-\f{|x-y|^2}{c(t+u)}\Big)\Big(1+\f{\sqrt {t+u}}{\rho_\nu(x)}+\f{\sqrt {t+u}}{\rho_\nu(y)}\Big)^{-(\nu_{\min}+1/2)} \f{du}{u}\\
		&\lesi \int_0^t  \f{u^{|k|/2}}{t^{|k|/2}} \f{1 }{t^{n/2}}\exp\Big(-\f{|x-y|^2}{ct}\Big)\Big(1+\f{\sqrt {t}}{\rho_\nu(x)}+\f{\sqrt {t}}{\rho_\nu(y)}\Big)^{-(\nu_{\min}+1/2)} \f{du}{u}\\
		& \ \ \ + \int_t^\vc    \f{1 }{u^{n/2}}\exp\Big(-\f{|x-y|^2}{cu}\Big)\Big(1+\f{\sqrt {u}}{\rho_\nu(x)}+\f{\sqrt {u}}{\rho_\nu(y)}\Big)^{-(\nu_{\min}+1/2)} \f{du}{u}\\
		&\lesi  \f{1}{|x-y|^n} \Big(1+\f{|x-y|}{\rho_\nu(x)}+\f{|x-y|}{\rho_\nu(y)}\Big)^{-(\nu_{\min}+1/2)}.
	\end{aligned}
	\]
	Similarly, by using Proposition \ref{prop-gradient x y d>2} we obtain the second  estimates.

	This completes our proof.
\end{proof}

We are ready to give the proof of Theorem \ref{thm-Riesz transform}.
\begin{proof}[Proof of Theorem \ref{thm-Riesz transform}:]
	From \eqref{eq- R kernel 1}, we have
	\[
	| \delta^k_\nu \mathcal L_\nu^{-|k|/2}(x,y)|\lesi \f{1}{|x-y|^n}, \ \ \ x\ne y.
	\]
	We now prove the H\"older's continuity of the Riesz kernel. If $|y-y'|\ge \max\{\rho_\nu(y),\rho_\nu(y')\}$, then from \eqref{eq- R kernel 1} we have
	\[
	\begin{aligned}
		| \delta^k_\nu \mathcal L_\nu^{-|k|/2}(x,y)-\delta^k_\nu \mathcal L_\nu^{-|k|/2}(x,y')|&\lesi | \delta^k_\nu \mathcal L_\nu^{-|k|/2}(x,y)|+|\delta^k_\nu \mathcal L_\nu^{-|k|/2}(x,y')| \\
		&\lesi \f{1}{|x-y|^n}\Big[\Big(\f{\rho_\nu(y)}{|x-y|}\Big)^{\nu_{\min}+1/2}+\Big(\f{\rho_\nu(y')}{|x-y|}\Big)^{\nu_{\min}+1/2}\Big]\\
		&\lesi \f{1}{|x-y|^n} \Big(\f{|y-y'|}{|x-y|}\Big)^{\nu_{\min}+1/2}
	\end{aligned}
	\]
	
	If $|y-y'|\lesi \max\{\rho_\nu(y),\rho_\nu(y')\}$, then by Lemma \ref{lem-critical function} we have $\rho_\nu(y)\sim \rho_\nu(y')$; moreover, for every $z\in B(y, 2|y-y'|)$, $\rho_\nu(z)\sim \rho_\nu(y)$. This, together with  the mean value theorem and \eqref{eq- R kernel 2}, gives
	\[
	\begin{aligned}
		| \delta^k_\nu \mathcal L_\nu^{-|k|/2}(x,y)&-\delta^k_\nu \mathcal L_\nu^{-|k|/2}(x,y')|\\
		&\lesi  |y-y'|\Big[\sum_{j=1}^n\f{1}{\rho_\nu(y)}+ \f{1}{|x-y|}\Big]\Big(1+\f{|x-y|}{\rho_\nu(x)}+\f{|x-y|}{\rho_\nu(y)}\Big)^{-(\nu_{\min}+1/2)}\f{1}{|x-y|^n}\\
		&\lesi \Big[ \f{|y-y'|}{\rho_\nu(y)}+ \f{|y-y'|}{|x-y|}\Big]\Big(1+\f{|x-y|}{\rho_\nu(x)}+\f{|x-y|}{\rho_\nu(y)}\Big)^{-(\nu_{\min}+1/2)}\f{1}{|x-y|^n}\\
		&\lesi \f{|y-y'|}{\rho_\nu(y)} \Big(1+\f{|x-y|}{\rho_\nu(x)}+\f{|x-y|}{\rho_\nu(y)}\Big)^{-(\nu_{\min}+1/2)}\f{1}{|x-y|^n}+ \f{|y-y'|}{|x-y|}\f{1}{|x-y|^n}\\
		&=: E_1 +E_2.
	\end{aligned}
	\]
	Since $|y-y'|\lesi \rho_\nu(y)$, we have
	\[
	\begin{aligned}
		E_1&\lesi \Big(\f{|y-y'|}{\rho_\nu(y)}\Big)^{\gamma_\nu} \Big(1+\f{|x-y|}{\rho_\nu(x)}+\f{|x-y|}{\rho_\nu(y)}\Big)^{-\gamma_\nu}\f{1}{|x-y|^n}\\
		&\lesi \Big(\f{|y-y'|}{\rho_\nu(y)}\Big)^{\gamma_\nu} \Big(\f{|x-y|}{\rho_\nu(x)}\Big)^{-\gamma_\nu}\f{1}{|x-y|^n}\\
		&\lesi \Big(\f{|y-y'|}{|x-y|}\Big)^{\gamma_\nu}\f{1}{|x-y|^n},
	\end{aligned}
	\]
	where  and $\gamma_\nu = \min\{1, \nu_{\min}+1/2\}$.
	
	For the same reason, since $|y-y'|\le |x-y|/2$, we have
	\[
	\begin{aligned}
		E_2	&\lesi \Big(\f{|y-y'|}{|x-y|}\Big)^{\gamma_\nu}\f{1}{|x-y|^n}.
	\end{aligned}
	\] 
	It follows that 
	\[
	\begin{aligned}
		| \delta^k_\nu \mathcal L_\nu^{-|k|/2}(x,y)-\delta^k_\nu \mathcal L_\nu^{-|k|/2}(x,y')|\lesi \Big(\f{|y-y'|}{x-y}\Big)^{\gamma_\nu}\f{1}{|x-y|^n},
	\end{aligned}
	\]
	whenever $|y-y'|\le |x-y|/2$.
	
	Similarly, by using \eqref{eq- R kernel 3} we obtain
	\[
	\begin{aligned}
		| \delta^k_\nu \mathcal L_\nu^{-|k|/2}(y,x)-\delta^k_\nu \mathcal L_\nu^{-|k|/2}(y',x)|&\lesi \Big(\f{|y-y'|}{|x-y|}\Big)^{\gamma_\nu}\f{1}{|x-y|^n},
	\end{aligned}
	\]
	whenever $|y-y'|\le |x-y|/2$.
	
	This completes our proof.
\end{proof}

We now give the proof of Theorem \ref{thm- boundedness on Hardy and BMO}.

\begin{proof}[Proof of Theorem \ref{thm- boundedness on Hardy and BMO}:]
	Fix $\f{n}{n+\gamma_\nu}<p\le 1$, $k \in \mathbb N^n$ and $M>n(1/p-1)$. 
	
	\noindent (i)  By Proposition \ref{prop-equivalence L +2} and Theorem \ref{mainthm2s}, $H^p_{\mathcal L_{\nu+(k+\vec{M})_\nu}}(\mathbb{R}^n_+)\equiv H^p_{\mathcal L_\nu}(\mathbb{R}^n_+)\equiv H^p_{\rho_\nu}(\mathbb R^n_+)$, where $\vec{M}=(M,\ldots,M)\in \mathbb R^n$ and $(k+2\vec{M})_\nu$ is defined as in \eqref{eq-ell nu}. Consequently, it suffices to prove that 
	\[
	\Big\|\sup_{t>0}|e^{-t\mathcal L_{\nu + (k+2\vec{M})_\nu}}\delta^k \mathcal L_\nu^{-|k|/2}a|\Big\|_p\lesi 1
	\]
	for all $(p,M)_{\mathcal L_\nu}$ atoms $a$.
	
	Let $a$ be a $(p,M)_{\mathcal L_\nu}$ atom associated to a ball $B$. We have
	\[
	\begin{aligned}
		\Big\|\sup_{t>0}|e^{-t\mathcal L_{\nu + (k+2\vec{M})_\nu}}\delta^k \mathcal L_\nu^{-|k|/2}a|\Big\|_p&\lesi \Big\|\sup_{t>0}|e^{-t\mathcal L_{\nu + (k+2\vec{M})_\nu}}\delta^k \mathcal L_\nu^{-|k|/2}a|\Big\|_{L^p(4B)}\\ & \ \ \ \ \ +\Big\|\sup_{t>0}|e^{-t\mathcal L_{\nu + (k+2\vec{M})_\nu}}\delta^k \mathcal L_\nu^{-|k|/2}a|\Big\|_{L^p(\mathbb R^n_+\backslash 4B)}.		
	\end{aligned}
	\]
	Using the $L^2$-boundedness of both $f\mapsto \sup_{t>0}|e^{-t \mathcal L_{\nu + k +\vec M} }f|$ and the Riesz transform $\delta^k \mathcal L_\nu^{-|k|/2}$ and the H\"older  inequality, by the standard argument, we have
	\[
	\Big\|\sup_{t>0}|e^{-t\mathcal L_{\nu + (k+2\vec{M})_\nu}}\delta^k \mathcal L_\nu^{-|k|/2}a|\Big\|_{L^p(4B)}\lesi 1.
	\]
	For the second term, using $a=\mathcal L_\nu^Mb$,
	\[
	e^{-t\mathcal L_{\nu + (k+2\vec{M})_\nu}}\delta^k \mathcal L_\nu^{-|k|/2}a = e^{-t\mathcal L_{\nu + (k+2\vec{M})_\nu}}\delta^k \mathcal L_\nu^{M-|k|/2}b.
	\]
	We have
	\[
	\begin{aligned}
		e^{-t\mathcal L_{\nu + (k+2\vec{M})_\nu}}\delta^k \mathcal L_\nu^{M-|k|/2}(x,y) &= c\int_0^\vc u^{|k|/2}e^{-t\mathcal L_{\nu + (k+2\vec{M})_\nu}}\delta^k \mathcal L_\nu^{M}e^{-u\mathcal L_\nu}(x,y) \f{du}{u}\\
		&= c\int_0^t \ldots \f{du}{u} + c\int_t^\vc \ldots \f{du}{u}\\
		&=I_1+I_2.
	\end{aligned}
	\]
	For $I_1$, by Proposition \ref{prop-heat kernel}, Proposition \ref{prop-gradient x y dual delta} and Lemma \ref{lem- product two kernels} we have
	\[
	\begin{aligned}
		I_1&\lesi \int_0^t \int_{\Rn_+}u^{|k|/2} |[\mathcal L_\nu^{M}(\delta_\nu^*)^k e^{-t\mathcal L_{\nu + (k+2\vec{M})_\nu}}](z,x)| |p_u^\nu  (z,y)| dz \f{du}{u}\\
		&\lesi \int_0^t \f{u^{|k|/2}}{t^{M+|k|/2}} \f{1}{t^{n/2}}\exp\Big(-\f{|x-y|^2}{ct}\Big) \f{du}{u}\\
		&\sim \f{1}{t^{M}} \f{1}{t^{n/2}}\exp\Big(-\f{|x-y|^2}{ct}\Big)\\
		&\lesi \f{1}{|x-y|^{2M}} \f{1}{|x-y|^n}  
	\end{aligned}
	\]
	and
	\[
	\begin{aligned}
		I_2&\lesi  \int_t^\vc \int_{\Rn_+}u^{|k|/2} |p_u^{\nu+(k+2\vec{M})_\nu}  (x,z)| |[\delta^k \mathcal L_\nu^{M} e^{-t\mathcal L_{\nu}}](z,y)| dz \f{du}{u}\\
		&\lesi \int_t^\vc \f{1}{u^{M}} \f{1}{u^{n/2}}\exp\Big(-\f{|x-y|^2}{cu}\Big)  \f{du}{u}\\
		&\lesi \int_0^{|x-y|^2}\ldots + \int_{|x-y|^2}^\vc\ldots \\ 		
		&\lesi \f{1}{|x-y|^{2M}} \f{1}{|x-y|^n}.
	\end{aligned}
	\]
	
	Hence,
	\[
	\begin{aligned}
		\Big\|\sup_{t>0}|e^{-t\mathcal L_{\nu + (k+2\vec{M})_\nu}}\delta^k \mathcal L_\nu^{-|k|/2}a|\Big\|^p_{L^p(\mathbb R^n_+\backslash 4B)}
		& \lesi \int_{\mathbb R^n_+\backslash 4B} \Big[\int_{B} \f{1}{|x-y|^{2M}} \f{1}{|x-y|^n}  |b(y)|dy\Big]^{p}dx\\
		& \lesi \int_{\mathbb R^n_+\backslash 4B} \Big[\int_{B} \f{1}{|x-x_B|^{2M}} \f{1}{|x-x_B|^n} |b(y)|dy\Big]^{p}dx\\
		&\lesi \int_{\mathbb R^n_+\backslash 4B}  \f{\|b\|_1^p}{|x-x_B|^{(n+2M)p}}  dx\\
		&\lesi \int_{\mathbb R^n_+\backslash 4B}  \f{r_B^{2Mp}|B|^{p-1}}{|x-x_B|^{(n+2M)p}}  dx\\		
		&\lesi 1,
	\end{aligned}
	\]
	as long as $M>n(1/p-1)$.
\bigskip

\noindent (ii)  By the duality in Theorem \ref{dual}, it suffices to prove that the conjugate $\mathcal L_\nu^{-1/2}\delta_{\nu_j}^*$ 
is bounded on the Hardy space  $H^p_{\rho_\nu}(\mathbb R^n_+)$. By Theorem \ref{mainthm2s}, we need only to prove that 
	\[
\Big\|\sup_{t>0}|e^{-t\mathcal L_{\nu}} \mathcal L_\nu^{-|k|/2}(\delta_\nu^*)^ka|\Big\|_p\lesi 1
\]
for all $(p,\rho_\nu)$ atoms $a$.
	
Suppose that  $a$ is a $(p,\rho_\nu)$ atom associated to a ball $B$. Then we write
\[
\begin{aligned}
	\Big\|\sup_{t>0}|e^{-t\mathcal L_{\nu}} \mathcal L_\nu^{-|k|/2}(\delta_\nu^*)^ka|\Big\|_p&\lesi \Big\|\sup_{t>0}|e^{-t\mathcal L_{\nu}} \mathcal L_\nu^{-|k|/2}(\delta_\nu^*)^ka|\Big\|_{L^p(4B)} +\Big\|\sup_{t>0}|e^{-t\mathcal L_{\nu}}\mathcal L_\nu^{-|k|/2}(\delta_\nu^*)^ka|\Big\|_{L^p(\mathbb R^n_+\backslash 4B)}.		
\end{aligned}
\]
Since  $f\mapsto \sup_{t>0}|e^{-t \mathcal L_{\nu}}f|$ and  $\mathcal L_\nu^{-|k|/2}(\delta_\nu^*)^k$ is bounded on $L^2(\mathbb R^n_+)$, by  the H\"older's inequality and the standard argument, we have
\[
\Big\|\sup_{t>0}|e^{-t\mathcal L_{\nu}} \mathcal L_\nu^{-|k|/2}(\delta_\nu^*)^ka|\Big\|_{L^p(4B)}\lesi 1.
\]
For the second term, we consider two cases.

\textbf{Case 1: $r_B=\rho_\nu(x_B)$.} By Proposition \ref{Prop-Riesz transform 2} (a), for $x\in (4B)^c$,
\[
\begin{aligned}
	|\sup_{t>0}|e^{-t\mathcal L_{\nu}} \mathcal L_\nu^{-|k|/2}(\delta_\nu^*)^ka(x)|&\lesi \sup_{t>0}\int_B  |\delta ^k \mathcal L_\nu^{-|k|/2} e^{-t\mathcal L_{\nu}}(y,x)||a(y)|dy\\  
		&\lesi   \f{1}{|x-x_B|^n}\Big(\f{\rho_\nu(x_0)}{|x-x_B|}\Big)^{\gamma_\nu}\|a\|_1\\
	&\lesi   \f{r^{\gamma_\nu}}{|x-x_B|^{n+\gamma_\nu}} |B|^{1-1/p},
\end{aligned}
\]
where  we used the fact $\rho_\nu(y)\sim \rho_\nu(x_0)$ for $y\in B$ (due to Lemma \ref{lem-critical function}) in the second inequality.

It follows that
\[
\Big\|\sup_{t>0}|e^{-t\mathcal L_{\nu}} \mathcal L_\nu^{-|k|/2}(\delta_\nu^*)^ka|\Big\|_{L^p(\mathbb R^n_+\backslash 4B)}\lesi 1,
\]
as long as $\f{n}{n+\gamma_\nu}<p\le 1$.

\bigskip

\textbf{Case 2: $r_B<\rho_\nu(x_B)$.} 

Using the cancellation property $\int a(x) dx= 0$, we have
\[
\begin{aligned}
\sup_{t>0} |e^{-t\mathcal L_{\nu}} \mathcal L_\nu^{-|k|/2}(\delta_\nu^*)^ka(x)|&= \sup_{t>0}\Big|\int_{B}[e^{-t\mathcal L_{\nu}} \mathcal L_\nu^{-|k|/2}(\delta_\nu^*)^k(x,y)-e^{-t\mathcal L_{\nu}} \mathcal L_\nu^{-|k|/2}(\delta_\nu^*)^k(x,x_B)]a(y)dy\Big|.
\end{aligned}
\]
By mean value theorem, Proposition \ref{Prop-Riesz transform 2}  and the fact that $\rho_\nu(y_j)\gtrsim \rho_\nu(x_0)$ and $\rho_\nu(y)\sim \rho_\nu(x_B)$ for all $y\in B$, we have, for $x\in (4B)^c$,
\[
\begin{aligned}
|e^{-t\mathcal L_{\nu}} \mathcal L_\nu^{-|k|/2}(\delta_\nu^*)^ka(x)|&\lesi   \int_{B}\Big(\f{|y-x_B|}{|x-y|}+\f{|y-x_B|}{\rho_\nu(x_B)}\Big)\f{1}{|x-y|^n} \Big(1+\f{|x-y|}{\rho_\nu(x_B)}\Big)^{-\gamma_\nu} |a(y)|dy\\
&\lesi   \int_{B}\Big(\f{r_B}{|x-x_B|}+\f{r_B}{\rho_\nu(x_B)}\Big)\f{1}{|x-x_B|^n} \Big(1+\f{|x-x_B|}{\rho_\nu(x_B)}\Big)^{-\gamma_\nu} |a(y)|dy\\
&\lesi   \int_{B} \f{r_B}{|x-x_B|} \f{1}{|x-x_B|^n}  |a(y)|dy\\
&\ \ \ \ \ + \int_{B} \f{r_B}{\rho_\nu(x_B)} \f{1}{|x-x_B|^n} \Big(1+\f{|x-x_B|}{\rho_\nu(x_B)}\Big)^{-\gamma_\nu} |a(y)|dy\\
&= F_1 + F_2.
\end{aligned}
\]
For the term $F_1$, it is straightforward to see that 
\[
\begin{aligned}
F_1&\lesi \f{r_B}{|x-x_B|}\f{1}{|x-x_B|^n}\|a\|_1\\
&\lesi \Big(\f{r_B}{|x-x_B|}\Big)^{\gamma_\nu}\f{1}{|x-x_B|}|B|^{1-1/p},
\end{aligned}
\]
where in the last inequality we used
\[
\f{r_B}{|x-x_B|}\le \Big(\f{r_B}{|x-x_B|}\Big)^{\gamma_\nu},
\]
since both $\f{r_B}{|x-x_B|}$ and $\gamma_\nu$ are less than or equal to $1$.

For $F_2$, since $r_B<\rho_\nu(x_B)$ and $\gamma_\nu\le 1$, we have $\f{r_B}{\rho_\nu(x_B)}\le \Big(\f{r_B}{\rho_\nu(x_B)}\Big)^{\gamma_\nu}$. Hence,
\[
\begin{aligned}
F_2&\lesi \Big(\f{r_B}{\rho_\nu(x_B)}\Big)^{\gamma_\nu}\f{1}{|x-x_B|^n} \Big(\f{|x-x_B|}{\rho_\nu(x_B)}\Big)^{-\gamma_\nu} \|a\|_1\\
&\lesi \Big(\f{r_B}{|x-x_B|}\Big)^{\gamma_\nu}\f{1}{|x-x_B|^n}|B|^{1-1/p}.
\end{aligned}
\]
Taking this and the estimate of $F_1$ into account, we obtain
\[
\sup_{t>0} |e^{-t\mathcal L_{\nu}} \mathcal L_\nu^{-|k|/2}(\delta_\nu^*)^ka(x)|\lesi \Big(\f{r_B}{|x-x_B|}\Big)^{\gamma_\nu}\f{1}{|x-x_B|^n}|B|^{1-1/p}	\]
for all $x\in (4B)^c$.
Consequently,
\[
\Big\|\sup_{t>0}|e^{-t\mathcal L_{\nu}} \mathcal L_\nu^{-|k|/2}(\delta_\nu^*)^ka|\Big\|_{L^p(\mathbb R^n_+\backslash 4B)}\lesi 1,
\]
as long as $\f{n}{n+\gamma_\nu}<p\le 1$.

\end{proof}

{\bf Acknowledgement.} The authors were supported by the research grant ARC DP220100285 from the Australian Research Council.

\end{document}